\documentclass{article} 

\usepackage[frenchb, english]{babel}
\usepackage{tikz}
\usepackage{tkz-graph}
\usepackage{amssymb}
\usepackage{graphicx}
\usetikzlibrary{arrows, decorations.markings}
\usepackage{rotating}
\usepackage{caption}
\usepackage{verbatim}
\usepackage{amsmath}
\usepackage{amsthm}
\usepackage{enumerate}
\usepackage{array}
\newcolumntype{L}[1]{>{\raggedright\let\newline\\\arraybackslash\hspace{0pt}}m{#1}}

\usetikzlibrary{decorations.markings, arrows.meta}
\usepackage{mathtools}
\usepackage{color}   
\usepackage{hyperref}
\hypersetup{
    colorlinks=false, %set true if you want colored links
    linktoc=all,     %set to all if you want both sections and subsections linked
}

\usepackage{float}
\usepackage{mathrsfs}
\usepackage{wrapfig}
\usepackage{picture}
\usepackage{multirow}
\usepackage{MnSymbol}
\usepackage{relsize}
\usepackage{diagbox}
\usetikzlibrary{arrows}
\usetikzlibrary{shapes,snakes}
\usepackage[applemac]{inputenc}
\usepackage[T1]{fontenc}
\usepackage{geometry}
\usepackage{stackengine}
\usepackage{scalerel}
\usepackage{bbm}
\usepackage{float}
\usepackage{chngpage}
\usepackage{longtable}
\usepackage{hhline}

\newlength{\bibitemsep}
\newlength{\bibparskip}
\let\oldthebibliography\thebibliography
\renewcommand\thebibliography[1]{%
  \oldthebibliography{#1}%
  \setlength{\parskip}{\bibitemsep}%
  \setlength{\itemsep}{\bibparskip}%
}

\newcommand{\cc}{\ensuremath{\mathbb{C}}\xspace}

\usepackage{tocloft}
\usetikzlibrary{arrows.meta}
\usetikzlibrary{decorations.pathreplacing,angles,quotes}

\newtheorem{theorem}{Theorem}[section]
\newtheorem{conjecture}[theorem]{Conjecture}

\newtheorem{definition}[theorem]{Definition}

\newtheorem{examples}[theorem]{Examples}
\newtheorem{remark}[theorem]{Remark}

\addto\captionsenglish{% Replace "english" with the language you use
  \renewcommand{\contentsname}%
    {Table of Contents}%
}

\begin{document}

\begin{center}
\begin{LARGE}
Root Systems and Quotients of Deformations of Simple Singularities
\end{LARGE}
\vspace{0.3cm}

\begin{Large}
Antoine CARADOT
\end{Large}
\end{center}
\vspace{1cm}

\begin{center}
\textbf{Abstract}
\end{center}

In this article we study quotients of deformations of simple singularities, and attempt to characterize them in terms of subsystems of simple root systems. The quotient of a semiuniversal deformation of a simple singularity of inhomogeneous type $B_r$ ($r \geq 2$), $C_r$ ($r \geq 3$), $F_4$ or $G_2$ by the natural symmetry of the associated Dynkin diagram is a deformation of a simple singularity of homogeneous type $X=D_s$, $E_6$ or $E_7$, but not semiuniversal anymore. Therefore not all subdiagrams of $X$ appear as singular configurations of the fibers of the deformation. We propose a conjecture for the types of singular configurations in terms of sub-root systems of a root system of type $X$. The conjecture is then proved for the types $B_2,B_3,C_3,F_4$ and $G_2$.

\tableofcontents

\section*{Introduction}
\addcontentsline{toc}{section}{Introduction}

Simple singularities have first been studied by F. Klein in \cite{Klein84}, where he classified them as quotients of the complex plane by the action of a finite subgroup $\Gamma$ of $\mathrm{SU}_2$. It was then showed by P. Du Val (\cite{DuVa34}) that the exceptional divisors of the minimal resolution of the isolated singularity of such a quotient form an arrangement of projective lines whose dual graph is a simply-laced Dynkin diagram $\Delta(\Gamma)$, and so the quotient $\cc^2/\Gamma$ is called a simple singularity of type $\Delta(\Gamma)$. P. Slodowy extended in 1978 (and subsequently in \cite{Slo80} Section 6.2)  the definition of a simple singularity to the non simply-laced types by adding a second finite subgroup $\Gamma'$ of $\mathrm{SU}_2$ containing $\Gamma$ as normal subgroup. Then $\Gamma'/\Gamma =\Omega$ acts on $\cc^2/\Gamma$ and this action can be lifted to the minimal resolution of the singularity and induces an action on the exceptional divisors which corresponds to a group of automorphisms of the Dynkin diagram of $\cc^2/\Gamma$.
\newline

Let $\alpha  :  X_\Gamma \rightarrow \mathfrak{h}/W$ be the semiuniversal deformation of a simple singularity of type $\Delta(\Gamma)=A_{2r-1}$ ($r \geq 2$), $D_{r}$ ($r \geq 4$) or $E_6$ obtained by the construction of H. Cassens and P. Slodowy in \cite{CaSlo98}, $\frak{h}$ and $W$ being the Cartan subalgebra and the associated Weyl group of the simple Lie algebra $\mathfrak{g}$ of the same type. It was shown in \cite{Cara17} that the automorphism group $\Omega$ of the Dynkin diagram of $\mathfrak{g}$ acts on $X_\Gamma$ and $\mathfrak{h}/W$ such that $\alpha$ is $\Omega$-equivariant. A result of P. Slodowy then implies that taking the restriction $\alpha^\Omega$ of $\alpha$ over the $\Omega$-fixed points of $\mathfrak{h}/W$ leads to a semiuniversal deformation of a simple singularity which is inhomogeneous and of the same type as the folding of the root system of $\mathfrak{g}$ by $\Omega$. As $\alpha$ is $\Omega$-equivariant, there is an action of $\Omega$ on every fiber of $\alpha^\Omega$, and the quotient leads to a new morphism $\overline{\alpha^\Omega}$, which was shown to be a non-semiuniversal deformation of a simple singularity of homogeneous type $\Delta(\Gamma')$ (cf. \cite{Cara17}). 
\newline 

The aim of this article is to propose a characterization of the types of singularities which appear in the fibers of $\overline{\alpha^\Omega}$ in terms of sub-root systems of a root system of type $\Delta(\Gamma')$. The special fiber $(\overline{\alpha^\Omega})^{-1}(\overline{0})$ is a simple singularity of type $\Delta(\Gamma')$. But as $\overline{\alpha^\Omega}$ is not semiuniversal, all subdiagrams of the Dynkin diagram of type $\Delta(\Gamma')$ cannot be found as singular configurations of the fibers of $\overline{\alpha^\Omega}$. We conjecture that there exists a subset $\Theta$ of simple roots of the root system of type $\Delta(\Gamma')$ such that the Dynkin diagram associated to the singular configuration of any fiber of $\overline{\alpha^\Omega}$ is a subdiagram of the Dynkin diagram of type $\Delta(\Gamma')$ containing the vertices associated to $\Theta$. Furthermore, all such subdiagrams are realized as singular configurations of some fibers of $\overline{\alpha^\Omega}$. 
\newline

In the first section we will present the construction of semiuniversal deformations of the simple homogeneous (\cite{CaSlo98}) and inhomogeneous (\cite{Cara17}) singularities. In the second section we study the regularity of the fibers of $\overline{\alpha^\Omega}$ and prove that they are always all singular. The third and fourth sections are devoted to the main conjecture as well as its proof for the types $B_2$, $B_3$, $C_3$, $F_4$ and $G_2$.
\newline

Throughout this article the base field is the complex number field $\cc$. Furthermore, a Dynkin diagram and its type will often be designated by the same symbol $\Delta$.

\section{Deformations of simple singularities}
\subsection{Simple singularities}\label{sec : BrieskornSlodowy}

\subsubsection{Simple singularities of type $A_r$, $D_r$, $E_6$, $E_7$ and $E_8$}\label{subsub : ADE}

Let $\Gamma$ be a finite subgroup of $\mathrm{SU}_2$. F. Klein showed in \cite{Klein84} that $\Gamma$ is isomorphic to either the cyclic group $\mathcal{C}_n$ of order $n$, the binary dihedral group $ \mathcal{D}_n$ of order $4n$, the binary tetrahedral group $\mathcal{T}$ of order $24$, the binary octahedral group $ \mathcal{O}$ of order $48$, or the binary icosahedral group $\mathcal{I}$ of order $120$, and that the quotient $\mathbb{C}^2/\Gamma$ is a surface which injects into $\mathbb{C}^3$ as the zero set of a polynomial $R_\Gamma \in \mathbb{C}[X,Y,Z]$. Furthermore this surface presents a unique isolated singularity and is called a \textit{simple singularity}. P. Du Val then proved (\cite{DuVa34}) that if $s \in \mathbb{C}^2/\Gamma$ is the singular point and $\pi :  \widetilde{S} \rightarrow \mathbb{C}^2/\Gamma$ is the minimal resolution of $\mathbb{C}^2/\Gamma$, then the preimage of $s$ is a union of projective lines whose intersection matrix is the additive inverse of a Cartan matrix of type $\Delta(\Gamma)=A_r$, $D_r$ or $E_r$.
\newline 

\renewcommand{\arraystretch}{1.4}
\begin{center}
\centering
\begin{longtable}{|c|c|c|}
\hline
 $\Gamma$ & $R_\Gamma$ &Type of $\Delta(\Gamma)$ \\
\hhline{|=|=|=|}
  $\mathcal{C}_n$ & $X^n + YZ$ &  $A_{n - 1}$ \\
\hline
  $\mathcal{D}_n$ &  $X(Y^2 - X^n) + Z^2$ & $D_{n + 2}$ \\
\hline
  $\mathcal{T}$ & $X^4 + Y^3 + Z^2$ & $E_6$ \\
\hline
  $\mathcal{O}$  & $X^3 + XY^3 + Z^2$ & $E_7$ \\
\hline
$\mathcal{I}$ & $X^5 + Y^3 + Z^2$& $E_8$ \\
 \hline
 \end{longtable}
 {\addtocounter{table}{-1}}
  \captionof{table}{}
\end{center}

\subsubsection{Simple singularities of type $B_r$, $C_r$, $F_4$ and $G_2$}\label{BCFGdefinition}

The definition of the simple singularities of inhomogeneous types is due to P. Slodowy (cf. \cite{Slo80} Section 6.2). 

\begin{definition}\label{def : inhomogeneous}
A \textit{simple singularity} of type $B_r \ (r \geq 2)$, $C_r \ (r \geq 3)$, $F_4$ or $G_2$ is a pair $(X_0,\Omega)$ with $X_0$ a simple singularity (in the sense of Section~\ref{subsub : ADE}) and $\Omega$ a group of automorphisms of $X_0$ according to the following list : 
\newline

\begin{center}
\renewcommand{\arraystretch}{1.4}
  \begin{tabular}{|c|c|c|c|c|}
\hline
Type of $(X_0,\Omega)$ & Type of $X_0$ & $\Gamma$ & $\Gamma'$ & $\Omega$ \\
\hhline{|=|=|=|=|=|}
$B_r \ (r \geq 2)$ & $A_{2r-1}$ & $\mathcal{C}_{2r}$ & $\mathcal{D}_r$ & $\mathbb{Z}/2\mathbb{Z}$\\
\hline
$C_r \ (r \geq 3)$ & $D_{r + 1}$ & $\mathcal{D}_{r - 1}$ & $\mathcal{D}_{2(r - 1)}$& $\mathbb{Z}/2\mathbb{Z}$ \\
\hline
$F_4$ & $E_6$ &  $\mathcal{T}$ & $\mathcal{O}$ & $\mathbb{Z}/2\mathbb{Z}$ \\
\hline
$G_2$  & $D_4$ &  $\mathcal{D}_{2}$ & $\mathcal{O}$ & $\mathfrak{S}_3$\\
\hline
\end{tabular} 
\captionof{table}{}
\end{center}
\end{definition}
\noindent In each row of this table, $\Gamma$ and $\Gamma'$ are finite subgroups of $\mathrm{SU}_2$ such that $X_0$ is a simple singularity of type $\Delta(\Gamma)$ (i.e. $X_0 \cong \cc^2/\Gamma$) and $\Gamma \lhd \Gamma'$. Furthermore there is a natural action of $\Omega=\Gamma'/\Gamma$ on the singularity $X_0$. This action lifts in a unique way to the minimal resolution $\widetilde{X}_0$ of $X_0$. As $\Omega$ fixes the singular point of $X_0$, it will stabilize the exceptional divisor in $\widetilde{X}_0$. Hence we obtain an action of $\Omega$ on the dual diagram $\Delta(\Gamma)$ of the exceptional divisor. It turns out that this action agrees with the automorphism group of the Dynkin diagram.
\newline

The inhomogeneous type associated to the simple singularity $(X_0,\Omega)$ will be referred as $\Delta(\Gamma,\Gamma')$.
\newline

A simple singularity of inhomogeneous type is then a simple homogeneous singularity with a symmetry of the Dynkin diagram. One can notice that the type of $(X_0,\Omega)$ is the same as the type of the folding by $\Omega$ of a root lattice of the same type as $X_0$ (cf. \cite{Cara17} Section 1.2 for details). 
\newline
\begin{remark}\label{A2rabsent}
 The type $(A_{2r},\mathbb{Z}/2\mathbb{Z})$ is the only case that appears in Table 1 but not in Table 2. This is because the action of the symmetry group of the Dynkin diagram of type $A_{2r}$ fails to lift to the exceptional locus of the minimal resolution of the associated simple singularity (cf. \cite{Cara17} Section 1.4.2.2).
\end{remark}
\medskip

The notion of symmetry has been added to the simple singularities, therefore it is necessary to include this symmetry in the definition of the deformations of singularities of type $B_r$, $C_r$, $F_4$ and $G_2$ (\cite{Slo80} Section 2.6) :
\medskip

\begin{definition}
A deformation of a simple singularity $(X_0,\Omega)$ is an $\Omega$-equivariant deformation of the singularity $X_0$ with a trivial action of $\Omega$ on the base space. \\
Set $\pi : X \rightarrow Y$ a deformation of $X_0$. A deformation $\psi : X' \rightarrow Y'$ of $X_0$ is said to be induced from $\pi$ by a morphism $\varphi : Y' \rightarrow Y$ if there exist a morphism $\Phi :X' \rightarrow Y'$ such that the following diagram is Cartesian :
\begin{center}
 \begin{tikzpicture}[>=angle 60]
\node  (1) at ( 0,0)  {$X'$};
\node  (2) at ( 2,0) {$X$};
\node (3) at ( 0,-2) {$Y'$};
\node (4) at ( 2,-2) {$Y$};
\node (5) at (1,0.25) {$\Phi$};
\node (6) at (-0.25,-1) {$\psi$};
\node (7) at (2.25,-1) {$\pi$};
\node (8) at (1,-2.25) {$\varphi$};

\draw [decoration={markings,mark=at position 1 with
    {\arrow[scale=1.2,>=stealth]{>}}},postaction={decorate}] (1)  --  (2);
\draw [decoration={markings,mark=at position 1 with
    {\arrow[scale=1.2,>=stealth]{>}}},postaction={decorate}] (1)  --  (3);
\draw [decoration={markings,mark=at position 1 with
    {\arrow[scale=1.2,>=stealth]{>}}},postaction={decorate}] (3)  --  (4);
    \draw [decoration={markings,mark=at position 1 with
    {\arrow[scale=1.2,>=stealth]{>}}},postaction={decorate}] (2)  --  (4);
\end{tikzpicture}
\end{center}
and
\begin{center}
 \begin{tikzpicture}[>=angle 60]
\node  (1) at ( 0,0)  {$X'$};
\node  (2) at ( 2,0) {$X$};
\node (3) at ( 1,1) {$X_0$};
\node (4) at ( 1,-0.25) {$\Phi$};

\draw [decoration={markings,mark=at position 1 with
    {\arrow[scale=1.2,>=stealth]{>}}},postaction={decorate}] (1)  --  (2);
\draw [right hook->,decoration={markings,mark=at position 1 with
    {\arrow[scale=1.2,>=stealth]{>}}},postaction={decorate}] (3)  --  (2);
\draw [left hook->,decoration={markings,mark=at position 1 with
    {\arrow[scale=1.2,>=stealth]{>}}},postaction={decorate}] (3)  --  (1);
\end{tikzpicture}
\end{center}
commutes. The condition of the first diagram being Cartesian means that $X'$ is isomorphic to the fiber product $X \times_Y Y'$. \\
A semiuniversal deformation $\pi_0 : X \rightarrow Y$  of a simple singularity $(X_0,\Omega)$ is a deformation of $(X,\Omega)$ such that any other deformation $\psi : X' \rightarrow Y'$ of $(X,\Omega)$ is induced from $\pi_0$ by an $\Omega$-equivariant morphism $\varphi : Y' \rightarrow Y$ with a uniquely determined differential $d_{y'} \varphi :T_{y'} Y' \rightarrow T_y Y$.
\end{definition}

Semiuniversal deformations of simple homogeneous and inhomogeneous singularities are unique up to isomorphism, which makes them ideal objects for the studies of their associated singularities.

\subsection{Deformations of homogeneous simple singularities}\label{SlodowyCassens}

In this section we present a construction due to H. Cassens and P. Slodowy of the pullback of a semiuniversal deformation of a simple singularity of homogeneous type. For further details on proofs and references of the statements, we invite the reader to look at \cite{CaSlo98}.
\newline

Let $\Gamma$ be a finite subgroup of $\mathrm{SU}_2$, $R$ its regular representation, $N$ its natural representation as a subgroup of $\mathrm{SU}_2$, $\Delta(\Gamma)$ the associated Dynkin diagram (cf. Section~\ref{subsub : ADE} ), and define $M(\Gamma)= (\mathrm{End}(R) \otimes N)^\Gamma$. Then $M(\Gamma)$ is the representation space of a \textit{McKay quiver} $Q$, i.e. a quiver whose vertices are the vertices of the extended Dynkin diagram $\widetilde{\Delta}(\Gamma)$, with two arrows (one in each direction) for any edge in $\widetilde{\Delta}(\Gamma)$.
\newline

The group $G(\Gamma)=(\prod_{i=0}^{r}\mathrm{GL}_{d_i}(\mathbb{C}))/\mathbb{C}^*$, with $(d_0, \ldots, d_r)$ the dimension vector of $M(\Gamma)$, acts on $M(\Gamma)$ by simultaneous conjugation. 
\newline

Using a formula due to G. Lusztig, one can define a non-degenerate $G(\Gamma)$-invariant symplectic form $\langle.,.\rangle$ on $M(\Gamma)$ that induces a moment map

\begin{center}
$\mu_{CS} :  M(\Gamma) \rightarrow (\mathrm{Lie \ }G(\Gamma))^*  \ \mathlarger{\mathlarger{\subset}} \ \displaystyle \bigoplus_{i=0}^{r}M_{d_i}(\mathbb{C})$.
\end{center}

\noindent We identify $\mathrm{Lie \ }G(\Gamma)$ and its dual $(\mathrm{Lie \ }G(\Gamma))^*$.
\newline

Let $Z$ be the dual of the center of $\mathrm{Lie \ }G(\Gamma)$. Because the moment map is $G(\Gamma)$-equivariant, $G(\Gamma)$ acts on any fiber $\mu_{CS}^{-1}(z)$, $z \in Z$. Hence for any $z \in Z$, the quotient $\mu_{CS}^{-1}(Z)//G(\Gamma)$ is well defined. It is then proved in \cite{CaSlo98} that the morphism

\begin{center}
$\mu_{CS}^{-1}(Z)//G(\Gamma) \longrightarrow Z$
\end{center}

\noindent is the pullback of a semiuniversal deformation of the simple singularity $\mathbb{C}^2/\Gamma$ by the natural morphism $\mathfrak{h} \rightarrow \mathfrak{h}/W$, with $\mathfrak{h}$ a Cartan subalgebra of a Lie algebra of type $\Delta(\Gamma)$ and $W$ the associated Weyl group.

\subsection{Deformations of inhomogeneous simple singularities}\label{Moi}

This section is devoted to the extension conducted in \cite{Cara17} of the construction of Section~\ref{SlodowyCassens} to the inhomogeneous simple singularities of type $B_r$ $(r \geq 2)$, $C_r$ $(r \geq 3)$, $F_4$ and $G_2$.
\newline

Let $\Delta(\Gamma)$ be a Dynkin diagram of type $A_{2r-1}$ ($r\geq 2$), $D_r$ ($r\geq 4$) or $E_6$, with $\Gamma$ being the associated finite subgroup of $\mathrm{SU}_2$ (cf. Section~\ref{subsub : ADE}). Using the previous section, one obtains the following diagram : 

\begin{center}
\begin{tikzpicture}[scale=1,  transform shape,>=angle 60]
\tikzstyle{point}=[circle,draw,fill]
\tikzstyle{ligne}=[thick]

\node (1) at ( 0,0) {$X_\Gamma \times_{\mathfrak{h}/W}\mathfrak{h}$};
\node (2) at ( 2.5,0) {$X_\Gamma$};
\node (3) at ( 2.5, - 2.5)  {$\mathfrak{h}/W$};
\node (4) at ( 0, - 2.5)  {$\mathfrak{h}$};
\node (10) at (- 2.5,0)  {$\mu_{CS}^{-1} (Z)//G(\Gamma)$};
\node (14) at (- 1.05,0)  {$=$};

\node (13) at (  - 0.5, - 2.5)  {$Z \cong$};

\node (5) at ( 1.5,0.2) {$\psi$};
\node (6) at (  - 0.2, - 1.25) {$\widetilde{\alpha}$};
\node (7) at ( 1.15, - 2.7)  {$\pi$};
\node (8) at ( 2.7, - 1.25)  {$\alpha$};
\node (9) at ( 1.25, - 1.25) {$\circlearrowleft$};

\draw  [decoration={markings,mark=at position 1 with
    {\arrow[scale=1.2,>=stealth]{>}}},postaction={decorate}] (1)  --  (2);
\draw  [decoration={markings,mark=at position 1 with
    {\arrow[scale=1.2,>=stealth]{>}}},postaction={decorate}] (2)  --  (3);
\draw  [decoration={markings,mark=at position 1 with
    {\arrow[scale=1.2,>=stealth]{>}}},postaction={decorate}] (1)  --  (4);
\draw  [decoration={markings,mark=at position 1 with
    {\arrow[scale=1.2,>=stealth]{>}}},postaction={decorate}] (4)  --  (3);

\end{tikzpicture}
\end{center}

\noindent with $\alpha$ a semiuniversal deformation of the simple singularity $\cc^2/\Gamma$ of type $\Delta(\Gamma)$, $\mathfrak{h}$ a Cartan subalgebra of type $\Delta(\Gamma)$ and $W$ the associated Weyl group. 
\newline

Let $\Gamma'$ be the finite subgroup of $\mathrm{SU}_2$ such that there exists a simple singularity of inhomogeneous type $\Delta(\Gamma,\Gamma')$ (cf. Definition~\ref{def : inhomogeneous}). Then $\Omega=\Gamma'/\Gamma$ acts on the singularity $X_{\Gamma,0}=\alpha^{-1}(0)$. Let us define $X_{\Gamma, \Omega}=\alpha^{-1}((\mathfrak{h}/W)^{\Omega})$ and $\alpha^\Omega=\alpha_{|X_{\Gamma, \Omega}}$. Our aim is to define natural actions of $\Omega$ on $X_\Gamma$ and $\mathfrak{h}/W$ such that $\alpha$ becomes $\Omega$-equivariant, because if so, using results of P. Slodowy, one can show that the restriction $\alpha^\Omega : X_{\Gamma, \Omega} \rightarrow (\mathfrak{h}/W)^{\Omega}$ is a semiuniversal deformation of an inhomogeneous simple singularity of type $\Delta(\Gamma,\Gamma')$ (cf. \cite{Cara17} Section 4.3.1).
\newline

In \cite{Cara17} it is shown that if the action of $\Omega$ on $M(\Gamma)$ is symplectic, then $\widetilde{\alpha}$ can be made into an $\Omega$-equivariant map, and so is $\alpha$. We then have the following result : 
 
\begin{theorem}\label{thm : compatibility}
Let $M(\Gamma)$ be the representation space of a McKay quiver built on a Dynkin diagram of type $A_{2r-1}$, $D_{r}$ or $E_6$ as explained in Section~\ref{SlodowyCassens}. Then there exists an action of $\Omega=\Gamma'/\Gamma$ on $M(\Gamma)$ that is both symplectic and induces the natural action on the singularity $\cc^2/\Gamma$. This action then turns $\alpha$ into an $\Omega$-equivariant morphism.
\end{theorem}

The previous theorem enables us to explicitly compute semiuniversal deformations for the inhomogeneous types. For details on these computations, see \cite{Cara17} Sections 4.3 and 4.4.

\section{Quotients of semiuniversal deformations of inhomogeneous types}\label{sec : quotientsingu}

In the previous section we have seen that the morphism $\alpha^{\Omega} : X_{\Gamma, \Omega} \rightarrow (\mathfrak{h}/W)^{\Omega}$ is $\Omega$-invariant. Thus $\Omega$ acts on each fiber of $\alpha^{\Omega}$ and the fibers can be quotiented. Furthermore it is known that $(\alpha^{\Omega})^{-1}(\overline{0}) = X_{\Gamma,0}=\cc^2/\Gamma$. Hence the quotient fiber is $(\alpha^{\Omega})^{-1}(\overline{0})/\Omega=X_{\Gamma,0}/\Omega = (\cc^2/\Gamma)/(\Gamma'/\Gamma) \cong \cc^2/\Gamma'$. It is a simple singularity because $\Gamma'$ is a finite subgroup of $\mathrm{SU}_2$. Therefore the family given by the quotient map $\overline{\alpha^{\Omega}} : X_{\Gamma, \Omega} /\Omega \rightarrow (\mathfrak{h}/W)^{\Omega}$ is a deformation of the simple singularity $\cc^2/\Gamma'$ of type $\Delta(\Gamma')$. We will say that the deformation $\overline{\alpha^{\Omega}}$ is obtained through $\Delta(\Gamma)-\Delta(\Gamma,\Gamma')-\Delta(\Gamma')$-procedure.
\newline

The results obtained in \cite{Cara17} Sections 4.3 and 4.4 lead to the following table : 
\newline

\begin{center}
\begin{tabular}{|c|c|c|}
\hline
Type of $\alpha^\Omega$ & Type of $\overline{\alpha^\Omega}$& Rank of $\overline{\alpha^\Omega}$ \\
\hhline{|=|=|=|}
$B_r$ $(r \geq 2)$ & $D_{r+2}$& r \\
\hline
$C_r$ $(r \geq 3)$ & $D_{2r}$& r \\
\hline
$F_4$ & $E_7$& 4 \\
\hline
$G_2$ & $E_7$& 2 \\
\hline
\end{tabular}
\captionof{table}{}
\end{center}
\smallskip
Here we slightly abuse notations by saying that a deformation is of type $\Delta$ if it is a deformation of a simple singularity of type $\Delta$. Furthermore the \textit{rank} of a deformation is the dimension of its base space.
\newline

\begin{remark}
In the case $\Delta(\Gamma, \Gamma')=G_2$, one can replace the symmetry group $\mathfrak{S}_3$ by $\mathbb{Z}/3\mathbb{Z}$. The rank of $\overline{\alpha^\Omega}$ remains $2$, but the type of $\overline{\alpha^\Omega}$ becomes $E_6$.
\end{remark}
\smallskip

\begin{examples}
$\bullet$ For $\Delta(\Gamma, \Gamma')=B_2$, the map $\alpha^\Omega$ is given by the projection
\begin{center}
\begin{tikzpicture}[scale=1]
\node (1) at ( 0,0) {$ X_{\Gamma, \Omega}=\{((x,y,z),(t_2,0,t_4)) \in \cc^3 \times \mathfrak{h}/W \ | \ z^4+t_2z^2+t_4+\frac{t_2^2}{8}=xy\}$};
\node (2) at ( 0,-2) {$(\mathfrak{h}/W)^\Omega=\{ (t_2,0,t_4) \in  \mathfrak{h}/W\}$};
\node (3) at ( 0.5, -1)  {$\alpha^\Omega$};

\draw  [decoration={markings,mark=at position 1 with
    {\arrow[scale=1.2,>=stealth]{>}}},postaction={decorate}] (1)  --  (2);
\end{tikzpicture}
\end{center}
with the action of $\Omega=\mathbb{Z}/2\mathbb{Z}=<\sigma>$ on a fiber $(\alpha^\Omega)^{-1}(t_2,0,t_4)$ being $\sigma.(x,y,z)=(y, x,-z)$. Here $t_i$ denotes the flat coordinate of degree $i$ (cf. \cite{SYS80}). Therefore the quotient $\overline{\alpha^\Omega}$ is given by
\begin{center}
\begin{tikzpicture}[scale=1]
\node (1) at ( 0,0) {$ X_{\Gamma, \Omega}/\Omega=\{((x,y,z),(t_2,0,t_4)) \in \cc^3 \times \mathfrak{h}/W \ | \ Z(X^2-4Z^2)+W^2-4t_2Z^2-4(t_4+\frac{t_2^2}{8})Z=0\}$};
\node (2) at ( 0,-2) {$(\mathfrak{h}/W)^\Omega=\{ (t_2,0,t_4) \in  \mathfrak{h}/W\}$};
\node (3) at ( 0.5, -1)  {$\overline{\alpha^\Omega}$};

\draw  [decoration={markings,mark=at position 1 with
    {\arrow[scale=1.2,>=stealth]{>}}},postaction={decorate}] (1)  --  (2);
\end{tikzpicture}
\end{center}
The map $\overline{\alpha^\Omega}$ is a deformation of a $D_4$-singularity and is of rank 2. \\

\hspace{2.05cm} $\bullet$ For $\Delta(\Gamma, \Gamma')=C_3$, the map $\alpha^\Omega$ is given by the projection
\begin{center}
\begin{tikzpicture}[scale=1]
\node (1) at ( 0,0) {$X_{\Gamma, \Omega}=\{((x,y,z),(t_2,t_4,t_6,0)) \in \cc^3 \times  \mathfrak{h}/W \ | \ z^2 =xy(x+y)-\frac{t_2}{2}xy-\frac{t_4}{4}x+ \frac{1}{4}(t_6 + \frac{t_2t_4}{6} + \frac{t_2^3}{108})\}$};
\node (2) at ( 0,-2) {$(\mathfrak{h}/W)^{\Omega}=\{(t_2,t_4,t_6,0) \in   \mathfrak{h}/W \}$};
\node (3) at ( 0.5, -1)  {$\alpha^\Omega$};

\draw  [decoration={markings,mark=at position 1 with
    {\arrow[scale=1.2,>=stealth]{>}}},postaction={decorate}] (1)  --  (2);
\end{tikzpicture}
\end{center}
with the action of $\Omega=\mathbb{Z}/2\mathbb{Z}=<\sigma>$ on a fiber $(\alpha^\Omega)^{-1}(t_2,t_4,t_6,0)$ being $\sigma.(x,y,z)=(x,-x-y+\frac{t_2}{2},-z)$. The coordinates on $\mathfrak{h}/W$ are once again the flat coordinates. Therefore the quotient $\overline{\alpha^\Omega}$ is given by
\begin{center}
\begin{tikzpicture}[scale=1]
\node (1) at ( 0,0) {$X_{\Gamma, \Omega} /\Omega=\{((X,W,Z),(t_2,t_4,t_6,0)) \in \cc^3 \times \mathfrak{h}/W \ | \ -\frac{1}{64}X^5+XY^2-W^2$};
\node (2) at (0.5,-0.5)  {$+A_{X^4}X^4+A_{X^3}X^3+A_{X^2}X^2+A_XX+A_YY +A_0=0\}$};
\node (3) at ( 0,-2.5) {$(\mathfrak{h}/W)^{\Omega}=\{(t_2,t_4,t_6,0) \in   \mathfrak{h}/W \}$};
\node (4) at ( 0.5, -1.5)  {$\overline{\alpha^\Omega}$};
\node (5) at ( 0, -0.7)  {};

\draw  [decoration={markings,mark=at position 1 with
    {\arrow[scale=1.2,>=stealth]{>}}},postaction={decorate}] (5)  --  (3);
\end{tikzpicture}
\end{center}
where the $A_i$ are polynomials in $t_2,t_4,t_6$ without constant terms. So the map $\overline{\alpha^\Omega}$ is indeed a deformation of a $D_6$-singularity and is of rank 3. \\

For details on the computations, see \cite{Cara17} Section 4.4.
\end{examples}
\smallskip

E. Brieskorn showed (\cite{Bries71}) that a semiuniversal deformation of a simple singularity of type $\Delta_r$ ($\Delta=A$, $D$ or $E$) can be constructed such that the base space is the quotient $\mathfrak{h}/W$ with $\mathfrak{h}$ a Cartan subalgebra of a simple Lie algebra of type $\Delta_r$ and $W$ the associated Weyl group. Hence the rank of a semiuniversal deformation of a simple singularity of type $\Delta_r$ is $r$. However one notices in Table 3 that it is not true for $\overline{\alpha^\Omega}$. Therefore $\overline{\alpha^\Omega}$ is not a semiuniversal deformation in any case.
\newline

We will now prove the following theorem regarding the regularity of the fibers of $\overline{\alpha^\Omega}$ : 
\medskip

\begin{theorem}\label{ThmSingular} Let $\alpha^\Omega  : X_{\Gamma,\Omega} \rightarrow (\mathfrak{h}/W)^\Omega$ be the semiuniversal deformation obtained in Section~\ref{Moi} of a simple singularity of inhomogeneous type $B_r$ ($r \geq 2$), $C_r$ ($r \geq 3$), $F_4$ or $G_2$. Then every fiber of the quotient $\overline{\alpha^\Omega} :  X_{\Gamma,\Omega}/\Omega \rightarrow (\mathfrak{h}/W)^\Omega$ is singular.
\end{theorem}

\begin{proof} In the following computations, the $t_i$'s will always be the flat coordinates on $\mathfrak{h}/W$ (cf. \cite{SYS80}), with $i$ being the degree. For $F_4$ and $G_2$, we will look at the explicit equation obtained in \cite{Cara17} Section 4.4 and see that each fiber of $\overline{\alpha^\Omega}$ is indeed singular. For the general types $B_r$ and $C_r$, another approach is needed. We will show that each fiber of $\alpha^\Omega$ has at least one isolated fixed point $p$. If a fiber has such a point $p$, then there are two possibilities :  \\

- $p$ is singular, so a neighbourhood of $p$ is isomorphic to $\cc^2/\Xi$ with $\Xi$ a finite subgroup of $\mathrm{SU}_2$, because $p$ is a singular point in a semiuniversal deformation of a simple singularity. Thus we have $\mathrm{dim } \ T_p(\cc^2/\Xi)>\mathrm{dim } \ \cc^2/\Xi$. %(cf. \cite{Shafa13} Section 2.1.4).
 Furthermore $T_{\pi(p)}((\cc^2/\Xi)/\Omega) \cong T_p(\cc^2/\Xi)/T_p(\Omega.p)$ with $\pi$ the quotient map. As $p$ is a fixed point, we obtain $\mathrm{dim} \ T_{\pi(p)}((\cc^2/\Xi)/\Omega) = \mathrm{dim} \ T_p(\cc^2/\Xi) > \mathrm{dim} \ \cc^2/\Xi = \mathrm{dim} \ (\cc^2/\Xi)/\Omega$ because $\Omega$ is a finite group. Hence $\mathrm{dim }\ T_{\pi(p)}((\cc^2/\Xi)/\Omega)> \mathrm{dim }\ (\cc^2/\Xi)/\Omega$, and $\pi(p)$ is a singular point of $(\cc^2/\Xi)/\Omega$. \\

- $p$ is smooth, then the fiber has a tangent space at this point and it is isomorphic to $\cc^2$. So locally around $p$, the action of $\Omega=\mathbb{Z}/2\mathbb{Z}$ is the same as its action on $\cc^2$ around the origin. Hence the quotient will be  isomorphic to $\cc^2/\Omega$, which is an $A_1$-singularity. The quotient will then be singular.
\newline

$\bullet$ Type $B_r$ ($r \geq 2$). The fiber of $\alpha^\Omega$ above a point $(t_2,0,t_4,0,...,t_{2r}) \in (\mathfrak{h}/W)^\Omega$ is given by the points $(x,y,z) \in \cc^3$ such that $z^{2r}+\sum_{i=1}^r f_{2i}(t_2,...,t_{2r})z^{2(r-i)}=xy$ (cf. \cite{Cara17} Section 4.3.6). The $f_i$'s are polynomials such that $f_i(t_2(\xi),...,t_{2r}(\xi))=\epsilon_i(\xi)$ for all $1 \leq i \leq 2r$, $\epsilon_i$ being the elementary symmetric polynomial of degree $i$ and $\xi$ being a coordinate vector on the Cartan subalgebra $\mathfrak{h}$ of type $A_{2r-1}$. The action of $\Omega=\mathbb{Z}/2\mathbb{Z}=<\sigma>$ on the fiber is given by
\begin{center}
 $\sigma.(x,y,z)=((-1)^ry,(-1)^rx,-z)$.
 \end{center}
 Hence a fixed point is a point $p=(x,(-1)^rx, 0)$ such that $x^2=(-1)^r f_2(t_2,...,t_{2r})$. As $f_2(t_2,...,t_{2r})=t_2$, we define $x=(-1)^{\frac{r}{2}}\sqrt{t_2}$. Then $p=(x,(-1)^rx, 0)$ is a fixed point of the fiber. It is isolated because there are at most two possible values for $x$.
 \newline
 
 $\bullet$ Type $C_r$ ($r \geq 3$). Based on results from A. Kas and M. Schlessinger (\cite{KasSchle72}), as well as computations from \cite{Cara17}, the fiber of the restriction $\alpha^\Omega$ of the semiuniversal deformation $\alpha$ of a singularity of type $D_{r+1}$ above $(t_2, t_4,...,t_{2r},0) \in (\mathfrak{h}/W)^\Omega$ has the following equation :
  \begin{center}
 $x^r+xy^2+z^2+\alpha_{r-1}x^{r-1}+...+\alpha_{1}x+\beta y+\alpha_0=0$,
 \end{center}
 with $\beta$ and the $\alpha_i$'s being complex polynomials in variables $t_2, t_4,...,t_{2r}$. Here the system of flat coordinates on $\mathfrak{h}/W$ is given by $(t_2, t_4,...,t_{2r},\psi)$, $\Omega$ fixes each $t_i$ and sends $\psi$ on $-\psi$. The preceding equation is a quasi-homogenous polynomial with
 \begin{center}
$\left\lbrace \begin{array}[h]{l}
x \text{ of degree  }2, \\
y \text{ of degree  } r-1, \\
z \text{ of degree  } r, \\
t_i \text{ of degree  } i, \ i=2,4,...,2r.
\end{array}\right.$
\end{center}
From the study of the special fiber  $(\alpha^\Omega)^{-1}(\overline{0})$ and the fact that the action of $\Omega=\mathbb{Z}/2\mathbb{Z}=<\sigma>$ preserves the degree, one can see that 
\begin{center}
$\left\lbrace\begin{tabular}{l}
$\sigma.x=x+p(x,y,z,t_2,...,t_{2r}) $, of degree 2,\\
$\sigma.y=-y+q(x,y,z,t_2,...,t_{2r}) $, of degree r-1,\\
$\sigma.z=-z+s(x,y,z,t_2,...,t_{2r}) $, of degree r,\\
$\sigma.t_i=t_i $, $i=2,4,...,2r$,
\end{tabular}\right.$
\end{center}
where $p, q$ and $s$ are homogeneous polynomials with respect to the degrees of $x,y,z$ and the $t_i$'s, and such that $p(x,y,z,0,...,0)=q(x,y,z,0,...,0)=s(x,y,z,0,...,0)=0$ for any $(x,y,z) \in \cc^3$. \\

Let us fix $(t_2,...,t_{2r},0) \in (\mathfrak{h}/W)^\Omega$ and study the fiber $(\alpha^\Omega)^{-1}(t_2,...,t_{2r},0)$ given by the equation given above. \\

Because $x$ is of smallest possible degree and $\sigma^2=\text{Id}$, the only possibility is $\sigma.x=x$. For the other variables, the action might depend on the parity of $r$. \\

\hspace*{5mm} $\blacksquare$ Assume that $r$ is even. $y$ is of degree $r-1$ which is odd, and all the others terms are of even degree. The only possibility is thus $\sigma.y=-y$. \\
$\beta$ is a polynomial of degree $2r-\mathrm{deg} \ y=r+1$, which is odd. However $\beta$ is a polynomial in the variables $t_2,...,t_{2r}$, all of even degree. Hence $\beta=0$.\\
$z$ is of degree $r$ which is even, so $\sigma.z=-z+\sum_{i=0}^{\frac{r}{2}}\lambda_ix^i$ with the $\lambda_i$'s being polynomials in $t_2,...,t_{2r}$. One can check that $\sigma^2z=z$.\\
We are then looking for the fixed points of the action
\begin{center}
$\left\lbrace \begin{tabular}{l}
$\sigma.x=x$,\\
$\sigma.y=-y$,\\
$\sigma.z=-z+\sum_{i=0}^{\frac{r}{2}}\lambda_ix^i$.
\end{tabular}\right.$
\end{center}
If $(x,y,z) \in (\alpha^\Omega)^{-1}(t_2,...,t_{2r},0)$, then $\sigma.(x,y,z)=(x,-y,-z+\sum_{i=0}^{\frac{r}{2}}\lambda_ix^i) \in (\alpha^\Omega)^{-1}(t_2,...,t_{2r},0)$. Therefore the following equation has to be verified : 
 \begin{center}
 \begin{tabular}{cl}
 &  $x^r+xy^2+(-z+\sum_{i=0}^{\frac{r}{2}}\lambda_ix^i)^2+\alpha_{r-1}x^{r-1}+...+\alpha_{1}x+\alpha_0=0$, \\
$\Leftrightarrow$ & $x^r+xy^2+z^2+\alpha_{r-1}x^{r-1}+...+\alpha_{1}x+\alpha_0+(\sum_{i=0}^{\frac{r}{2}}\lambda_ix^i)(-2z+\sum_{i=0}^{\frac{r}{2}}\lambda_ix^i)=0$, \\
$\Leftrightarrow$ & $(\sum_{i=0}^{\frac{r}{2}}\lambda_ix^i)(-2z+\sum_{i=0}^{\frac{r}{2}}\lambda_ix^i)=0$
\end{tabular}
 \end{center}
because $(x,y,z) \in (\alpha^\Omega)^{-1}(t_2,...,t_{2r},0)$. \\
If $\lambda_i=0$ for all $i \in \{0,...,\frac{r}{2}\}$, then the action is   \raisebox{-.55\height}{$\left\lbrace \begin{tabular}{l}
$\sigma.x=x$,\\
$\sigma.y=-y$,\\
$\sigma.z=-z$.
\end{tabular}\right.$} \\
 If there exists $i \in \{0,...,\frac{r}{2}\}$ such that $\lambda_i \neq 0$, then the equation $\sum_{i=0}^{\frac{r}{2}}\lambda_ix^i=0$ is non-trivial in $x$ and has therefore at least one solution (and at most $\frac{r}{2}$). Given that in $(\alpha^\Omega)^{-1}(t_2,...,t_{2r},0)$ the variable $x$ is not constant, there exists $x$ such that $\sum_{i=0}^{\frac{r}{2}}\lambda_ix^i \neq 0$ and such that there exist $y, z \in \cc$ with $(x,y,z) \in (\alpha^\Omega)^{-1}(t_2,...,t_{2r},0)$. Hence we get that $z=\frac{1}{2}\sum_{i=0}^{\frac{r}{2}}\lambda_ix^i \neq 0$. So for such an $x$, the value of $z$ is uniquely determined and is not zero. However, based on the equation of $(\alpha^\Omega)^{-1}(t_2,...,t_{2r},0)$, we see that if $(x,y,z)$ is a point on the surface, then so is $(x,y,-z)$. If a non-zero value of $z$ exists, a second value exists also. This contradicts the unicity of the value of $z$. \\
 
We have shown that the action of $\Omega$ on $(\alpha^\Omega)^{-1}(t_2,...,t_{2r},0)$ is given by 
\begin{center}
$\left\lbrace \begin{tabular}{l}
$\sigma.x=x$,\\
$\sigma.y=-y$,\\
$\sigma.z=-z$.
\end{tabular}\right.$
\end{center}
A fixed point is thus given by $(x,0,0)$ with $x$ verifying the non-trivial equation $x^r+\alpha_{r-1}x^{r-1}+...+\alpha_{1}x+\alpha_0=0$. So a fixed point exists, there are at most $r$ of them, and they are all isolated. \\
\medskip

\hspace*{5mm} $\blacksquare$ Assume that $r$ is odd. $z$ is of degree $r$ which is odd, and all the others terms are of even degree. The only possibility is thus $\sigma.z=-z$. \\
$\beta$ is a polynomial of degree $2r-\mathrm{deg} \ y=r+1$, which is even. So it is possible that $\beta \neq 0$.\\
$y$ is of degree $r-1$ which is even, so $\sigma.y=-y+\sum_{i=0}^{\frac{r-1}{2}}\mu_ix^i$ with the $\mu_i$'s being polynomials in $t_2,...,t_{2r}$. One can check that $\sigma^2y=y$.\\
We are then looking for the fixed points of the action
\begin{center}
$\left\lbrace \begin{tabular}{l}
$\sigma.x=x$,\\
$\sigma.y=-y+\sum_{i=0}^{\frac{r-1}{2}}\mu_ix^i$,\\
$\sigma.z=-z$.
\end{tabular}\right.$
\end{center}
If $(x,y,z) \in (\alpha^\Omega)^{-1}(t_2,...,t_{2r},0)$, then $\sigma.(x,y,z)=(x,-y+\sum_{i=0}^{\frac{r-1}{2}}\mu_ix^i,-z) \in (\alpha^\Omega)^{-1}(t_2,...,t_{2r},0)$. Therefore the following equation has to be verified : 
 \begin{center}
 \begin{tabular}{cl}
 &  $x^r+x(-y+\sum_{i=0}^{\frac{r-1}{2}}\mu_ix^i)^2+z^2+\alpha_{r-1}x^{r-1}+...+\alpha_{1}x+\beta(-y+\sum_{i=0}^{\frac{r-1}{2}}\mu_ix^i)+\alpha_0=0$, \\
$\Leftrightarrow$ & $x^r+xy^2+z^2+\alpha_{r-1}x^{r-1}+...+\alpha_{1}x+\beta y+\alpha_0+(\sum_{i=0}^{\frac{r-1}{2}}\mu_ix^i)(\beta-2xy+x\sum_{i=0}^{\frac{r-1}{2}}\mu_ix^i)=2\beta y$, \\
$\Leftrightarrow$ & $(\sum_{i=0}^{\frac{r-1}{2}}\mu_ix^i)(\beta-2xy+x\sum_{i=0}^{\frac{r-1}{2}}\mu_ix^i)=2\beta y$
\end{tabular}
 \end{center}
because $(x,y,z) \in (\alpha^\Omega)^{-1}(t_2,...,t_{2r},0)$. \\

\noindent If $\mu_i=0$ for all $i \in \{0,...,\frac{r-1}{2}\}$, then $2\beta y=0$ for all $y$ such that $(x,y,z) \in  (\alpha^\Omega)^{-1}(t_2,...,t_{2r},0)$. But there exists at least one point in the surface for which $y \neq 0$. Hence $\beta=0$, and the action is given by \begin{center} \raisebox{-.55\height}{$\left\lbrace \begin{tabular}{l}
$\sigma.x=x$,\\
$\sigma.y=-y$,\\
$\sigma.z=-z$.
\end{tabular}\right.$} \end{center}

\noindent If there exists $i \in \{0,...,\frac{r-1}{2}\}$ such that $\mu_i \neq 0$, we have to consider two cases :  $\beta=0$ and $\beta \neq 0$. \\
$\blacktriangle$ Assume that $\beta=0$. Then $x(\sum_{i=0}^{\frac{r-1}{2}}\mu_ix^i)(-2y+\sum_{i=0}^{\frac{r-1}{2}}\mu_ix^i)=0$. The equation $x(\sum_{i=0}^{\frac{r-1}{2}}\mu_ix^i)=0$ is non-trivial in $x$ and has therefore at least one solution (and at most $\frac{r+1}{2}$). Given that in $(\alpha^\Omega)^{-1}(t_2,...,t_{2r},0)$ the variable $x$ is not constant, there exists $x$ such that $x(\sum_{i=0}^{\frac{r-1}{2}}\mu_ix^i) \neq 0$ and such that there exist $y,z \in \cc$ with $(x,y,z) \in (\alpha^\Omega)^{-1}(t_2,...,t_{2r},0)$. Hence we get that $y=\frac{1}{2}\sum_{i=0}^{\frac{r-1}{2}}\mu_ix^i \neq 0$. So for such an $x$, the value of $y$ is uniquely determined and is not zero. However, based on the equation of $(\alpha^\Omega)^{-1}(t_2,...,t_{2r},0)$, we see that if $(x,y,z)$ is a point on the surface, then so is $(x,-y,z)$ (because $\beta=0$). If a non-zero value of $y$ exists, a second value exists also. This contradicts the unicity of the value of $y$. \\
$\blacktriangle$ Assume that $\beta \neq 0$. We have 
 \begin{center}
 \begin{tabular}{cl}
 &  $(\sum_{i=0}^{\frac{r-1}{2}}\mu_ix^i)(\beta-2xy+x\sum_{i=0}^{\frac{r-1}{2}}\mu_ix^i)=2\beta y$, \\
$\Leftrightarrow$ & $(\beta +x\sum_{i=0}^{\frac{r-1}{2}}\mu_ix^i)(-2y+\sum_{i=0}^{\frac{r-1}{2}}\mu_ix^i)=0$.
\end{tabular}
 \end{center}
As $\beta \neq 0$, for any $x \in \cc$, there exist an infinity of $(y, z) \in \cc^2$ such that $(x,y,z) \in (\alpha^\Omega)^{-1}(t_2,...,t_{2r},0)$. Set $x \in \cc$ such that $x(\beta +x\sum_{i=0}^{\frac{r-1}{2}}\mu_ix^i) \neq 0$ (we get rid of a finite number of possible values). Set $y \in \cc$ such that $2xy+\beta \neq 0$ (we get rid of one possible value). Because of the equation of the surface, there exists $z \in \cc$ such that $(x,y,z) \in (\alpha^\Omega)^{-1}(t_2,...,t_{2r},0)$. One can then check that the point $(x,y-\frac{\beta+2xy}{x},z)$ is well defined and belongs to the surface. But because of the choice of $x$, we know that $y=\frac{1}{2}\sum_{i=0}^{\frac{r-1}{2}}\mu_ix^i$. So for such an $x$, the value of $y$ is uniquely determined. However, we have just seen that $(x,y,z)$ and $(x,y-\frac{\beta+2xy}{x},z)$ are two distinct points of the surface. This contradicts the unicity of the value of $y$. \\
 
We have shown that the action of $\Omega$ on $(\alpha^\Omega)^{-1}(t_2,...,t_{2r},0)$ is given by 
\begin{center}
$\left\lbrace \begin{tabular}{l}
$\sigma.x=x$,\\
$\sigma.y=-y$,\\
$\sigma.z=-z$.
\end{tabular}\right.$
\end{center}
A fixed point is thus given by $(x,0,0)$ with $x$ verifying the non-trivial equation $x^r+\alpha_{r-1}x^{r-1}+...+\alpha_{1}x+\alpha_0=0$. Therefore a fixed point exists, there are at most $r$ of them, and they are all isolated. \\

We have thus proved that, for any $r \geq 3$ and any $(t_2,...,t_{2r},0) \in (\mathfrak{h}/W)^\Omega$, the equation of $(\alpha^\Omega)^{-1}(t_2,...,t_{2r},0)$ is 
\begin{center}
$x^r+xy^2+z^2+\alpha_{r-1}x^{r-1}+...+\alpha_{1}x+\alpha_0=0$,
\end{center}
 the action of $\Omega$ is given by \begin{center} \raisebox{-.55\height}{$\left\lbrace \begin{tabular}{l}
$\sigma.x=x$,\\
$\sigma.y=-y$,\\
$\sigma.z=-z$,
\end{tabular}\right.$} \end{center}
and there is at least one isolated fixed point.
\newline

  $\bullet$ Type $F_4$.  The equation of the fiber of $\alpha^\Omega$ above a point $(t_2,0,t_6,t_8,0,t_{12}) \in (\mathfrak{h}/W)^\Omega$ is given by the points $(x,y,z) \in \cc^3$ such that 
 \begin{center}
 $-\frac{1}{4}x^4 +y^3 +z^2 -\frac{t_2}{4}x^2y+\frac{1}{48}(t_6-\frac{t_2^3}{8})x^2+\frac{1}{48}(-t_8+\frac{t_6t_2}{4}-\frac{t_2^4}{192})y +\frac{1}{576}(t_{12}-\frac{t_8t_2^2}{8}-\frac{t_6^2}{8}+\frac{t_6t_2^3}{96})=0$.
 \end{center}
The action of $\Omega=\mathbb{Z}/2\mathbb{Z}=<\sigma>$ on the fiber is then given by
\begin{center}
 $\sigma.(x,y,z)=(-x,y,-z)$.
 \end{center}
We can compute the equation of the quotient of the fiber and obtain
\begin{center}
$-\frac{1}{4}X^3+XY^3+Z^2-\frac{t_2}{4}X^2Y+\frac{1}{48}(t_6-\frac{t_2^3}{8})X^2+\frac{1}{48}(-t_8+\frac{t_6t_2}{4}-\frac{t_2^4}{192})XY+\frac{1}{576}(t_{12}-\frac{t_8t_2^2}{8}-\frac{t_6^2}{8}+\frac{t_6t_2^3}{96})X=0$.
\end{center}
One can check that for any  $(t_2,0,t_6,t_8,0,t_{12}) \in (\mathfrak{h}/W)^\Omega$, the surface just obtained is indeed singular.
\newline

  $\bullet$ Type $G_2=(D_4,\mathbb{Z}/3\mathbb{Z})$. The equation of the fiber of $\alpha^\Omega$ above a point $(t_2,0,t_6,0) \in (\mathfrak{h}/W)^\Omega$ is given by the points $(x,y,z) \in \cc^3$ such that
  \begin{center}
  $z^2=xy(x+y)-\frac{t_2}{2}xy+\frac{1}{4}(t_6+\frac{t_2^3}{108})$.
  \end{center}
The action of $\Omega=\mathbb{Z}/3\mathbb{Z}=<\rho>$ on the fiber is thus given by
\begin{center}
 $\rho.(x,y,z)=(y,-x-y+\frac{t_2}{2},z)$.
 \end{center}
We can compute the equation of the quotient of the fiber and obtain
\begin{center}
$11664X^4-Y^3-Z^2-324t_2X^2Y-(189t_2^3+5832t_6)X^2+(81t_2t_6+\frac{15}{16}t_2^4)Y+\frac{11}{32}t_2^6+\frac{189}{4}t_2^3t_6+729t_6^2=0$.
\end{center}
One can easily check that for any $(t_2,0,t_6,0) \in (\mathfrak{h}/W)^\Omega$, the surface defined by the previous equation has at least one singular point.
  \newline
  
   $\bullet$ Type $G_2=(D_4,\mathfrak{S}_3)$. The fixed point of $\Omega=\mathfrak{S}_3=<\rho,\sigma>$ in $\mathfrak{h}/W$ are the same as the ones by the action of $\mathbb{Z}/3\mathbb{Z}$ :  $\{(t_2,0,t_6,0) \in \mathfrak{h}/W\}$. The equation of the fiber of $\alpha^\Omega$ above a fixed point in $\mathfrak{h}/W$ is then the same as in the case of $(D_4,\mathbb{Z}/3\mathbb{Z}$). The action of $\Omega$ on such a fiber is given by
\begin{center}
 $\rho.(x,y,z)=(y,-x-y+\frac{t_2}{2},z)$ and $\sigma.(x,y,z)=(x,-x-y+\frac{t_2}{2},-z)$.
 \end{center}
The quotient of the fiber has then the equation
 \begin{center}
\resizebox{.99 \textwidth}{!} 
{$X^3Y -11664Y^3+Z^2+324t_2XY^2+(189t_2^3+5832t_6)Y^2-(\frac{15}{16}t_2^4+81t_2t_6)XY-(\frac{11}{32}t_2^6+\frac{189}{4}t_2^3t_6+729t_6^2)Y=0$.}
\end{center}
Like before, for any $(t_2,0,t_6,0) \in (\mathfrak{h}/W)^\Omega$, the surface has at least one singular point.

\end{proof}

\section{Singular configurations and root systems}\label{sec : singuconf}

Let $\mathfrak{g}$ be a simple Lie algebra of type $A_r$, $D_r$ or $E_r$, $\mathfrak{h}$ a Cartan subalgebra of $\mathfrak{g}$, $W$ the associated Weyl group, and $e \in \mathfrak{g}$ a nilpotent element. As $e$ is nilpotent, it can be included in an $\mathfrak{sl}_2$-triple $(e,f,h)$ by Jacobson-Morozov's theorem. Let us look at the restriction of the adjoint quotient $\chi :  \mathfrak{g} \rightarrow \mathfrak{h}/W$ to the Slodowy slice $S_e=e+\mathfrak{z}_\mathfrak{g}(f)$. Consider the diagram
\begin{center}
\begin{tikzpicture}[scale=1,  transform shape]
\node (3) at ( 0,2) {$S_{e}$};
\node (4) at ( 0.5,1) {$\chi_{\big| S_{e}}$};
\node (5) at ( 0,0) {$\mathfrak{h}/W$};
\node (6) at (0,-1) {$ \mathbb{D}$};
\node (7) at ( -1.5,0.3) {$\pi$};
\node (8) at ( -3,0) {$\mathfrak{h}$};
\node (9) at ( -3,-0.5) {$\bigcup$};
\node (10) at ( -3,-1) { $\displaystyle \bigcup_{\alpha \in \Phi^ + }H_\alpha$};

\draw  [decoration={markings,mark=at position 1 with
    {\arrow[scale=1.2,>=stealth]{>}}},postaction={decorate}] (3)  --  (5);
\draw  [decoration={markings,mark=at position 1 with
    {\arrow[scale=1.2,>=stealth]{>}}},postaction={decorate}] (8)  --  (5);
\draw  [|->,decoration={markings,mark=at position 1 with
    {\arrow[scale=1.2,>=stealth]{>}}},postaction={decorate}] (10)  --  (6);
\draw  [decoration={markings,mark=at position 1 with
    {\arrow[scale=1.2,>=stealth]{>}}},postaction={decorate}] (8)  --  (5);
\end{tikzpicture}
\end{center}
with $\pi$ the natural projection, $\Phi^+$ the set of positive roots of $\mathfrak{g}$, the $H_\alpha$'s the reflection hyperplanes with respect to the roots $\alpha \in \Phi^+$, and $\mathbb{D}$ the discriminant of $\chi$. E. Brieskorn proved in \cite{Bries71} that if $e$ is a subregular (i.e. $\mathrm{dim} \ Z_{G}(e)=\mathrm{rank} \ \mathfrak{g}+2$, with $G$ the simple Lie group associated to $\mathfrak{g}$) nilpotent element of $\mathfrak{g}$, then $ \chi_{| S_{e}}$ is a semiuniversal deformation of a simple singularity of the same type as $\mathfrak{g}$, and is surjective (\cite{Slo80} Section 7.4, Corollary 1). Furthermore, it is shown in \cite{Slo80} (Section 6.5, Lemmas 1, 2, 3) that the type of the singular configuration which appears in $S_e$ above a point $\pi(h) \in \mathbb{D}$ is given by the sub-root system $\{\alpha \in \Phi \ | \ h \in H_\alpha\}$. \\

Let $\mathfrak{g}'$ be a simple Lie algebra of type $\Delta(\Gamma')$ with root system $\Phi'$, $\mathfrak{h}'  \ \mathlarger{\mathlarger{\subset}} \  \mathfrak{g}'$ a Cartan subalgebra and $W'$ the associated Weyl group. Set $e' \in \mathfrak{g}'$ a subregular nilpotent element and $S_{e'}$ the associated Slodowy slice. Hence the restriction $\chi'_{| S_{e'}} :  S_{e'} \rightarrow \mathfrak{h}'/W'$ of the adjoint quotient $\chi'$ of $\mathfrak{g}'$ to $S_{e'}$ is a semiuniversal deformation of a simple singularity of type $\Delta(\Gamma')$. We have seen before that $\overline{\alpha^\Omega} : X_{\Gamma,\Omega}\rightarrow (\mathfrak{h}/W)^\Omega$ is a non-semiuniversal deformation of a simple singularity $X_0$ of type $\Delta(\Gamma')$. Therefore $\overline{\alpha^\Omega}$ is isomorphic to the pullback $\chi'_f$ of $\chi'_{| S_{e'}}$ by a base change $f : (\mathfrak{h}/W)^\Omega \rightarrow \mathfrak{h}'/W'$ whose differential at $0$ is uniquely determined (cf. \cite{Slo80(2)} Section 1.4). We thus have the following diagram : 

\begin{center}
\begin{tikzpicture}[scale=1,  transform shape]
\tikzstyle{point}=[circle,draw,fill]
\tikzstyle{ligne}=[thick]

\node (1) at ( 1.05,3) {$\mathfrak{g}'$};
\node (2) at ( 1,2.5) {$\bigcup$};
\node (3) at ( 1,2) {$(S_{e'},e')$};
\node (4) at ( 1.5,1) {$\chi'_{\big| S_{e'}}$};
\node (5) at ( 1,0) {$(\mathfrak{h}'/W',\overline{0})$};
\node (6) at (1,-1) {$ \mathbb{D}'$};
\node (7) at ( -0.7,0.3) {$\pi'$};
\node (8) at ( -2,0) {$\mathfrak{h}'$};
\node (9) at ( -2,-0.5) {$\bigcup$};
\node (10) at ( -2,-1) {$ \bigcup_{\alpha \in \Phi'^ + }H_\alpha$};
\node (11) at ( 6,0) {$((\mathfrak{h}/W)^\Omega,\overline{0})$};
\node (12) at ( 6.5,1) {$\chi'_f$};
\node (13) at ( 6,2) {$(S_{e'},e') \times_{(\mathfrak{h}'/W',\overline{0})} ((\mathfrak{h}/W)^\Omega,\overline{0})$};
\node (14) at ( 3.5,-0.3) {$f$};
\node (15) at ( 2.8,2.2) {$\varphi$};
\node (16) at ( 12,0) {$((\mathfrak{h}/W)^\Omega,\overline{0})$};
\node (17) at ( 12,2) {$X_{\Gamma,\Omega}/\Omega$};
\node (18) at (9.5,3) {$(X_0,0)$};
\node (19) at ( 12.5,1) {$\overline{\alpha^\Omega}$};
\node (20) at ( 3.5,1) {$\circlearrowleft$};
\node (21) at ( 9,1) {$\circlearrowleft$};
\node (22) at ( 10,2.2) {$\cong$};
\node (23) at ( 9.25,-0.2) {$=$};
\node (24) at ( 9.5,2.5) {$\circlearrowleft$};

\draw  [decoration={markings,mark=at position 1 with
    {\arrow[scale=1.2,>=stealth]{>}}},postaction={decorate}] (3)  --  (5);
\draw  [decoration={markings,mark=at position 1 with
    {\arrow[scale=1.2,>=stealth]{>}}},postaction={decorate}] (8)  --  (5);
\draw  [|-,decoration={markings,mark=at position 1 with
    {\arrow[scale=1.2,>=stealth]{>}}},postaction={decorate}] (10)  --  (6);
\draw  [decoration={markings,mark=at position 1 with
    {\arrow[scale=1.2,>=stealth]{>}}},postaction={decorate}] (8)  --  (5);
\draw  [decoration={markings,mark=at position 1 with
    {\arrow[scale=1.2,>=stealth]{>}}},postaction={decorate}] (11)  --  (5);
 \draw  [decoration={markings,mark=at position 1 with
    {\arrow[scale=1.2,>=stealth]{>}}},postaction={decorate}] (13)  --  (11);
\draw  [decoration={markings,mark=at position 1 with
    {\arrow[scale=1.2,>=stealth]{>}}},postaction={decorate}] (13)  --  (3);
\draw  [decoration={markings,mark=at position 1 with
    {\arrow[scale=1.2,>=stealth]{>}}},postaction={decorate}] (17)  --  (16);
\draw  [decoration={markings,mark=at position 1 with
    {\arrow[scale=1.2,>=stealth]{>}}},postaction={decorate}] (17)  --  (13);
\draw  [decoration={markings,mark=at position 1 with
    {\arrow[scale=1.2,>=stealth]{>}}},postaction={decorate}] (16)  --  (11);    
\draw  [right hook-,decoration={markings,mark=at position 1 with
    {\arrow[scale=1.2,>=stealth]{>}}},postaction={decorate}] (18)  --  (17);    
\draw  [left hook-,decoration={markings,mark=at position 1 with
    {\arrow[scale=1.2,>=stealth]{>}}},postaction={decorate}] (18)  --  (13);
    
\end{tikzpicture}
\end{center}
Set $t \in (\mathfrak{h}/W)^\Omega$. Through the isomorphism, the fiber $(\overline{\alpha^\Omega})^{-1}(t)$ is the space 
\begin{center}
$\{ (\lambda, t) \in S_{e'} \times \{t\} \ | \ \chi'_{\big| S_{e'}}(\lambda) = f(t)\} \cong (\chi'_{\big| S_{e'}})^{-1}(f(t))$. 
\end{center}
Therefore the type of the singular configuration of $(\overline{\alpha^\Omega})^{-1}(t)$ is the same as the one of $(\chi'_{| S_{e'}})^{-1}(f(t))$, which is given by the sub-root system $\phi'=\{\alpha \in \Phi'^+ \ | \ \alpha(h)=0, \mathrm{with } \ h \in \mathfrak{h}' \ \mathrm{ a \ representative \ of } \ f(t)\}$. \\
\newline

The following conjecture describes the link between the singular configurations of the fibers of $\overline{\alpha^\Omega}$ and sub-root systems of the root system $\Phi'$ of type $\Delta(\Gamma')$ : 
\medskip

\begin{conjecture}\label{Conjecture}
Let $\overline{\alpha^\Omega}$ be a deformation of a simple singularity of type $\Delta(\Gamma')$ obtained through $\Delta(\Gamma)-\Delta(\Gamma,\Gamma')-\Delta(\Gamma')$-procedure. With the notations of the previous commutative diagram, there exists a subset $\Theta$ of the set simple roots of $\Phi'$ such that the base change $f$ is an isomorphism on its image $f(\mathfrak{h}/W)= \pi'(\bigcap_{\alpha \in \Theta}H_\alpha)$, and such that if $h' \in \bigcap_{\alpha \in \phi'}H_\alpha  \ \mathlarger{\mathlarger{\subset}} \  \mathfrak{h}'$ with $\Theta  \ \mathlarger{\mathlarger{\subset}} \  \phi'$ and $\phi'$ maximal for $h'$ (i.e. $\alpha(h') \neq 0$ if $\alpha \nin \phi'$), then the singular configuration of the fiber $(\overline{\alpha^\Omega})^{-1}(f^{-1}(\pi'(h')))$ is of the same type as the root system $\phi'$. Furthermore, all singular configurations of the fibers of $\overline{\alpha^\Omega}$ are obtained in this way. \\
Depending on the type of $\overline{\alpha^\Omega}$, the subset $\Theta$ is given in the following table (the numbering of the root systems is the one of Bourbaki (cf. \cite{Bou68})) : 
\renewcommand{\arraystretch}{1.4}
\begin{center}
\begin{adjustwidth}{-1cm}{-1cm}
\begin{longtable}{|c|c|c|}
\hline
 $\Delta(\Gamma)-\Delta(\Gamma,\Gamma')-\Delta(\Gamma')$ & $\Theta$ & $\Theta$ $\mathrm{in}$ $\Delta(\Gamma')$ \\
\hhline{|=|=|=|}
  \multirow{-3}{*}{$A_{2r-1}-B_r-D_{r+2}$} &   \multirow{-3}{*}{$\{\alpha_{r+1},\alpha_{r+2}\}$} & 
 \begin{tikzpicture}[scale=0.6,  transform shape]
\tikzstyle{pointin}=[circle,draw,fill]
\tikzstyle{pointout}=[circle,draw]
\tikzstyle{ligne}=[thick]
\tikzstyle{pointille}=[thick,dotted]

\node (1) at ( -7,0) [pointout] {};
\node (2) at ( -5,0) [pointout] {};
\node (3) at ( -3,0) [pointout] {};
\node (4) at ( 0,0) [pointout] {};
\node (5) at ( 2,1) [pointin] {};
\node (6) at ( 2,-1) [pointin] {};

\node (7) at ( -7,0.5)  {1};
\node (8) at ( -5,0.5)  {2};
\node (9) at ( -3,0.5)  {3};
\node (10) at ( 0,0.5) {r};
\node (11) at ( 2.6,1) {r+1};
\node (12) at ( 2.6,-1) {r+2};

\draw [ligne] (1)  --  (2);
\draw [ligne] (2)  --  (3);
\draw [pointille] (3)  --  (4);
\draw [ligne] (4)  --  (5);
\draw [ligne] (4)  --  (6);

\end{tikzpicture}\\
\hline

$D_{r+1}-C_r-D_{2r}$ & $\{\alpha_{2i+1}\}_{0 \leq i \leq r-1}$ & 
\begin{tabular}[h]{l} \begin{tikzpicture}[scale=0.5,  transform shape]
\tikzstyle{pointin}=[circle,draw,fill]
\tikzstyle{pointout}=[circle,draw]
\tikzstyle{ligne}=[thick]
\tikzstyle{pointille}=[thick,dotted]

\node (1) at ( -12,0) [pointin] {};
\node (2) at ( -10,0) [pointout] {};
\node (3) at ( -7,0) [pointout] {};
\node (4) at ( -5,0) [pointin] {};
\node (5) at ( -3,0) [pointout] {};
\node (6) at ( 0,0) [pointin] {};
\node (7) at ( 2,0) [pointout] {};
\node (8) at ( 4,1) [pointin] {};
\node (9) at ( 4,-1) [pointout] {};

\node (10) at ( -12,0.5)  {1};
\node (11) at ( -10,0.5)  {2};

\node (12) at ( -7,0.5) {r-1};
\node (13) at ( -5,0.5) {r};
\node (14) at ( -3,0.5) {r+1};

\node (15) at ( 0,0.5) {2r-3};
\node (16) at ( 2,0.5) {2r-2};
\node (17) at ( 4.6,1) {2r-1};
\node (18) at ( 4.6,-1) {2r};
\node (19) at ( -4,-1.2) {\begin{Large} $r$ $\mathrm{odd}$ \end{Large}};

\draw [ligne] (1)  --  (2);
\draw [pointille] (2)  --  (3);
\draw [ligne] (3)  --  (4);
\draw [ligne] (4)  --  (5);
\draw [pointille] (5)  --  (6);
\draw [ligne] (6)  --  (7);
\draw [ligne] (7)  --  (8);
\draw [ligne] (7)  --  (9);

\end{tikzpicture}\\
 \begin{tikzpicture}[scale=0.5,  transform shape]
\tikzstyle{pointin}=[circle,draw,fill]
\tikzstyle{pointout}=[circle,draw]
\tikzstyle{ligne}=[thick]
\tikzstyle{pointille}=[thick,dotted]

\node (1) at ( -12,0) [pointin] {};
\node (2) at ( -10,0) [pointout] {};
\node (3) at ( -7,0) [pointin] {};
\node (4) at ( -5,0) [pointout] {};
\node (5) at ( -3,0) [pointin] {};
\node (6) at ( 0,0) [pointin] {};
\node (7) at ( 2,0) [pointout] {};
\node (8) at ( 4,1) [pointin] {};
\node (9) at ( 4,-1) [pointout] {};

\node (10) at ( -12,0.5)  {1};
\node (11) at ( -10,0.5)  {2};

\node (12) at ( -7,0.5) {r-1};
\node (13) at ( -5,0.5) {r};
\node (14) at ( -3,0.5) {r+1};

\node (15) at ( 0,0.5) {2r-3};
\node (16) at ( 2,0.5) {2r-2};
\node (17) at ( 4.6,1) {2r-1};
\node (18) at ( 4.6,-1) {2r};
\node (19) at ( -4,-1.2) {\begin{Large} $r$ $\mathrm{even}$ \end{Large}};

\draw [ligne] (1)  --  (2);
\draw [pointille] (2)  --  (3);
\draw [ligne] (3)  --  (4);
\draw [ligne] (4)  --  (5);
\draw [pointille] (5)  --  (6);
\draw [ligne] (6)  --  (7);
\draw [ligne] (7)  --  (8);
\draw [ligne] (7)  --  (9);

\end{tikzpicture}
\end{tabular} \\
\hline
  \multirow{-3.5}{*}{$E_{6}-F_4-E_{7}$} &   \multirow{-3.5}{*}{$\{\alpha_2,\alpha_5,\alpha_7\}$} &
 \begin{tikzpicture}[scale=0.6,  transform shape]
\tikzstyle{pointin}=[circle,draw,fill]
\tikzstyle{pointout}=[circle,draw]
\tikzstyle{ligne}=[thick]
\tikzstyle{pointille}=[thick,dotted]

\node (1) at ( -4,0) [pointout] {};
\node (2) at ( 0,-2) [pointin] {};
\node (3) at ( -2,0) [pointout] {};
\node (4) at ( 0,0) [pointout] {};
\node (5) at ( 2,0) [pointin] {};
\node (6) at ( 4,0) [pointout] {};
\node (7) at ( 6,0) [pointin] {};

\node (8) at ( -4,0.5)  {1};
\node (9) at ( -0.5,-2)  {2};
\node (10) at ( -2,0.5)  {3};
\node (11) at ( 0,0.5) {4};
\node (12) at ( 2,0.5) {5};
\node (13) at ( 4,0.5) {6};
\node (14) at ( 6,0.5) {7};

\draw [ligne] (1)  --  (3);
\draw [ligne] (3)  --  (4);
\draw [ligne] (4)  --  (2);
\draw [ligne] (4)  --  (5);
\draw [ligne] (5)  --  (6);
\draw [ligne] (6)  --  (7);

\end{tikzpicture} \\
\hline
  \multirow{-3.5}{*}{$D_{4}-G_2-E_{6}$} &   \multirow{-3.5}{*}{$\{\alpha_1,\alpha_3,\alpha_5,\alpha_6\}$} & 
 \begin{tikzpicture}[scale=0.6,  transform shape]
\tikzstyle{pointin}=[circle,draw,fill]
\tikzstyle{pointout}=[circle,draw]
\tikzstyle{ligne}=[thick]
\tikzstyle{pointille}=[thick,dotted]

\node (1) at ( -4,0) [pointin] {};
\node (2) at ( 0,-2) [pointout] {};
\node (3) at ( -2,0) [pointin] {};
\node (4) at ( 0,0) [pointout] {};
\node (5) at ( 2,0) [pointin] {};
\node (6) at ( 4,0) [pointin] {};

\node (7) at ( -4,0.5)  {1};
\node (8) at ( -0.5,-2)  {2};
\node (9) at ( -2,0.5)  {3};
\node (10) at ( 0,0.5) {4};
\node (11) at ( 2,0.5) {5};
\node (12) at ( 4,0.5) {6};

\draw [ligne] (1)  --  (3);
\draw [ligne] (3)  --  (4);
\draw [ligne] (4)  --  (2);
\draw [ligne] (4)  --  (5);
\draw [ligne] (5)  --  (6);

\end{tikzpicture}  \\
\hline
  \multirow{-3.5}{*}{$D_{4}-G_2-E_{7}$} &    \multirow{-3.5}{*}{$\{\alpha_1,\alpha_2,\alpha_3,\alpha_5,\alpha_7\}$} &
 \begin{tikzpicture}[scale=0.6,  transform shape]
\tikzstyle{pointin}=[circle,draw,fill]
\tikzstyle{pointout}=[circle,draw]
\tikzstyle{ligne}=[thick]
\tikzstyle{pointille}=[thick,dotted]

\node (1) at ( -4,0) [pointin] {};
\node (2) at ( 0,-2) [pointin] {};
\node (3) at ( -2,0) [pointin] {};
\node (4) at ( 0,0) [pointout] {};
\node (5) at ( 2,0) [pointin] {};
\node (6) at ( 4,0) [pointout] {};
\node (7) at ( 6,0) [pointin] {};

\node (8) at ( -4,0.5)  {1};
\node (9) at ( -0.5,-2)  {2};
\node (10) at ( -2,0.5)  {3};
\node (11) at ( 0,0.5) {4};
\node (12) at ( 2,0.5) {5};
\node (13) at ( 4,0.5) {6};
\node (14) at ( 6,0.5) {7};

\draw [ligne] (1)  --  (3);
\draw [ligne] (3)  --  (4);
\draw [ligne] (4)  --  (2);
\draw [ligne] (4)  --  (5);
\draw [ligne] (5)  --  (6);
\draw [ligne] (6)  --  (7);

\end{tikzpicture} \\
\hline
 \end{longtable}
 {\addtocounter{table}{-1}}
 \end{adjustwidth}
\vspace*{-0.5cm}
  \captionof{table}{}
  \label{Table}
\end{center}

\end{conjecture}

The next statement validates the conjecture for small rank cases.

\begin{theorem}\label{Maintheorem}
The Conjecture~\ref{Conjecture} is true for the types : 
\begin{itemize}
\centering
\item $A_3-B_2-D_4$, 
\item $A_5-B_3-D_5$,
\item  $D_4-C_3-D_6$,
\item  $D_4-G_2-E_6$,
\item  $D_4-G_2-E_7$,
\item  $E_6-F_4-E_7$.
\end{itemize}

\end{theorem}

\section{Proof of Theorem~\ref{Maintheorem}}\label{Proof}

In what follows, the set of coordinates on $\mathfrak{h}/W$ and $\mathfrak{h}'/W'$ will always be the flat coordinates (cf. \cite{Saito79} and \cite{Saito83} for details). 

\subsection{Strategy of the proof}
The strategy for the proof of each case is twofold :
\begin{itemize}

\item Part 1 : \begin{itemize} 
\item Compute the deformation $\alpha^\Omega$ and its discriminant.
\item Decompose the discriminant depending on what singular configurations appear in the fibers of $\alpha^\Omega$.
\item Compute the deformation $\overline{\alpha^\Omega}$.
\item Determine the types of singular configurations appearing in the fibers of $\overline{\alpha^\Omega}$ above the previously obtained decomposition of $(\mathfrak{h}/W)^\Omega$.
\end{itemize}

\item Part 2 : \begin{itemize} 
\item Compute all realizations of the sub-root systems of the root system of type $\Delta(\Gamma')$ containing $\Theta$.
\item Compute the flat coordinates of the Cartan subalgebra $\mathfrak{h}'$ of type $\Delta(\Gamma')$, and the restriction of the map $\pi' : \mathfrak{h}' \rightarrow \mathfrak{h}'/W'$ to the subspace $\bigcap_{\alpha \in \Theta}H_\alpha$, with $H_\alpha=\{h \in \mathfrak{h}' \ | \ \alpha(h)=0\}$.
\item Construct the base change $f : (\mathfrak{h}/W)^\Omega \xrightarrow{\hspace*{1cm}} \mathfrak{h}'/W'$.
\item Verify that $f : (\mathfrak{h}/W)^\Omega \stackrel{\cong}{\xrightarrow{\hspace*{1cm}}}   \pi'(\bigcap_{\alpha \in \Theta}H_\alpha)$.
\item Check that the morphism $f$ realizes a bijection between the singular configurations of the fibers of $\overline{\alpha^\Omega}$ and the sub-root systems of $\Delta(\Gamma')$ containing $\Theta$.
\end{itemize}
\end{itemize}

\subsection{Case $A_3-B_2-D_4$}
In this section $\Gamma=\mathcal{C}_4$, $\Gamma'=\mathcal{D}_2$ and $\Omega=\mathbb{Z}/2\mathbb{Z}=<\sigma>$. It is known from \cite{Cara17} that the map $\alpha^\Omega$ is given by the projection
\begin{center}
\begin{tikzpicture}[scale=1]
\node (1) at ( 0,0) {$ X_{\Gamma, \Omega}=\{((x,y,z),(t_2,0,t_4)) \in \cc^3 \times \mathfrak{h}/W \ | \ z^4+t_2z^2+t_4+\frac{t_2^2}{8}=xy\}$};
\node (2) at ( 0,-2) {$(\mathfrak{h}/W)^\Omega=\{ (t_2,0,t_4) \in  \mathfrak{h}/W\}$};
\node (3) at ( 0.5, -1)  {$\alpha^\Omega$};

\draw  [decoration={markings,mark=at position 1 with
    {\arrow[scale=1.2,>=stealth]{>}}},postaction={decorate}] (1)  --  (2);
\end{tikzpicture}
\end{center}
Furthermore, the action of $\Omega$ on a fiber is $\sigma.(x,y,z)=(y,x,-z)$, and so the quotient $\overline{\alpha^\Omega}$ is given by
\begin{center}
\begin{tikzpicture}[scale=1]
\node (1) at ( 0,0) {$ X_{\Gamma, \Omega}/\Omega=\{((X,Y,Z),(t_2,0,t_4)) \in \cc^3 \times \mathfrak{h}/W \ | \ Z(X^2-4Z^2)+W^2-4t_2Z^2-4(t_4+\frac{t_2^2}{8})Z=0\}$};
\node (2) at ( 0,-2) {$(\mathfrak{h}/W)^\Omega=\{ (t_2,0,t_4) \in  \mathfrak{h}/W\}$};
\node (3) at ( 0.5, -1)  {$\overline{\alpha^\Omega}$};

\draw  [decoration={markings,mark=at position 1 with
    {\arrow[scale=1.2,>=stealth]{>}}},postaction={decorate}] (1)  --  (2);
\end{tikzpicture}
\end{center}
Before quotient, the discriminant of $\alpha^\Omega$ is 

\begin{center}
$\{(t_2,0,t_4) \in  \mathfrak{h}/W \ | \ (t_2^2+8t_4)(t_2^2-8t_4)=0\}$. 
\end{center}

\noindent After quotient, it is known from Theorem~\ref{ThmSingular} that all of $(\mathfrak{h}/W)^\Omega$ composes the discriminant. By studying explicitly the singularities in the fibers $(\alpha^\Omega)^{-1}(t_2,0,t_4)$ and $(\overline{\alpha^\Omega})^{-1}(t_2,0,t_4)$, we obtain the following types : 
\smallskip
\begin{center}
\renewcommand{\arraystretch}{1.4}
\begin{tabular}{|c|c|c|}
\hline
$(t_2,0,t_4)$ & Configuration of $(\alpha^\Omega)^{-1}(t_2,0,t_4)$ & Configuration of $(\overline{\alpha^\Omega})^{-1}(t_2,0,t_4)$ \\
\hhline{|=|=|=|}
Generic point & $p_1+p_2$ & $A_1+A_1$ \\
\hline
$t_4=-\frac{t_2^2}{8}\neq 0$ & $A_1$ & $A_3$ \\
\hline
 \multirow{-3}{*}{$t_4=\frac{t_2^2}{8} \neq 0$} & \begin{tikzpicture}[scale=0.8]
\node (9) at (  - 4,0)  {$p_1$};
\node (10) at (  - 3,0)  {$+$};
\node (1) at (  - 2,0)  {$p_2$};
\node (2) at ( 0,0)  {$A_1$};
\node (3) at ( 2,0) {$A_1$};
\node (4) at (  0.2, - 0.25) {};
\node (5) at ( 1.8, - 0.25) {};
\node (6) at ( 1, - 1) {$\Omega$};
\node (7) at ( -1, 0) {$+$};
\node (8) at ( 1, 0) {$+$};
\draw[<->] (4.west) to[out= - 70, in= - 110] (5.east);
\end{tikzpicture} &  \multirow{-3}{*}{$A_1+A_1+A_1$} \\
\hline
$t_2=t_4=0$ & $A_3$ & $D_4$ \\
\hline
\end{tabular}
 \captionof{table}{}
  \label{TypesofsingularA3}
\end{center}
\medskip

\noindent with $p_i$ ($i=1,2$) denoting a smooth point, and if the action of $\Omega$ is not mentioned, the point (smooth or singular) is $\Omega$-fixed. \\
\medskip

Let $\Pi'=\{\alpha_1,\alpha_2,\alpha_3,\alpha_4\}$ be a set of simple roots of the root system of type $D_4$ (cf. Table~\ref{Table} for the numbering). All the sub-root systems of $D_4$ containing $\Theta=\{\alpha_3,\alpha_4\}$ are either of type $A_1+A_1$, $A_1+A_1+A_1$, $A_3$ or $D_4$, and their realizations $\phi'$ are given in the following table : 
\smallskip

\begin{center}
\renewcommand{\arraystretch}{1.4}
\begin{tabular}{|c|c|}
\hline
Type of $\phi'$ & Realizations of $\phi'$ \\
\hhline{|=|=|}
$A_1+A_1$ &  $<\alpha_3,\alpha_4>$ \\
\hline
$A_1+A_1+A_1$  & $<\alpha_1,\alpha_3,\alpha_4>$, $<\alpha_1+2\alpha_2+\alpha_3+\alpha_4,\alpha_3,\alpha_4>$ \\
\hline
 $A_3$ & $<\alpha_2,\alpha_3,\alpha_4>$, $<\alpha_1+\alpha_2,\alpha_3,\alpha_4>$ \\
\hline
$D_4$  & $<\alpha_1,\alpha_2,\alpha_3,\alpha_4>$ \\
\hline
\end{tabular}
 \captionof{table}{}
  \label{realizationsA3}
\end{center}
where $<...>$ means the root system generated by $"..."$.
\medskip

Set $(e_1,e_2,e_3,e_4)$ an orthonormal basis of $\mathfrak{h}'$ of type $D_4$, and define $(\epsilon_1,\epsilon_2,\epsilon_3,\epsilon_4)$ its dual basis. Following \cite{Bou68}, the simple roots of $D_4$ are 
\medskip

\begin{center}
\renewcommand{\arraystretch}{1.2}
\raisebox{-.55\height}{$\left\lbrace\begin{array}{ccl}
\alpha_1 & = & \epsilon_1-\epsilon_2, \\
\alpha_2 & = & \epsilon_2-\epsilon_3, \\
\alpha_3 & = & \epsilon_3-\epsilon_4, \\
\alpha_4 & = & \epsilon_3+\epsilon_4. \\
\end{array}\right.$}
\end{center}
\medskip

Let $h'=(\xi_1,\xi_2,\xi_3,\xi_4) \in \mathfrak{h}'$, with the coordinates being expressed in the basis $(e_1,e_2,e_3,e_4)$. Then $h' \in H_{\alpha_3} \bigcap H_{\alpha_4}$ if and only if $\xi_3=\xi_4=0$. \\

Using \cite{SYS80}, one can compute the flat coordinates of $D_4$ restricted to the subspace $H_{\alpha_3} \bigcap H_{\alpha_4}  \ \mathlarger{\mathlarger{\subset}} \  \mathfrak{h}'$ and obtains
\medskip

\begin{center}
\renewcommand{\arraystretch}{1.2}
\raisebox{-.55\height}{$\left\lbrace\begin{array}{ccl}
\psi_2(\xi_1,\xi_2,0,0) & = & \xi_1^2+\xi_2^2, \\
\psi_4(\xi_1,\xi_2,0,0) & = & -\frac{1}{4}( \xi_1^2-\xi_2^2)^2, \\
\psi_6(\xi_1,\xi_2,0,0) & = & -\frac{1}{6}(\xi_1^2+\xi_2^2)\xi_1^2\xi_2^2+\frac{7}{216}(\xi_1^2+\xi_2^2)^3, \\
\psi(\xi_1,\xi_2,0,0) & = & 0. \\
\end{array}\right.$}
\end{center}
\medskip
One notices that $\psi_6=-\frac{1}{108}\psi_2^3-\frac{1}{6}\psi_2\psi_4$. \\

Define 
\begin{center}
\renewcommand{\arraystretch}{1.2}
$\begin{array}[t]{rccc}
f :  & (\mathfrak{h}/W)^\Omega & \rightarrow & \mathfrak{h}'/W' \\
   & \begin{pmatrix}
   t_2  \\
   0 \\
   t_4
   \end{pmatrix} & \mapsto & \begin{pmatrix}
   t_2  \\
   t_4-\frac{t_2^2}{8} \\
   \frac{5}{432}t_2^3-\frac{1}{6}t_2t_4 \\
   0
   \end{pmatrix}
\end{array}$
\end{center}
The morphism $f$ is clearly injective, and if $(\psi_2,\psi_4,\psi_6,0) \in \pi'(H_{\alpha_3} \bigcap H_{\alpha_4})$, then by setting $t_2=\psi_2$ and $t_4=\psi_4+\frac{1}{8}\psi_2^2$, it follows that $f(t_2,0,t_4)=(\psi_2,\psi_4,\psi_6,0)$. Hence
\begin{center}
 \begin{tikzpicture}[scale=1,  transform shape]
\node (1) at ( 0,0)  {$f :  (\mathfrak{h}/W)^\Omega$};
\node (2) at ( 5,0)  {$\pi'(H_{\alpha_3} \bigcap H_{\alpha_4})  \ \mathlarger{\mathlarger{\subset}} \  \mathfrak{h}'/W'.$};
\node (3) at ( 2,0.2)  {$\cong$};
\draw  [decoration={markings,mark=at position 1 with
    {\arrow[scale=1.2,>=stealth]{>}}},postaction={decorate}] (1)  --  (2);
\end{tikzpicture}
\end{center}

Let $\phi'$ be a sub-root system in Table~\ref{realizationsA3}, and set $h' \in \bigcap_{\alpha \in \phi'}H_\alpha$. With the formulas given before, one can compute $(t_2,0,t_4)=f^{-1}(\pi'(h'))$, and verify the following correspondence : \begin{center}
\renewcommand{\arraystretch}{1.5}
\begin{tabular}[h]{lcl}
$\phi'$ of type $A_1+A_1$ &  \begin{tikzpicture}[scale=1,  transform shape]
\node (1) at ( 0,0)  {};
\node (2) at ( 2,0)  {};
\draw  [decoration={markings,mark=at position 1 with
    {\arrow[scale=1.2,>=stealth]{>}}},postaction={decorate}] (1)  --  (2);
\draw  [decoration={markings,mark=at position 1 with
    {\arrow[scale=1.2,>=stealth]{>}}},postaction={decorate}] (2)  --  (1);
\end{tikzpicture} & $(t_2,0,t_4)$ generic, \\
$\phi'$ of type $A_1+A_1+A_1$ &  \begin{tikzpicture}[scale=1,  transform shape]
\node (1) at ( 0,0)  {};
\node (2) at ( 2,0)  {};
\draw  [decoration={markings,mark=at position 1 with
    {\arrow[scale=1.2,>=stealth]{>}}},postaction={decorate}] (1)  --  (2);
\draw  [decoration={markings,mark=at position 1 with
    {\arrow[scale=1.2,>=stealth]{>}}},postaction={decorate}] (2)  --  (1);
\end{tikzpicture} & $(t_2,0,t_4)$ verifies $t_4=\frac{t_2^2}{8} \neq 0$, \\
$\phi'$ of type $A_3$ &  \begin{tikzpicture}[scale=1,  transform shape]
\node (1) at ( 0,0)  {};
\node (2) at ( 2,0)  {};
\draw  [decoration={markings,mark=at position 1 with
    {\arrow[scale=1.2,>=stealth]{>}}},postaction={decorate}] (1)  --  (2);
\draw  [decoration={markings,mark=at position 1 with
    {\arrow[scale=1.2,>=stealth]{>}}},postaction={decorate}] (2)  --  (1);
\end{tikzpicture} & $(t_2,0,t_4)$ verifies $t_4=-\frac{t_2^2}{8} \neq 0$, \\
$\phi'$ of type $D_4$ &  \begin{tikzpicture}[scale=1,  transform shape]
\node (1) at ( 0,0)  {};
\node (2) at ( 2,0)  {};
\draw  [decoration={markings,mark=at position 1 with
    {\arrow[scale=1.2,>=stealth]{>}}},postaction={decorate}] (1)  --  (2);
\draw  [decoration={markings,mark=at position 1 with
    {\arrow[scale=1.2,>=stealth]{>}}},postaction={decorate}] (2)  --  (1);
\end{tikzpicture} & $(t_2,0,t_4)=(0,0,0)$. \\
\end{tabular}
\end{center}
With Table~\ref{TypesofsingularA3}, we see that the singular configuration in the fiber of $\overline{\alpha^\Omega}$ above $f^{-1}(\pi'(h'))$ is of the same type as $\phi'$. Therefore for the type $A_3-B_2-D_4$, the map $f$ realizes a bijection between the singular configurations of the fibers of $\overline{\alpha^\Omega}$ and the sub-root systems of $D_4$ containing $\Theta=\{\alpha_3,\alpha_4\}$.
\medskip

\subsection{Case $A_5-B_3-D_5$}
In this part $\Gamma=\mathcal{C}_6$, $\Gamma'=\mathcal{D}_3$ and $\Omega=\mathbb{Z}/2\mathbb{Z}=<\sigma>$. The map $\alpha^\Omega$ is computed in \cite{Cara17} and is given by the projection
\begin{center}
\begin{tikzpicture}[scale=1]
\node (1) at ( 0,0) {$X_{\Gamma, \Omega}=\{((x,y,z),(t_2,0,t_4,0,t_6)) \in \cc^3 \times  \mathfrak{h}/W \ | \ z^6+t_2z^4+(t_4+\frac{t_2^2}{4})z^2+t_6+\frac{t_2t_4}{6}+\frac{t_2^3}{108}=xy\}$};
\node (2) at ( 0,-2) {$(\mathfrak{h}/W)^\Omega=\{ (t_2,0,t_4,0,t_6) \in  \mathfrak{h}/W\}$};
\node (3) at ( 0.5, -1)  {$\alpha^\Omega$};

\draw  [decoration={markings,mark=at position 1 with
    {\arrow[scale=1.2,>=stealth]{>}}},postaction={decorate}] (1)  --  (2);
\end{tikzpicture}
\end{center}
Furthermore, the action of $\Omega$ on a fiber is $\sigma.(x,y,z)=(-y,-x,-z)$. The quotient $\overline{\alpha^\Omega}$ is then given by
\begin{center}
\begin{tikzpicture}[scale=1]
\node (1) at ( 0,0) {$ X_{\Gamma, \Omega}/\Omega=\{((X,Y,Z),(t_2,0,t_4,0,t_6)) \in \cc^3 \times \mathfrak{h}/W \ | \ Z(X^2+4Z^3)+W^2+4t_2Z^3$};
\node (2) at ( 0,-0.5) {$+4(t_4+\frac{t_2^2}{4})Z^2+4(t_6+\frac{t_2t_4}{6}+\frac{t_2^3}{108})Z=0\}$};
\node (3) at ( 0,-2.5) {$(\mathfrak{h}/W)^\Omega=\{ (t_2,0,t_4,0,t_6) \in  \mathfrak{h}/W\}$};
\node (4) at ( 0.5, -1.5)  {$\overline{\alpha^\Omega}$};

\draw  [decoration={markings,mark=at position 1 with
    {\arrow[scale=1.2,>=stealth]{>}}},postaction={decorate}] (2)  --  (3);
\end{tikzpicture}
\end{center}
Before quotient, the discriminant of $\alpha^\Omega$ is given by $H_1 \bigcup H_2  \ \mathlarger{\mathlarger{\subset}} \  (\mathfrak{h}/W)^{\Omega}$ with 
\begin{center}
$H_1=\{t_6+\frac{t_2t_4}{6}+\frac{t_2^3}{108}=0\}$, $H_2=\{-\frac{t_2^6}{432}+\frac{t_2^4t_4}{12}-\frac{t_2^2t_4^2}{4}-9t_2t_4t_6+4t_4^3+27t_6^2=0\}$.
\end{center}
 After quotient, Theorem~\ref{ThmSingular} says that the discriminant is all of $(\mathfrak{h}/W)^\Omega$. By studying explicitly the singularities in the fibers $(\alpha^\Omega)^{-1}(t_2,0,t_4,0,t_6)$ and $(\overline{\alpha^\Omega})^{-1}(t_2,0,t_4,0,t_6)$, we obtain the following types : 
\medskip

\begin{center}
\renewcommand{\arraystretch}{1.4}
\begin{tabular}{|c|c|c|}
\hline
$(t_2,0,t_4,0,t_6)$  & Configuration of $(\alpha^\Omega)^{-1}(t_2,0,t_4,0,t_6)$ & $(\overline{\alpha^\Omega})^{-1}(t_2,0,t_4,0,t_6)$ \\
\hhline{|=|=|=|}
Generic point & $p_1+p_2$ & $A_1+A_1$ \\
\hline
$H_1 \bigcap \{t_4 \neq 0 \}$ & $A_1$ & $A_3$ \\
\hline
 \multirow{-3}{*}{$H_1 \bigcap \{t_4 = 0, t_2 \neq 0 \}$} & \begin{tikzpicture}[scale=0.8]
\node (1) at (  - 2,0)  {$A_1$};
\node (2) at ( 0,0)  {$A_1$};
\node (3) at ( 2,0) {$A_1$};
\node (4) at (  0.2, - 0.25) {};
\node (5) at ( 1.8, - 0.25) {};
\node (6) at ( 1, - 1) {$\Omega$};
\node (7) at ( -1, 0) {$+$};
\node (8) at ( 1, 0) {$+$};
\draw[<->] (4.west) to[out= - 70, in= - 110] (5.east);
\end{tikzpicture} &  \multirow{-3}{*}{$A_3+A_1$} \\
\hline
$H_1 \bigcap \{t_4 =-\frac{t_2^2}{4} \neq 0 \}$ & $A_3$ & $D_4$ \\
\hline
 \multirow{-3}{*}{$H_2 \backslash \{0\}$} &  \begin{tikzpicture}[scale=0.8]
\node (10) at (  - 3,0)  {$+$};
\node (9) at (  - 4,0)  {$p_1$};
\node (1) at (  - 2,0)  {$p_2$};
\node (2) at ( 0,0)  {$A_1$};
\node (3) at ( 2,0) {$A_1$};
\node (4) at (  0.2, - 0.25) {};
\node (5) at ( 1.8, - 0.25) {};
\node (6) at ( 1, - 1) {$\Omega$};
\node (7) at ( -1, 0) {$+$};
\node (8) at ( 1, 0) {$+$};
\draw[<->] (4.west) to[out= - 70, in= - 110] (5.east);
\end{tikzpicture} &  \multirow{-3}{*}{$A_1+A_1+A_1$} \\
\hline
 \multirow{-3}{*}{$H_2 \bigcap \{t_4 = \frac{t_2^2}{12} \neq 0 \}$} &  \begin{tikzpicture}[scale=0.8]
\node (10) at (  - 3,0)  {$+$};
\node (9) at (  - 4,0)  {$p_1$};
\node (1) at (  - 2,0)  {$p_2$};
\node (2) at ( 0,0)  {$A_2$};
\node (3) at ( 2,0) {$A_2$};
\node (4) at (  0.2, - 0.25) {};
\node (5) at ( 1.8, - 0.25) {};
\node (6) at ( 1, - 1) {$\Omega$};
\node (7) at ( -1, 0) {$+$};
\node (8) at ( 1, 0) {$+$};
\draw[<->] (4.west) to[out= - 70, in= - 110] (5.east);
\end{tikzpicture} &  \multirow{-3}{*}{$A_1+A_1+A_2$} \\
\hline
$t_2=t_4=t_6=0$ & $A_5$ & $D_5$ \\
\hline
\end{tabular}
 \captionof{table}{}
  \label{TypesofsingularA5}
\end{center}
\medskip

\noindent with $p_i$ ($i=1,2$) denoting a smooth point, and if the action of $\Omega$ is not mentioned, the point (smooth or singular) is fixed. \\
\newline

Let $\Pi'=\{\alpha_1,\alpha_2,\alpha_3,\alpha_4,\alpha_5\}$ be a set of simple roots of the root system of type $D_5$ (cf. Table~\ref{Table} for the numbering). All the sub-root systems $\phi'$ of $D_5$ containing $\Theta=\{\alpha_4,\alpha_5\}$, as well as their realizations are given in the following table : 
\medskip

\begin{center}
\renewcommand{\arraystretch}{1.5}
\begin{longtable}{|c|>{\centering\arraybackslash}p{12cm}|}
\hline
Type of $\phi'$ & Realizations of $\phi'$ \\
\hhline{|=|=|}
$A_1+A_1$ &  $<\alpha_4,\alpha_5>$ \\
\hline
 \multirow{2}{*}{$A_1+A_1+A_1$}  & $<\alpha_1,\alpha_4,\alpha_5>$, $<\alpha_2,\alpha_4,\alpha_5>$, $<\alpha_1+\alpha_2+\alpha_4,\alpha_5>$,  $<\alpha_1+2\alpha_2+2\alpha_3+\alpha_4+\alpha_5,\alpha_4,\alpha_5>$, $<\alpha_1+\alpha_2+2\alpha_3+\alpha_4+\alpha_5,\alpha_4,\alpha_5>$,  $<\alpha_2+2\alpha_3+\alpha_4+\alpha_5,\alpha_4,\alpha_5>$ \\
\hline
 $A_3$ &  $<\alpha_3,\alpha_4,\alpha_5>$,  $<\alpha_2+\alpha_3,\alpha_4,\alpha_5>$,  $<\alpha_1+\alpha_2+\alpha_3,\alpha_4,\alpha_5>$ \\
\hline
$A_1+A_1+A_2$ &  $<\alpha_1,\alpha_2,\alpha_4,\alpha_5>$, $<\alpha_1,\alpha_2+2\alpha_3+\alpha_4+\alpha_5,\alpha_4,\alpha_5>$, $<\alpha_2,\alpha_1+\alpha_2+2\alpha_3+\alpha_4+\alpha_5,\alpha_4,\alpha_5>$,  $<\alpha_1+\alpha_2,\alpha_2+2\alpha_3+\alpha_4+\alpha_5,\alpha_4,\alpha_5>$\\
\hline
$A_3+A_1$ & $<\alpha_1,\alpha_3,\alpha_4,\alpha_5>$, $<\alpha_1+\alpha_2,\alpha_2+\alpha_3,\alpha_4,\alpha_5>$,  $<\alpha_1+2\alpha_2+2\alpha_3+\alpha_4+\alpha_5,\alpha_3,\alpha_4,\alpha_5>$, $<\alpha_1+\alpha_2+2\alpha_3+\alpha_4+\alpha_5,\alpha_2+\alpha_3,\alpha_4,\alpha_5>$, $<\alpha_1+\alpha_2+\alpha_3,\alpha_2,\alpha_4,\alpha_5>$, $<\alpha_1+\alpha_2+\alpha_3,\alpha_2+2\alpha_3+\alpha_4+\alpha_5,\alpha_4,\alpha_5>$,\\
\hline
$D_4$  &  $<\alpha_2,\alpha_3,\alpha_4,\alpha_5>$,  $<\alpha_1+\alpha_2,\alpha_3,\alpha_4,\alpha_5>$,  $<\alpha_1,\alpha_2+\alpha_3,\alpha_4,\alpha_5>$\\
\hline
$D_5$  &   $<\alpha_1,\alpha_2,\alpha_3,\alpha_4,\alpha_5>$\\
\hline
\end{longtable}
{\addtocounter{table}{-1}}
 \captionof{table}{}
  \label{realizationsA5}
\end{center}
\medskip

Set $(e_1,e_2,e_3,e_4,e_5)$ an orthonormal basis of $\mathfrak{h}'$ of type $D_5$, and define $(\epsilon_1,\epsilon_2,\epsilon_3,\epsilon_4,\epsilon_5)$ its dual basis. Following \cite{Bou68}, the simple roots of $D_5$ are 
\medskip

\begin{center}
\renewcommand{\arraystretch}{1.2}
\raisebox{-.55\height}{$\left\lbrace\begin{array}{ccl}
\alpha_1 & = & \epsilon_1-\epsilon_2, \\
\alpha_2 & = & \epsilon_2-\epsilon_3, \\
\alpha_3 & = & \epsilon_3-\epsilon_4, \\
\alpha_4 & = & \epsilon_4-\epsilon_5. \\
\alpha_5 & = & \epsilon_4+\epsilon_5. \\
\end{array}\right.$}
\end{center}
\medskip

Let $h'=(\xi_1,\xi_2,\xi_3,\xi_4,\xi_5) \in \mathfrak{h}'$, with the coordinates being expressed in the basis $(e_1,e_2,e_3,e_4,e_5)$. Then $h' \in H_{\alpha_4} \bigcap H_{\alpha_5}$ if and only if $\xi_4=\xi_5=0$. \\

Using \cite{SYS80}, one can compute the flat coordinates of $D_5$ restricted to the subspace $H_{\alpha_4} \bigcap H_{\alpha_5}  \ \mathlarger{\mathlarger{\subset}} \  \mathfrak{h}'$ and obtain
\medskip

\begin{center}
\renewcommand{\arraystretch}{1.2}
\raisebox{-.55\height}{$\left\lbrace\begin{array}{ccl}
\psi_2(\xi_1,\xi_2,\xi_3) & = & \xi_1^2+\xi_2^2+\xi_3^2, \\
\psi_4(\xi_1,\xi_2,\xi_3) & = & \xi_1^2\xi_2^2+\xi_1^2\xi_3^2+\xi_2^2\xi_3^2-\frac{5}{16}(\xi_1^2+\xi_2^2+\xi_3^2)^2, \\
\psi_6(\xi_1,\xi_2,\xi_3) & = & \xi_1^2\xi_2^2\xi_3^2-\frac{3}{8}(\xi_1^2+\xi_2^2+\xi_3^2)(\xi_1^2\xi_2^2+\xi_1^2\xi_3^2+\xi_2^2\xi_3^2)+\frac{11}{128}(\xi_1^2+\xi_2^2+\xi_3^2)^3, \\
\psi_8(\xi_1,\xi_2,\xi_3) & = & \begin{array}[t]{l} -\frac{1}{8}(\xi_1^2+\xi_2^2+\xi_3^2)\xi_1^2\xi_2^2\xi_3^2-\frac{1}{16}(\xi_1^2\xi_2^2+\xi_1^2\xi_3^2+\xi_2^2\xi_3^2)^2 \\ +\frac{9}{128}(\xi_1^2\xi_2^2+\xi_1^2\xi_3^2+\xi_2^2\xi_3^2)(\xi_1^2+\xi_2^2+\xi_3^2)^2-\frac{51}{4096}(\xi_1^2+\xi_2^2+\xi_3^2)^4, \end{array} \\
\psi(\xi_1,\xi_2,\xi_3) & = & 0. \\
\end{array}\right.$}
\end{center}
\medskip
One can see that $\psi_8=-\frac{1}{2048}\psi_2^4-\frac{1}{8}\psi_2\psi_6-\frac{1}{64}\psi_2^2\psi_4-\frac{1}{16}\psi_4^2$. \\

Define 
\begin{center}
$\begin{array}[t]{rccc}
f :  & (\mathfrak{h}/W)^\Omega & \rightarrow & \mathfrak{h}'/W' \\
   & \begin{pmatrix}
   t_2  \\
   0 \\
   t_4 \\
   0 \\
   t_6
   \end{pmatrix} & \mapsto & \renewcommand{\arraystretch}{1.2}\begin{pmatrix}
   t_2  \\
   t_4-\frac{1}{16}t_2^2 \\
  t_6-\frac{5}{24}t_2t_4+\frac{5}{3456}t_2^3 \\
  \frac{7}{110592}t_2^4-\frac{1}{8}t_2t_6+\frac{7}{384}t_2^2t_4-\frac{1}{16}t_4^2 \\
   0
   \end{pmatrix}
\end{array}$
\end{center}
We see that $f$ is an injective morphism, and if $(\psi_2,\psi_4,\psi_6,\psi_8,0) \in \pi'(H_{\alpha_4} \bigcap H_{\alpha_5})$, then by setting 
\begin{center}
\renewcommand{\arraystretch}{1.2}
\raisebox{-.55\height}{$\left\lbrace\begin{array}{l}
t_2=\psi_2, \\
t_4=\psi_4+\frac{1}{16}\psi_2^2, \\
t_6=\psi_6+\frac{5}{24}\psi_2\psi_4+\frac{5}{432}\psi_2^3,
\end{array}\right.$}
\end{center}
it follows that $f(t_2,0,t_4,0,t_6)=(\psi_2,\psi_4,\psi_6,\psi_8,0)$. Hence
\begin{center}
 \begin{tikzpicture}[scale=1,  transform shape]
\node (1) at ( 0,0)  {$f :  (\mathfrak{h}/W)^\Omega$};
\node (2) at ( 5,0)  {$\pi'(H_{\alpha_4} \bigcap H_{\alpha_5})  \ \mathlarger{\mathlarger{\subset}} \  \mathfrak{h}'/W'.$};
\node (3) at ( 2,0.2)  {$\cong$};
\draw  [decoration={markings,mark=at position 1 with
    {\arrow[scale=1.2,>=stealth]{>}}},postaction={decorate}] (1)  --  (2);
\end{tikzpicture}
\end{center}

Let $\phi'$ be a sub-root system in Table~\ref{realizationsA5}, and set $h' \in \bigcap_{\alpha \in \phi'}H_\alpha$. With the formulas just given, one can compute $(t_2,0,t_4,0,t_6)=f^{-1}(\pi'(h'))$, and check the following correspondence : \begin{center}
\renewcommand{\arraystretch}{1.5}
\begin{tabular}[h]{lcl}
$\phi'$ of type $A_1+A_1$ &  \begin{tikzpicture}[scale=1,  transform shape]
\node (1) at ( 0,0)  {};
\node (2) at ( 2,0)  {};
\draw  [decoration={markings,mark=at position 1 with
    {\arrow[scale=1.2,>=stealth]{>}}},postaction={decorate}] (1)  --  (2);
\draw  [decoration={markings,mark=at position 1 with
    {\arrow[scale=1.2,>=stealth]{>}}},postaction={decorate}] (2)  --  (1);
\end{tikzpicture} & $(t_2,0,t_4)$ generic, \\
$\phi'$ of type $A_1+A_1+A_1$ &  \begin{tikzpicture}[scale=1,  transform shape]
\node (1) at ( 0,0)  {};
\node (2) at ( 2,0)  {};
\draw  [decoration={markings,mark=at position 1 with
    {\arrow[scale=1.2,>=stealth]{>}}},postaction={decorate}] (1)  --  (2);
\draw  [decoration={markings,mark=at position 1 with
    {\arrow[scale=1.2,>=stealth]{>}}},postaction={decorate}] (2)  --  (1);
\end{tikzpicture} & $(t_2,0,t_4)$ verifies $(t_2,0,t_4,0,t_6) \in H_2 \backslash \{0\}$, \\
$\phi'$ of type $A_3$ &  \begin{tikzpicture}[scale=1,  transform shape]
\node (1) at ( 0,0)  {};
\node (2) at ( 2,0)  {};
\draw  [decoration={markings,mark=at position 1 with
    {\arrow[scale=1.2,>=stealth]{>}}},postaction={decorate}] (1)  --  (2);
\draw  [decoration={markings,mark=at position 1 with
    {\arrow[scale=1.2,>=stealth]{>}}},postaction={decorate}] (2)  --  (1);
\end{tikzpicture} & $(t_2,0,t_4)$ verifies $(t_2,0,t_4,0,t_6) \in H_1 \bigcap \{t_4 \neq 0\}$, \\
$\phi'$ of type $A_2+A_1+A_1$ &  \begin{tikzpicture}[scale=1,  transform shape]
\node (1) at ( 0,0)  {};
\node (2) at ( 2,0)  {};
\draw  [decoration={markings,mark=at position 1 with
    {\arrow[scale=1.2,>=stealth]{>}}},postaction={decorate}] (1)  --  (2);
\draw  [decoration={markings,mark=at position 1 with
    {\arrow[scale=1.2,>=stealth]{>}}},postaction={decorate}] (2)  --  (1);
\end{tikzpicture} & $(t_2,0,t_4,0,t_6) \in H_2 \bigcap \{t_4 =\frac{t_2^2}{12} \neq 0\}$, \\
$\phi'$ of type $A_3+A_1$ &  \begin{tikzpicture}[scale=1,  transform shape]
\node (1) at ( 0,0)  {};
\node (2) at ( 2,0)  {};
\draw  [decoration={markings,mark=at position 1 with
    {\arrow[scale=1.2,>=stealth]{>}}},postaction={decorate}] (1)  --  (2);
\draw  [decoration={markings,mark=at position 1 with
    {\arrow[scale=1.2,>=stealth]{>}}},postaction={decorate}] (2)  --  (1);
\end{tikzpicture} & $(t_2,0,t_4,0,t_6) \in H_1 \bigcap \{t_4 =0, t_2 \neq 0\}$, \\
$\phi'$ of type $D_4$ &  \begin{tikzpicture}[scale=1,  transform shape]
\node (1) at ( 0,0)  {};
\node (2) at ( 2,0)  {};
\draw  [decoration={markings,mark=at position 1 with
    {\arrow[scale=1.2,>=stealth]{>}}},postaction={decorate}] (1)  --  (2);
\draw  [decoration={markings,mark=at position 1 with
    {\arrow[scale=1.2,>=stealth]{>}}},postaction={decorate}] (2)  --  (1);
\end{tikzpicture} &  $(t_2,0,t_4,0,t_6) \in H_1 \bigcap \{t_4 =-\frac{t_2^2}{4} \neq 0\}$, \\
$\phi'$ of type $D_5$ &  \begin{tikzpicture}[scale=1,  transform shape]
\node (1) at ( 0,0)  {};
\node (2) at ( 2,0)  {};
\draw  [decoration={markings,mark=at position 1 with
    {\arrow[scale=1.2,>=stealth]{>}}},postaction={decorate}] (1)  --  (2);
\draw  [decoration={markings,mark=at position 1 with
    {\arrow[scale=1.2,>=stealth]{>}}},postaction={decorate}] (2)  --  (1);
\end{tikzpicture} &  $(t_2,0,t_4,0,t_6)=(0,0,0,0,0)$. \\
\end{tabular}
\end{center}
With Table~\ref{TypesofsingularA5}, we see that the singular configuration in the fiber of $\overline{\alpha^\Omega}$ above $f^{-1}(\pi'(h'))$ is of the same type as $\phi'$. Therefore for the type $A_5-B_3-D_5$, the map $f$ realizes a bijection between the singular configurations of the fibers of $\overline{\alpha^\Omega}$ and the sub-root systems of $D_5$ containing $\Theta=\{\alpha_4,\alpha_5\}$.
\medskip

\subsection{Case $D_4-C_3-D_6$}

Here $\Gamma=\mathcal{D}_2$, $\Gamma'=\mathcal{D}_4$ and $\Omega=\mathbb{Z}/2\mathbb{Z}=<\sigma>$. It is known from \cite{Cara17} that the map $\alpha^\Omega$ is the projection
\begin{center}
\begin{tikzpicture}[scale=1]
\node (1) at ( 0,0) {$X_{\Gamma, \Omega}=\{((x,y,z),(t_2,t_4,t_6,0)) \in \cc^3 \times  \mathfrak{h}/W \ | \ z^2=xy(x+y)-\frac{t_2}{2}xy-\frac{t_4}{4}x+\frac{1}{4}(t_6+\frac{t_2t_4}{6}+\frac{t_2^3}{108})\}$};
\node (2) at ( 0,-2) {$(\mathfrak{h}/W)^\Omega=\{ (t_2,t_4,t_6,0) \in  \mathfrak{h}/W\}$};
\node (3) at ( 0.5, -1)  {$\alpha^\Omega$};

\draw  [decoration={markings,mark=at position 1 with
    {\arrow[scale=1.2,>=stealth]{>}}},postaction={decorate}] (1)  --  (2);
\end{tikzpicture}
\end{center}
Furthermore, the action of $\Omega$ on a fiber is $\sigma.(x,y,z)=(x,-x-y+\frac{t_2}{2},-z)$, and the quotient $\overline{\alpha^\Omega}$ is given by
\begin{center}
\begin{tikzpicture}[scale=1]
\node (1) at ( 0,0) {$ X_{\Gamma, \Omega}/\Omega=\{((X,Y,Z),(t_2,t_4,t_6,0)) \in \cc^3 \times \mathfrak{h}/W \ | \ -\frac{1}{64}X^5+XY^2-W^2$};
\node (2) at ( 0,-0.5) {$+A_{X^4}X^4+A_{X^3}X^3+A_{X^2}X^2+A_{X}X+A_{Y}Y+A_0=0\}$};
\node (3) at ( 0,-2.5) {$(\mathfrak{h}/W)^\Omega=\{ (t_2,t_4,t_6,0) \in  \mathfrak{h}/W\}$};
\node (4) at ( 0.5, -1.5)  {$\overline{\alpha^\Omega}$};
\draw  [decoration={markings,mark=at position 1 with
    {\arrow[scale=1.2,>=stealth]{>}}},postaction={decorate}] (2)  --  (3);
\end{tikzpicture}
\end{center}
with
\begin{center}
 $\raisebox{-.4\height}{$\left\lbrace  \renewcommand{\arraystretch}{1.5}  \begin{array}[h]{l}
 A_{X^4}= \frac{t_2}{32},\\
A_{X^3}=-\frac{3}{128}t_2^2-\frac{1}{32}t_4, \\
A_{X^2}=\frac{7}{192}t_2t_4+\frac{1}{32}t_6+\frac{7}{864}t_2^3, \\
A_{X}=-\frac{1}{32}t_6t_2-\frac{5}{384}t_2^2t_4-\frac{35}{27648}t_2^4-\frac{1}{64}t_4^2, \\
A_{Y}=\frac{1}{4}t_6+\frac{1}{24}t_2t_4+\frac{1}{432}t_2^3, \\
A_0=\frac{1}{128}t_6t_2^2+\frac{1}{32}t_6t_4+\frac{11}{6912}t_2^3t_4+\frac{1}{192}t_2t_4^2+\frac{1}{13824}t_2^5.
\end{array}\right.$}$ 
\end{center}
Before quotient, the discriminant of $\alpha^\Omega$ is given by $L \bigcup H  \ \mathlarger{\mathlarger{\subset}} \  (\mathfrak{h}/W)^{\Omega}$ with 
\begin{center}
$L=\{t_6+\frac{t_2t_4}{6}+\frac{t_2^3}{108}=0\}$, $H=\{\frac{t_2^6}{6912}-\frac{t_2^4t_4}{192}+\frac{t_2^2t_4^2}{64}+\frac{9}{16}t_2t_4t_6-\frac{t_4^3}{4}-\frac{27}{16}t_6^2=0\}$. 
\end{center}
After quotient, it is known from Theorem~\ref{ThmSingular} that all of $(\mathfrak{h}/W)^\Omega$ is the discriminant. By studying explicitly the singularities in the fibers $(\alpha^\Omega)^{-1}(t_2,t_4,t_6,0)$ and $(\overline{\alpha^\Omega})^{-1}(t_2,t_4,t_6,0)$, we obtain the following types
\medskip

\begin{center}
\renewcommand{\arraystretch}{1.5}
\begin{tabular}{|c|c|c|}
\hline
$(t_2,t_4,t_6,0)$  & Configuration of $(\alpha^\Omega)^{-1}(t_2,t_4,t_6,0)$ & $(\overline{\alpha^\Omega})^{-1}(t_2,t_4,t_6,0)$ \\
\hhline{|=|=|=|}
$(\cc^3 \times \{0\}) \backslash (L \bigcup H)$ & $p_1+p_2+p_3$ & $A_1+A_1+A_1$ \\
\hline
 \multirow{-2.5}{*}{$L \backslash \{t_4 = -\frac{t_2^2}{4} \}$} &  \begin{tikzpicture}[scale=0.8]
\node (1) at ( 0,0)  {$A_1$};
\node (2) at ( 1, 0) {$+$};
\node (3) at ( 2,0) {$A_1$};
\node (4) at (  0.2, - 0.25) {};
\node (5) at ( 1.8, - 0.25) {};
\node (6) at ( 1, - 1) {$\Omega$};
\node (7) at ( 3, 0) {$+$};
\node (8) at (  4,0)  {$p_1$};
\node (9) at ( 5, 0) {$+$};
\node (10) at ( 6,0)  {$p_2$};
\node (11) at ( 7, 0) {$+$};
\node (12) at ( 8,0)  {$p_3$};
\draw[<->] (4.west) to[out= - 70, in= - 110] (5.east);
\end{tikzpicture} & \multirow{-3}{*}{$A_1+A_1+A_1+A_1$}\\
\hline
$H \backslash ((L \bigcap \{t_4 = -\frac{t_2^2}{4}\}) \bigcup \{t_4 = \frac{t_2^2}{12}\})$ & $A_1+p$ & $A_3+A_1$ \\
\hline
$L \bigcap \{t_4 = -\frac{t_2^2}{4} \neq 0\}$ & $A_3+p$ & $D_4+A_1$  \\
\hline
$H \bigcap  \{t_4 = \frac{t_2^2}{12} \neq 0\}$ & $A_2$ &  $A_5$ \\
\hline
 \multirow{-3}{*}{$(L \bigcap H) \backslash \{t_4 = -\frac{t_2^2}{4} \}$} &  \begin{tikzpicture}[scale=0.8]
\node (1) at ( 0,0)  {$A_1$};
\node (2) at ( 1, 0) {$+$};
\node (3) at ( 2,0) {$A_1$};
\node (4) at (  0.2, - 0.25) {};
\node (5) at ( 1.8, - 0.25) {};
\node (6) at ( 1, - 1) {$\Omega$};
\node (7) at ( 3, 0) {$+$};
\node (8) at ( 4,0)  {$A_1$};
\node (9) at (5, 0) {$+$};
\node (10) at (6,0)  {$p$};
\draw[<->] (4.west) to[out= - 70, in= - 110] (5.east);
\end{tikzpicture} &  \multirow{-3}{*}{$A_3+A_1+A_1$} \\
\hline
$t_2=t_4=t_6=0$ & $D_4$ & $D_6$ \\
\hline
\end{tabular}
 \captionof{table}{}
  \label{TypesofsingularD4andZ/2Z}
\end{center}
\medskip
\noindent with $p_i$ ($i=1,2,3,\emptyset$) denoting a smooth point, and if the action of $\Omega$ is not mentioned, the point (smooth or singular) is fixed. \\
\newline

Let $\Pi'=\{\alpha_1,\alpha_2,\alpha_3,\alpha_4,\alpha_5,\alpha_6\}$ be a set of simple roots of the root system of type $D_6$ (cf. Table~\ref{Table} for the numbering). All the sub-root systems $\phi'$ of $D_6$ containing $\Theta=\{\alpha_1,\alpha_3,\alpha_5\}$, as well as their realizations are given in the following table : 
\medskip

\begin{center}
\renewcommand{\arraystretch}{1.5}
\begin{tabular}{|c|>{\centering\arraybackslash}p{12cm}|}
\hline
Type of $\phi'$ & Realizations of $\phi'$ \\
\hhline{|=|=|}
$A_1+A_1+A_1$ &  $<\alpha_1,\alpha_3,\alpha_5>$ \\
\hline
 \multirow{2}{*}{$A_1+A_1+A_1+A_1$}  & $<\alpha_1,\alpha_3,\alpha_5,\alpha_6>$, $<\alpha_1,\alpha_3,\alpha_5,\alpha_1+2\alpha_2+2\alpha_3+2\alpha_4+\alpha_5+\alpha_6>$, $<\alpha_1,\alpha_3,\alpha_5,\alpha_3+2\alpha_4+\alpha_5+\alpha_6>$ \\
\hline
  \multirow{2}{*}{$A_3+A_1$} &  $<\alpha_1,\alpha_3,\alpha_5,\alpha_2>$, $<\alpha_1,\alpha_3,\alpha_5,\alpha_4>$, $<\alpha_1,\alpha_3,\alpha_5,\alpha_2+\alpha_3+\alpha_4>$, $<\alpha_1,\alpha_3,\alpha_5,\alpha_2+\alpha_3+\alpha_4+\alpha_6>$, $<\alpha_1,\alpha_3,\alpha_5,\alpha_4+\alpha_6>$, $<\alpha_1,\alpha_3,\alpha_5,\alpha_2+\alpha_3+2\alpha_4+\alpha_5+\alpha_6>$ \\
\hline
\multirow{4}{*}{$A_3+A_1+A_1$} & $<\alpha_1,\alpha_3,\alpha_5,\alpha_2,\alpha_6>$, $<\alpha_1,\alpha_3,\alpha_5,\alpha_4,\alpha_1+2\alpha_2+2\alpha_3+2\alpha_4+\alpha_5+\alpha_6>$, $<\alpha_1,\alpha_3,\alpha_5,\alpha_2+\alpha_3+\alpha_4,\alpha_3+2\alpha_4+\alpha_5+\alpha_6>$, $<\alpha_1,\alpha_3,\alpha_5,\alpha_2+\alpha_3+\alpha_4+\alpha_6,\alpha_3+2\alpha_4+\alpha_5+\alpha_6>$, $<\alpha_1,\alpha_3,\alpha_5,\alpha_4+\alpha_6,\alpha_1+2\alpha_2+2\alpha_3+2\alpha_4+\alpha_5+\alpha_6>$, $<\alpha_1,\alpha_3,\alpha_5,\alpha_2+\alpha_3+2\alpha_4+\alpha_5+\alpha_6,\alpha_6>$ \\
\hline
$D_4+A_1$ & $<\alpha_1,\alpha_3,\alpha_5, \alpha_2, \alpha_3+2\alpha_4+\alpha_5+\alpha_6>$, $<\alpha_1,\alpha_3,\alpha_5, \alpha_4, \alpha_6>$, $<\alpha_1,\alpha_3,\alpha_5, \alpha_2+\alpha_3+\alpha_4, \alpha_6>$\\
\hline
$A_5$  & $<\alpha_1,\alpha_3,\alpha_5,\alpha_2,\alpha_4>$, $<\alpha_1,\alpha_3,\alpha_5,\alpha_2+\alpha_3+\alpha_4+\alpha_6,\alpha_4>$, $<\alpha_1,\alpha_3,\alpha_5,\alpha_2,\alpha_4+\alpha_6>$, $<\alpha_1,\alpha_3,\alpha_5,\alpha_4+\alpha_6,\alpha_2+\alpha_3+\alpha_4>$ \\
\hline
$D_6$  & $<\alpha_1,\alpha_2,\alpha_3,\alpha_4,\alpha_5,\alpha_6>$ \\
\hline
\end{tabular}
 \captionof{table}{}
  \label{realizationsD4andZ/2Z}
\end{center}
\medskip

Set $(e_1,e_2,e_3,e_4,e_5,e_6)$ an orthonormal basis of $\mathfrak{h}'$ of type $D_6$, and define $(\epsilon_1,\epsilon_2,\epsilon_3,\epsilon_4,\epsilon_5,\epsilon_6)$ its dual basis. Following \cite{Bou68}, the simple roots of $D_5$ are 
\medskip

\begin{center}
\renewcommand{\arraystretch}{1.2}
\raisebox{-.55\height}{$\left\lbrace\begin{array}{ccl}
\alpha_1 & = & \epsilon_1-\epsilon_2, \\
\alpha_2 & = & \epsilon_2-\epsilon_3, \\
\alpha_3 & = & \epsilon_3-\epsilon_4, \\
\alpha_4 & = & \epsilon_4-\epsilon_5, \\
\alpha_5 & = & \epsilon_5-\epsilon_6, \\
\alpha_6 & = & \epsilon_5+\epsilon_6.
\end{array}\right.$}
\end{center}
\medskip

Let $h'=(\xi_1,\xi_2,\xi_3,\xi_4,\xi_5,\xi_6) \in \mathfrak{h}'$, with the coordinates being expressed in the basis $(e_1,e_2,e_3,e_4,e_5,e_6)$. Then $h' \in H_{\alpha_1} \bigcap H_{\alpha_3} \bigcap H_{\alpha_5}$ if and only if $\xi_1=\xi_2$, $\xi_3=\xi_4$ and $\xi_5=\xi_6$. \\
\newline

Using \cite{SYS80}, one can compute the flat coordinates of $D_6$ restricted to the subspace $H_{\alpha_1} \bigcap H_{\alpha_3} \bigcap H_{\alpha_5}  \ \mathlarger{\mathlarger{\subset}} \  \mathfrak{h}'$ and obtain
\medskip

\begin{center}
\renewcommand{\arraystretch}{1.2}
\raisebox{-.55\height}{$\left\lbrace\begin{array}{ccl}
\psi_2(\xi_1,\xi_3,\xi_5) & = & 2(\xi_1^2+\xi_3^2+\xi_5^2), \\
\psi_4(\xi_1,\xi_3,\xi_5) & = & -\frac{2}{5}\xi_1^4-\frac{2}{5}\xi_3^4-\frac{2}{5}\xi_5^4+\frac{6}{5}\xi_1^2\xi_3^2+\frac{6}{5}\xi_1^2\xi_5^2+\frac{6}{5}\xi_3^2\xi_5^2, \\
\psi_6(\xi_1,\xi_3,\xi_5) & = & 2 \xi_1^2\xi_3^2\xi_5^2, \\
\psi_8(\xi_1,\xi_3,\xi_5) & = & \begin{array}[t]{l} \frac{4}{125}\xi_1^8+ \frac{4}{125}\xi_3^8+ \frac{4}{125}\xi_5^8- \frac{14}{125}\xi_1^6\xi_3^2-\frac{14}{125}\xi_1^6\xi_5^2-\frac{14}{125}\xi_1^2\xi_3^6-\frac{14}{125}\xi_1^2\xi_5^6 \\ -\frac{14}{125}\xi_3^6\xi_5^2-\frac{14}{125}\xi_3^2\xi_5^6+\frac{14}{125}\xi_1^4\xi_3^4+\frac{14}{125}\xi_1^4\xi_5^4+\frac{14}{125}\xi_3^4\xi_5^4+\frac{98}{125}\xi_1^4\xi_3^2\xi_5^2 \\ +\frac{98}{125}\xi_1^2\xi_3^2\xi_5^4+\frac{98}{125}\xi_1^2\xi_3^4\xi_5^2,  \end{array} \\
\psi_{10}(\xi_1,\xi_3,\xi_5) & = & \begin{array}[t]{l} -\frac{108}{625}\xi_1^2\xi_3^2\xi_5^6-\frac{108}{625}\xi_1^6\xi_3^2\xi_5^2-\frac{108}{625}\xi_1^2\xi_3^6\xi_5^2-\frac{22}{3125}\xi_1^{10}-\frac{22}{3125}\xi_3^{10}-\frac{22}{3125}\xi_5^{10} \\ -\frac{24}{625}\xi_1^6\xi_3^4-\frac{24}{625}\xi_1^6\xi_5^4-\frac{24}{625}\xi_1^4\xi_3^6-\frac{24}{625}\xi_1^4\xi_5^6-\frac{24}{625}\xi_3^6\xi_5^4-\frac{24}{625}\xi_3^4\xi_5^6 \\ +\frac{18}{625}\xi_1^8\xi_3^2+\frac{18}{625}\xi_1^8\xi_5^2+\frac{18}{625}\xi_1^2\xi_3^8+\frac{18}{625}\xi_1^2\xi_5^8+\frac{18}{625}\xi_3^8\xi_5^2+\frac{18}{625}\xi_3^2\xi_5^8 \\ +\frac{648}{625}\xi_1^4\xi_3^2\xi_5^4+\frac{648}{625}\xi_1^2\xi_3^4\xi_5^4+\frac{648}{625}\xi_1^4\xi_3^4\xi_5^2,  \end{array} \\
\psi(\xi_1,\xi_3,\xi_5) & = & \xi_1^2\xi_3^2\xi_5^2. \\
\end{array}\right.$}
\end{center}
\medskip

\noindent We can verify that \renewcommand{\arraystretch}{1.4}
\raisebox{-.55\height}{$\left\lbrace\begin{array}{ccl} 
\psi_8 & = & \frac{1}{5}\psi_2\psi_6-\frac{1}{100}\psi_2^2\psi_4+\frac{1}{10}\psi_4^2, \\
\psi_{10} & = & -\frac{1}{50000}\psi_2^5+\frac{1}{50}\psi_2^2\psi_6-\frac{1}{50}\psi_2\psi_4^2+\frac{2}{5}\psi_4\psi_6,\\
\psi & = & \frac{1}{2}\psi_6. \\
\end{array}\right.$}

Define 
\begin{center}
 
$\begin{array}[t]{rccc}
f :  & (\mathfrak{h}/W)^\Omega & \rightarrow & \mathfrak{h}'/W' \\
   & \begin{pmatrix}
   t_2  \\
   t_4 \\
   t_6 \\
   0 
   \end{pmatrix} & \mapsto & \renewcommand{\arraystretch}{1.4}\begin{pmatrix}
   t_2  \\
   \frac{1}{2}t_4+\frac{1}{40}t_2^2 \\
  \frac{1}{4}t_6+\frac{1}{24}t_2t_4+\frac{1}{432}t_2^3 \\
   \frac{1}{20}t_2t_6+\frac{7}{1200}t_2^2t_4+\frac{119}{432000}t_2^4+\frac{1}{40}t_4^2    \\
   \frac{133}{3600000}t_2^5+\frac{3}{400}t_2^2t_6+\frac{131}{108000}t_2^3t_4+\frac{1}{300}t_2t_4^2+\frac{1}{20}t_4t_6 \\
   \frac{1}{8}t_6+\frac{1}{48}t_2t_4+\frac{1}{864}t_2^3
   \end{pmatrix}
\end{array}$
\end{center}
It is clear that $f$ is an injective morphism, and if $(\psi_2,\psi_4,\psi_6,\psi_8,\psi_{10},\psi) \in \pi'(H_{\alpha_1} \bigcap H_{\alpha_3}\bigcap H_{\alpha_5})$, then by setting 
\begin{center}
\renewcommand{\arraystretch}{1.2}
\raisebox{-.55\height}{$\left\lbrace\begin{array}{l}
t_2=\psi_2, \\
t_4=2\psi_4-\frac{1}{20}\psi_2^2, \\
t_6=4\psi_6-\frac{1}{3}\psi_2\psi_4-\frac{1}{1080}\psi_2^3,
\end{array}\right.$}
\end{center}
it follows that $f(t_2,t_4,t_6,0)=(\psi_2,\psi_4,\psi_6,\psi_8,\psi_{10},\psi)$. Hence
\begin{center}
 \begin{tikzpicture}[scale=1,  transform shape]
\node (1) at ( 0,0)  {$f :  (\mathfrak{h}/W)^\Omega$};
\node (2) at ( 5,0)  {$\pi'(H_{\alpha_1} \bigcap H_{\alpha_3}\bigcap H_{\alpha_5})  \ \mathlarger{\mathlarger{\subset}} \  \mathfrak{h}'/W'.$};
\node (3) at ( 1.8,0.2)  {$\cong$};
\draw  [decoration={markings,mark=at position 1 with
    {\arrow[scale=1.2,>=stealth]{>}}},postaction={decorate}] (1)  --  (2);
\end{tikzpicture}
\end{center}

Let $\phi'$ be a sub-root system in Table~\ref{realizationsD4andZ/2Z}, and set $h' \in \bigcap_{\alpha \in \phi'}H_\alpha$. With the formulas just given, one can compute $(t_2,t_4,t_6,0)=f^{-1}(\pi'(h'))$, and verify the following correspondence : 
\begin{center}
\renewcommand{\arraystretch}{1.5}
\begin{tabular}[h]{lcl}
$\phi'$ of type $A_1+A_1+A_1$ &  \begin{tikzpicture}[scale=1,  transform shape]
\node (1) at ( 0,0)  {};
\node (2) at ( 2,0)  {};
\draw  [decoration={markings,mark=at position 1 with
    {\arrow[scale=1.2,>=stealth]{>}}},postaction={decorate}] (1)  --  (2);
\draw  [decoration={markings,mark=at position 1 with
    {\arrow[scale=1.2,>=stealth]{>}}},postaction={decorate}] (2)  --  (1);
\end{tikzpicture} & $(t_2,t_4,t_6,0)$ generic, \\
$\phi'$ of type $A_1+A_1+A_1+A_1$ &  \begin{tikzpicture}[scale=1,  transform shape]
\node (1) at ( 0,0)  {};
\node (2) at ( 2,0)  {};
\draw  [decoration={markings,mark=at position 1 with
    {\arrow[scale=1.2,>=stealth]{>}}},postaction={decorate}] (1)  --  (2);
\draw  [decoration={markings,mark=at position 1 with
    {\arrow[scale=1.2,>=stealth]{>}}},postaction={decorate}] (2)  --  (1);
\end{tikzpicture} & $(t_2,t_4,t_6,0) \in L \backslash \{t_4=-\frac{t_2^2}{4}\}$, \\
$\phi'$ of type $A_3+A_1$ &  \begin{tikzpicture}[scale=1,  transform shape]
\node (1) at ( 0,0)  {};
\node (2) at ( 2,0)  {};
\draw  [decoration={markings,mark=at position 1 with
    {\arrow[scale=1.2,>=stealth]{>}}},postaction={decorate}] (1)  --  (2);
\draw  [decoration={markings,mark=at position 1 with
    {\arrow[scale=1.2,>=stealth]{>}}},postaction={decorate}] (2)  --  (1);
\end{tikzpicture} & $(t_2,t_4,t_6,0) \in H \backslash ((L \bigcap \{t_4 = -\frac{t_2^2}{4}\}) \bigcup \{t_4 = \frac{t_2^2}{12}\})$, \\
$\phi'$ of type $A_3+A_1+A_1$ &  \begin{tikzpicture}[scale=1,  transform shape]
\node (1) at ( 0,0)  {};
\node (2) at ( 2,0)  {};
\draw  [decoration={markings,mark=at position 1 with
    {\arrow[scale=1.2,>=stealth]{>}}},postaction={decorate}] (1)  --  (2);
\draw  [decoration={markings,mark=at position 1 with
    {\arrow[scale=1.2,>=stealth]{>}}},postaction={decorate}] (2)  --  (1);
\end{tikzpicture} & $(t_2,t_4,t_6,0) \in (L \bigcap H) \backslash \{t_4 = -\frac{t_2^2}{4} \}$, \\
$\phi'$ of type $D_4+A_1$ &  \begin{tikzpicture}[scale=1,  transform shape]
\node (1) at ( 0,0)  {};
\node (2) at ( 2,0)  {};
\draw  [decoration={markings,mark=at position 1 with
    {\arrow[scale=1.2,>=stealth]{>}}},postaction={decorate}] (1)  --  (2);
\draw  [decoration={markings,mark=at position 1 with
    {\arrow[scale=1.2,>=stealth]{>}}},postaction={decorate}] (2)  --  (1);
\end{tikzpicture} &  $(t_2,t_4,t_6,0) \in L \bigcap  \{t_4 = -\frac{t_2^2}{4} \neq 0\}$, \\
$\phi'$ of type $A_5$ &  \begin{tikzpicture}[scale=1,  transform shape]
\node (1) at ( 0,0)  {};
\node (2) at ( 2,0)  {};
\draw  [decoration={markings,mark=at position 1 with
    {\arrow[scale=1.2,>=stealth]{>}}},postaction={decorate}] (1)  --  (2);
\draw  [decoration={markings,mark=at position 1 with
    {\arrow[scale=1.2,>=stealth]{>}}},postaction={decorate}] (2)  --  (1);
\end{tikzpicture} &   $(t_2,t_4,t_6,0) \in H \bigcap  \{t_4 = \frac{t_2^2}{12} \neq 0\}$, \\
$\phi'$ of type $D_6$ &  \begin{tikzpicture}[scale=1,  transform shape]
\node (1) at ( 0,0)  {};
\node (2) at ( 2,0)  {};
\draw  [decoration={markings,mark=at position 1 with
    {\arrow[scale=1.2,>=stealth]{>}}},postaction={decorate}] (1)  --  (2);
\draw  [decoration={markings,mark=at position 1 with
    {\arrow[scale=1.2,>=stealth]{>}}},postaction={decorate}] (2)  --  (1);
\end{tikzpicture} &  $(t_2,t_4,t_6,0)=(0,0,0,0)$. \\
\end{tabular}
\end{center}
With Table~\ref{TypesofsingularD4andZ/2Z}, we see that the singular configuration in the fiber of $\overline{\alpha^\Omega}$ above $f^{-1}(\pi'(h'))$ is of the same type as $\phi'$. Therefore for the type $D_4-C_3-D_6$, the map $f$ realizes a bijection between the singular configurations of the fibers of $\overline{\alpha^\Omega}$ and the sub-root systems of $D_5$ containing $\Theta=\{\alpha_1,\alpha_3,\alpha_5\}$.
\medskip

\subsection{Case $D_4-G_2-E_6$}

In this section $\Gamma=\mathcal{D}_2$, $\Gamma'=\mathcal{T}$ and $\Omega=\mathbb{Z}/3\mathbb{Z}=<\rho>$. The restriction $\alpha^\Omega$ of the semiuniversal deformation $\alpha$ of a simple singularity of type $D_4$  is
\begin{center}
\begin{tikzpicture}[scale=1]
\node (1) at ( 0,0) {$X_{\Gamma, \Omega}=\{((x,y,z),(t_2,0,t_6,0)) \in \cc^3 \times  \mathfrak{h}/W \ | \ z^2=xy(x+y)-\frac{t_2}{2}xy+\frac{1}{4}(t_6+\frac{t_2^3}{108})\}$};
\node (2) at ( 0,-2) {$(\mathfrak{h}/W)^\Omega=\{ (t_2,0,t_6,0) \in  \mathfrak{h}/W\}$};
\node (3) at ( 0.5, -1)  {$\alpha^\Omega$};

\draw  [decoration={markings,mark=at position 1 with
    {\arrow[scale=1.2,>=stealth]{>}}},postaction={decorate}] (1)  --  (2);
\end{tikzpicture}
\end{center}
The action of $\Omega$ on a fiber is $\rho.(x,y,z)=(y,-x-y+\frac{t_2}{2},z)$, and so the quotient $\overline{\alpha^\Omega}$ is given by
\begin{center}
\begin{tikzpicture}[scale=1]
\node (1) at ( 0,0) {$ X_{\Gamma, \Omega}/\Omega=\{((X,Y,Z),(t_2,0,t_6,0)) \in \cc^3 \times \mathfrak{h}/W \ | \ 11664X^4-Y^3-Z^2-324t_2X^2Y$};
\node (2) at ( 0,-0.5) {$-(189t_2^3+5832t_6)X^2+(81t_2t_6+\frac{15}{16}t_2^4)Y+\frac{11}{32}t_2^6+\frac{189}{4}t_2^3t_6+729t_6^2=0\}$};
\node (3) at ( 0,-2.5) {$(\mathfrak{h}/W)^\Omega=\{ (t_2,0,t_6,0) \in  \mathfrak{h}/W\}$};
\node (4) at ( 0.5, -1.5)  {$\overline{\alpha^\Omega}$};

\draw  [decoration={markings,mark=at position 1 with
    {\arrow[scale=1.2,>=stealth]{>}}},postaction={decorate}] (2)  --  (3);
\end{tikzpicture}
\end{center}
Before quotient, the discriminant of $\alpha^\Omega$ is given by $\{(t_6+\frac{t_2^3}{108})(t_6-\frac{t_2^3}{108})=0\}$ , and after quotient, it is all of $(\mathfrak{h}/W)^\Omega$ because Theorem~\ref{ThmSingular}. By studying explicitly the singularities in the fibers $(\alpha^\Omega)^{-1}(t_2,0,t_6,0)$ and $(\overline{\alpha^\Omega})^{-1}(t_2,0,t_6,0)$, we obtain the following types :
\medskip

\begin{center}
\renewcommand{\arraystretch}{1.5}
\begin{tabular}{|c|c|c|}
\hline
$(t_2,0,t_6,0)$  & Configuration of $(\alpha^\Omega)^{-1}(t_2,0,t_6,0)$ & $(\overline{\alpha^\Omega})^{-1}(t_2,0,t_6,0)$ \\
\hhline{|=|=|=|}
Generic point & $p_1+p_2$ & $A_2+A_2$ \\
\hline
\multirow{-4}{*}{$t_6=-\frac{t_2^3}{108} \neq 0$} & \begin{tikzpicture}[scale=0.8]
\node (1) at ( 4,0)  {$p_1$};
\node (2) at ( 6,0)  {$p_2$};
\node (3) at ( 3, 0) {$+$};
\node (4) at ( 5, 0) {$+$};
\node (5) at (  - 2,0)  {$A_1$};
\node (6) at ( 0,0)  {$A_1$};
\node (7) at ( 2,0) {$A_1$};
\node (8) at (  0.2, - 0.25) {};
\node (9) at ( 1.8, - 0.25) {};
\node (10) at (  -0.2, - 0.25) {};
\node (11) at ( -1.8, - 0.25) {};
\node (12) at ( 2, 0.35) {};
\node (13) at ( -2, 0.35) {};
\node (14) at ( 1, - 1) {$\Omega$};
\node (15) at ( -1, - 1) {$\Omega$};
\node (16) at ( 0, 1.5) {$\Omega$};
\node (17) at ( -1, 0) {$+$};
\node (18) at ( 1, 0) {$+$};
\draw[->] (8.west) to[out= - 70, in= - 110] (9.east);
\draw[->] (11.west) to[out= - 70, in= - 110] (10.east);
\draw[->] (12.west) to[out=  130, in=  50] (13.east);
\end{tikzpicture} &  \multirow{-4}{*}{$A_2+A_2+A_1$} \\
\hline
$t_6=\frac{t_2^3}{108} \neq 0$ & $A_1$ & $A_5$ \\
\hline
$t_2=t_6= 0$ & $D_4$ & $E_6$ \\
\hline
\end{tabular}
 \captionof{table}{}
  \label{TypesofsingularD4andZ/3Z}
\end{center}
\medskip

\noindent with $p_i$ ($i=1,2$) denoting a smooth point, and if the action of $\Omega$ is not mentioned, the point (smooth or singular) is fixed. \\
\newline

Let $\Pi'=\{\alpha_1,\alpha_2,\alpha_3,\alpha_4,\alpha_5,\alpha_6\}$ be a set of simple roots of the root system of type $E_6$. All the sub-root systems $\phi'$ of $E_6$ containing $\Theta=\{\alpha_1,\alpha_3,\alpha_5,\alpha_6\}$, as well as their realizations are given in the following table : 
\medskip

\begin{center}
\renewcommand{\arraystretch}{1.5}
\begin{tabular}{|c|>{\centering\arraybackslash}p{12cm}|}
\hline
Type of $\phi'$ & Realizations of $\phi'$ \\
\hhline{|=|=|}
$A_2+A_2$ &  $<\alpha_1,\alpha_3,\alpha_5,\alpha_6>$ \\
\hline
 \multirow{2}{*}{$A_2+A_2+A_1$}  & $<\alpha_1,\alpha_3,\alpha_5,\alpha_6, \alpha_2>$, $<\alpha_1,\alpha_3,\alpha_5,\alpha_6, \alpha_1+\alpha_2+2\alpha_3+3\alpha_4+2\alpha_5+\alpha_6>$, $<\alpha_1,\alpha_3,\alpha_5,\alpha_6,  \alpha_1+2\alpha_2+2\alpha_3+3\alpha_4+2\alpha_5+\alpha_6>$  \\
\hline
  \multirow{2}{*}{$A_5$} &  $<\alpha_1,\alpha_3,\alpha_5,\alpha_6, \alpha_4>$, $<\alpha_1,\alpha_3,\alpha_5,\alpha_6,\alpha_2+\alpha_4>$, $<\alpha_1,\alpha_3,\alpha_5,\alpha_6,\alpha_2+\alpha_3+2\alpha_4+\alpha_5>$ \\
\hline
$E_6$  & $<\alpha_1,\alpha_2,\alpha_3,\alpha_4,\alpha_5,\alpha_6>$ \\
\hline
\end{tabular}
 \captionof{table}{}
  \label{realizationsD4andZ/3Z}
\end{center}
\medskip

Set $h'=(\xi_1,\xi_2,\xi_3,\xi_4,\xi_5,\xi_6) \in \mathfrak{h}'$, with the coordinates being expressed in the basis of fundamental coweights $(\omega^\vee_1,\omega^\vee_2,\omega^\vee_3,\omega^\vee_4,\omega^\vee_5,\omega^\vee_6)$. Then $h' \in H_{\alpha_1} \bigcap H_{\alpha_3} \bigcap H_{\alpha_5}  \bigcap H_{\alpha_6}$ if and only if $\xi_1=\xi_3=\xi_5=\xi_6=0$. \\
\newline

Using \cite{Frame51} and \cite{SYS80}, we compute the flat coordinates of $E_6$ restricted to the subspace \\
$H_{\alpha_1} \bigcap H_{\alpha_3} \bigcap H_{\alpha_5} \bigcap H_{\alpha_6}   \ \mathlarger{\mathlarger{\subset}} \  \mathfrak{h}'$ and obtain
\medskip

\begin{center}
\renewcommand{\arraystretch}{1.2}
\raisebox{-.55\height}{$\left\lbrace\begin{array}{ccl}
\psi_2(\xi_2,\xi_4) & = & 2\xi_2^2+6\xi_2\xi_4+6\xi_4^2, \\
\psi_5(\xi_2,\xi_4) & = & 0, \\
\psi_6(\xi_2,\xi_4) & = & -\xi_2^6-9\xi_2^5\xi_4-30\xi_2^4\xi_4^2-45\xi_2^3\xi_4^3-30\xi_2^2\xi_4^4-9\xi_2\xi_4^5-3\xi_4^6, \\
\psi_8(\xi_2,\xi_4) & = & \frac{1}{12}(\xi_2^2+3\xi_2\xi_4+3\xi_4^2)(5\xi_2^6+45\xi_2^5\xi_4+144\xi_2^4\xi_4^2+189\xi_2^3\xi_4^3+72\xi_2^2\xi_4^4-27\xi_2\xi_4^5-9\xi_4^6),  \\
\psi_9(\xi_2,\xi_4) & = & 0, \\
\psi_{12}(\xi_2,\xi_4) & = & \begin{array}[t]{l} \frac{693}{4}\xi_2^2\xi_4^{10}+\frac{189}{2}\xi_2\xi_4^{11}-\frac{2277}{2}\xi_2^5\xi_4^7-\frac{1947}{2}\xi_2^7\xi_4^5-\frac{9}{2}\xi_2^{11}\xi_4-\frac{143}{4}\xi_2^{10}\xi_4^2-165\xi_2^9\xi_4^3 \\ -\frac{1089}{2}\xi_2^4\xi_4^8-\frac{5225}{4}\xi_2^6\xi_4^6+\frac{63}{4}\xi_4^{12}-\frac{1}{4}\xi_2^{12}-\frac{979}{2}\xi_2^8\xi_4^4. \end{array}\\
\end{array}\right.$}
\end{center}
\medskip

\noindent There are the following relations among the flat coordinates : 
\begin{center}
 \renewcommand{\arraystretch}{1.4}
\raisebox{-.55\height}{$\left\lbrace\begin{array}{ccl} 
\psi_8 & = & -\frac{1}{192}\psi_2^4-\frac{1}{4}\psi_2\psi_6, \\
\psi_{12} & = & \frac{1}{1536}\psi_2^6-\frac{1}{8}\psi_6^2+\frac{1}{48}\psi_2^3\psi_6.\\
\end{array}\right.$}
\end{center}

Define 
\begin{center}
 
$\begin{array}[t]{rccc}
f :  & (\mathfrak{h}/W)^\Omega & \rightarrow & \mathfrak{h}'/W' \\
   & \begin{pmatrix}
   t_2  \\
   0 \\
   t_6 \\
   0 
   \end{pmatrix} & \mapsto & \renewcommand{\arraystretch}{1.4}\begin{pmatrix}
   t_2  \\
  0 \\
  -6t_6-\frac{5}{72}t_2^3 \\
  \frac{7}{576}t_2^4+\frac{3}{2}t_2t_6 \\
  0 \\
  -\frac{29}{20736}t_2^6-\frac{9}{2}t_6^2-\frac{11}{48}t_6t_2^3
   \end{pmatrix}
\end{array}$
\end{center}
The map $f$ is an injective morphism, and if $(\psi_2,\psi_5,\psi_6,\psi_8,\psi_9,\psi_{12}) \in \pi'(H_{\alpha_1} \bigcap H_{\alpha_3}\bigcap H_{\alpha_5} \bigcap H_{\alpha_6})$, then by setting 
\begin{center}
\renewcommand{\arraystretch}{1.2}
\raisebox{-.55\height}{$\left\lbrace\begin{array}{l}
t_2=\psi_2, \\
t_6=-\frac{1}{6}\psi_6-\frac{5}{432}\psi_2^3,
\end{array}\right.$}
\end{center}
it follows that $f(t_2,0,t_6,0)=(\psi_2,\psi_5,\psi_6,\psi_8,\psi_9,\psi_{12})$. Hence
\begin{center}
 \begin{tikzpicture}[scale=1,  transform shape]
\node (1) at ( 0,0)  {$f :  (\mathfrak{h}/W)^\Omega$};
\node (2) at ( 5,0)  {$\pi'(H_{\alpha_1} \bigcap H_{\alpha_3}\bigcap H_{\alpha_5} \bigcap H_{\alpha_6})  \ \mathlarger{\mathlarger{\subset}} \  \mathfrak{h}'/W'.$};
\node (3) at ( 1.55,0.2)  {$\cong$};
\draw  [decoration={markings,mark=at position 1 with
    {\arrow[scale=1.2,>=stealth]{>}}},postaction={decorate}] (1)  --  (2);
\end{tikzpicture}
\end{center}

Let $\phi'$ be a sub-root system in Table~\ref{realizationsD4andZ/3Z}, and set $h' \in \bigcap_{\alpha \in \phi'}H_\alpha$. With the preceding expressions, we compute $(t_2,0,t_6,0)=f^{-1}(\pi'(h'))$, and verify the following correspondence : 
\begin{center}
\renewcommand{\arraystretch}{1.5}
\begin{tabular}[h]{lcl}
$\phi'$ of type $A_2+A_2$ &  \begin{tikzpicture}[scale=1,  transform shape]
\node (1) at ( 0,0)  {};
\node (2) at ( 2,0)  {};
\draw  [decoration={markings,mark=at position 1 with
    {\arrow[scale=1.2,>=stealth]{>}}},postaction={decorate}] (1)  --  (2);
\draw  [decoration={markings,mark=at position 1 with
    {\arrow[scale=1.2,>=stealth]{>}}},postaction={decorate}] (2)  --  (1);
\end{tikzpicture} & $(t_2,0,t_6,0)$ generic, \\
$\phi'$ of type $A_2+A_2+A_1$ &  \begin{tikzpicture}[scale=1,  transform shape]
\node (1) at ( 0,0)  {};
\node (2) at ( 2,0)  {};
\draw  [decoration={markings,mark=at position 1 with
    {\arrow[scale=1.2,>=stealth]{>}}},postaction={decorate}] (1)  --  (2);
\draw  [decoration={markings,mark=at position 1 with
    {\arrow[scale=1.2,>=stealth]{>}}},postaction={decorate}] (2)  --  (1);
\end{tikzpicture} & $(t_2,0,t_6,0) \in \{t_6=-\frac{t_2^3}{108} \neq 0\}$, \\
$\phi'$ of type $A_5$ &  \begin{tikzpicture}[scale=1,  transform shape]
\node (1) at ( 0,0)  {};
\node (2) at ( 2,0)  {};
\draw  [decoration={markings,mark=at position 1 with
    {\arrow[scale=1.2,>=stealth]{>}}},postaction={decorate}] (1)  --  (2);
\draw  [decoration={markings,mark=at position 1 with
    {\arrow[scale=1.2,>=stealth]{>}}},postaction={decorate}] (2)  --  (1);
\end{tikzpicture} & $(t_2,0,t_6,0) \in \{t_6=\frac{t_2^3}{108} \neq 0\}$, \\
$\phi'$ of type $E_6$ &  \begin{tikzpicture}[scale=1,  transform shape]
\node (1) at ( 0,0)  {};
\node (2) at ( 2,0)  {};
\draw  [decoration={markings,mark=at position 1 with
    {\arrow[scale=1.2,>=stealth]{>}}},postaction={decorate}] (1)  --  (2);
\draw  [decoration={markings,mark=at position 1 with
    {\arrow[scale=1.2,>=stealth]{>}}},postaction={decorate}] (2)  --  (1);
\end{tikzpicture} & $(t_2,0,t_6,0)=(0,0,0,0)$. \\

\end{tabular}
\end{center}
With Table~\ref{TypesofsingularD4andZ/3Z}, we see that the singular configuration in the fiber of $\overline{\alpha^\Omega}$ above $f^{-1}(\pi'(h'))$ is of the same type as $\phi'$. Hence for $D_4-G_2-E_6$, the map $f$ realizes a bijection between the singular configurations of the fibers of $\overline{\alpha^\Omega}$ and the sub-root systems of $E_6$ containing $\Theta=\{\alpha_1,\alpha_3,\alpha_5,\alpha_6\}$.
\medskip

\subsection{Case $D_4-G_2-E_7$}\label{D_4-G_2-E_7}

In this section $\Gamma=\mathcal{D}_2$, $\Gamma'=\mathcal{O}$ and $\Omega=\mathfrak{S}_3=<\rho,\sigma>$. The restriction $\alpha^\Omega$ above the fixed points is the same as the one in the preceding part, ie
\begin{center}
\begin{tikzpicture}[scale=1]
\node (1) at ( 0,0) {$X_{\Gamma, \Omega}=\{((x,y,z),(t_2,0,t_6,0) \in \cc^3 \times  \mathfrak{h}/W \ | \ z^2=xy(x+y)-\frac{t_2}{2}xy+\frac{1}{4}(t_6+\frac{t_2^3}{108})\}$};
\node (2) at ( 0,-2) {$(\mathfrak{h}/W)^\Omega=\{ (t_2,0,t_6,0) \in  \mathfrak{h}/W\}$};
\node (3) at ( 0.5, -1)  {$\alpha^\Omega$};

\draw  [decoration={markings,mark=at position 1 with
    {\arrow[scale=1.2,>=stealth]{>}}},postaction={decorate}] (1)  --  (2);
\end{tikzpicture}
\end{center}
and the action of $\Omega$ on a fiber is 
\begin{center}
$\sigma.(x,y,z)=(x,-x-y+\frac{t_2}{2},-z)$ and $\rho.(x,y,z)=(y,-x-y+\frac{t_2}{2},z)$.
\end{center}
The quotient $\overline{\alpha^\Omega}$ is then given by
\begin{center}
\begin{tikzpicture}[scale=1]
\node (1) at ( 0,0) {$ X_{\Gamma, \Omega}/\Omega=\{((X,Y,Z),(t_2,0,t_6,0) \in \cc^3 \times \mathfrak{h}/W \ | \ X^3Y-11664Y^3+Z^2+324t_2XY^2$};
\node (2) at ( 0,-0.5) {$+(189t_2^3+5832t_6)Y^2-(\frac{15}{16}t_2^4+81t_2t_6)XY-(\frac{11}{32}t_2^6+\frac{189}{4}t_2^3t_6+729t_6^2)Y=0\}$};
\node (3) at ( 0,-2.5) {$(\mathfrak{h}/W)^\Omega=\{ (t_2,0,t_6,0) \in  \mathfrak{h}/W\}$};
\node (4) at ( 0.5, -1.5)  {$\overline{\alpha^\Omega}$};

\draw  [decoration={markings,mark=at position 1 with
    {\arrow[scale=1.2,>=stealth]{>}}},postaction={decorate}] (2)  --  (3);
\end{tikzpicture}
\end{center}
The discriminant of the deformation before quotient is $\{(t_6+\frac{t_2^3}{108})(t_6-\frac{t_2^3}{108})=0\}$, and after quotient it is $(\mathfrak{h}/W)^\Omega$. The study of the fibers of $\alpha^\Omega$ and $\overline{\alpha^\Omega}$ gives the following singular configurations :  
\medskip

\begin{center}
\renewcommand{\arraystretch}{1.5}
\begin{tabular}{|c|c|c|}
\hline
$(t_2,0,t_6,0)$  & Configuration of $(\alpha^\Omega)^{-1}(t_2,0,t_6,0)$ & $(\overline{\alpha^\Omega})^{-1}(t_2,0,t_6,0)$ \\
\hhline{|=|=|=|}
Generic point & smooth and no $\Omega$-fixed point & $A_2+A_1+A_1+A_1$ \\
\hline
\multirow{-4}{*}{$t_6=-\frac{t_2^3}{108} \neq 0$} & \begin{tikzpicture}[scale=0.8]
\node (1) at (  - 2,0)  {$A_1$};
\node (2) at ( 0,0)  {$A_1$};
\node (3) at ( 2,0) {$A_1$};
\node (4) at (  0.2, - 0.25) {};
\node (5) at ( 1.8, - 0.25) {};
\node (6) at (  -0.2, - 0.25) {};
\node (7) at ( -1.8, - 0.25) {};
\node (8) at ( 2, 0.35) {};
\node (9) at ( 2, 0.15) {};
\node (10) at ( -2, 0.35) {};
\node (11) at ( 0, 0.15) {};
\node (12) at ( -2.1, 0) {};
\node (13) at ( 1, - 1) {$\rho$};
\node (14) at ( -1, - 1) {$\rho$};
\node (15) at ( 0, 1.5) {$\rho$};
\node (16) at ( -3.2, 0) {$\sigma$};
\node (17) at ( 1, 0.7) {$\sigma$};
\node (18) at ( -1, 0) {$+$};
\node (19) at ( 1, 0) {$+$};
\draw[->] (4.west) to[out= - 70, in= - 110] (5.east);
\draw[->] (7.west) to[out= - 70, in= - 110] (6.east);
\draw[->] (8.west) to[out=  130, in=  50] (10.east);
\draw[<->] (9.west) to[out=  130, in=  50] (11.east);
\draw[->] (12) to[out=  135, in=  225, looseness=10] (12);
\end{tikzpicture} &  \multirow{-4}{*}{$A_3+A_2+A_1$} \\
\hline
$t_6=\frac{t_2^3}{108} \neq 0$ & $A_1$ ($\Omega$-fixed) & $D_5+A_1$ \\
\hline
$t_2=t_6=0$ & $D_4$ ($\Omega$-fixed) & $E_7$ \\
\hline
\end{tabular}
 \captionof{table}{}
  \label{TypesofsingularD4andS3}
\end{center}
\medskip

Let $\Pi'=\{\alpha_1,\alpha_2,\alpha_3,\alpha_4,\alpha_5,\alpha_6,\alpha_7\}$ be a set of simple roots of the root system of type $E_7$. All the sub-root systems $\phi'$ of $E_7$ containing $\Theta=\{\alpha_1,\alpha_2,\alpha_3,\alpha_5,\alpha_7\}$, as well as their realizations are given in the following table : 
\medskip

\begin{center}
\renewcommand{\arraystretch}{1.5}
\begin{longtable}{|c|>{\centering\arraybackslash}p{12cm}|}
\hline
Type of $\phi'$ & Realizations of $\phi'$ \\
\hhline{|=|=|}
$A_2+A_1+A_1+A_1$ &  $<\alpha_1,\alpha_2,\alpha_3,\alpha_5,\alpha_7>$ \\
\hline
 \multirow{2}{*}{$A_3+A_2+A_1$}  & $<\alpha_1,\alpha_2,\alpha_3,\alpha_5,\alpha_7,\alpha_6>$, $<\alpha_1,\alpha_2,\alpha_3,\alpha_5,\alpha_7,\alpha_1+\alpha_2+2\alpha_3+3\alpha_4+2\alpha_5+2\alpha_6+\alpha_7>$   \\
\hline
  \multirow{2}{*}{$D_5+A_1$} &  $<\alpha_1,\alpha_2,\alpha_3,\alpha_5,\alpha_7,\alpha_4>$, $<\alpha_1,\alpha_2,\alpha_3,\alpha_5,\alpha_7,\alpha_4+\alpha_5+\alpha_6>$, $<\alpha_1,\alpha_2,\alpha_3,\alpha_5,\alpha_7,\alpha_2+\alpha_3+2\alpha_4+\alpha_5+\alpha_6>$ \\
\hline
$E_7$  & $<\alpha_1,\alpha_2,\alpha_3,\alpha_4,\alpha_5,\alpha_6,\alpha_7>$ \\
\hline
\end{longtable}
{\addtocounter{table}{-1}}
 \captionof{table}{}
  \label{realizationsD4andS3}
\end{center}
\medskip

Set $\mathfrak{h}'=\{ (\xi_1,\xi_2,\xi_3,\xi_4,\xi_5,\xi_6,\xi_7,\xi_8) \in \cc^8 \ | \ \xi_7+\xi_8=0\}$, $(e_1,...,e_8)$ the canonical basis of $\cc^8$, and $(\epsilon_1,...,\epsilon_8)$ its dual basis. Following \cite{Bou68}, the simple roots of $E_7$ are 
\medskip

\begin{center}
\renewcommand{\arraystretch}{1.2}
\raisebox{-.55\height}{$\left\lbrace\begin{array}{ccl}
\alpha_1 & = & \frac{1}{2}(\epsilon_1+\epsilon_8)- \frac{1}{2}(\epsilon_2+\epsilon_3+\epsilon_4+\epsilon_5+\epsilon_6+\epsilon_7), \\
\alpha_2 & = & \epsilon_1+\epsilon_2, \\
\alpha_3 & = & \epsilon_2-\epsilon_1, \\
\alpha_4 & = & \epsilon_3-\epsilon_2, \\
\alpha_5 & = & \epsilon_4-\epsilon_3, \\
\alpha_6 & = & \epsilon_5-\epsilon_4, \\
\alpha_7 & = & \epsilon_6-\epsilon_5.
\end{array}\right.$}
\end{center}
\medskip

Set $h'=(\xi_1,\xi_2,\xi_3,\xi_4,\xi_5,\xi_6,\xi_7) \in \mathfrak{h}'$ (we write here $(\xi_1,\xi_2,\xi_3,\xi_4,\xi_5,\xi_6,\xi_7)$ for $(\xi_1,\xi_2,\xi_3,\xi_4,\xi_5,\xi_6,\xi_7,-\xi_7)$). Then $h' \in H_{\alpha_1} \bigcap H_{\alpha_2} \bigcap H_{\alpha_3} \bigcap H_{\alpha_5}  \bigcap H_{\alpha_7}$ if and only if $\xi_1=\xi_2=0$, $\xi_3=\xi_4$, $\xi_5=\xi_6$ and $\xi_7=-\xi_3-\xi_5$. \\
\newline

Using \cite{Abria09} and \cite{Mehta88}, one can compute the flat coordinates of $E_7$ restricted to the subspace \\
$H_{\alpha_1}  \bigcap H_{\alpha_2} \bigcap H_{\alpha_3} \bigcap H_{\alpha_5} \bigcap H_{\alpha_7}   \ \mathlarger{\mathlarger{\subset}} \  \mathfrak{h}'$. We will therefore look at the flat coordinates as functions of $\xi_3$ and $\xi_5$. The expressions are
\medskip

%\begin{array}[t]{l}

\begin{center}
\renewcommand{\arraystretch}{1.2}
\raisebox{-.55\height}{$\left\lbrace\begin{array}{ccl}
\psi_2(\xi_3,\xi_5) & = & \frac{2}{5}(\xi_3^2+\xi_3\xi_5+\xi_5^2), \\
\psi_6(\xi_3,\xi_5) & = & \frac{32176}{225}\xi_3^6+\frac{32176}{75}\xi_3^5\xi_5+\frac{53552}{75}\xi_3^4\xi_5^2+\frac{160432}{225}\xi_3^3\xi_5^3+\frac{53552}{75}\xi_3^2\xi_5^4+\frac{32176}{75}\xi_3\xi_5^5+\frac{32176}{225}\xi_5^6, \\
\psi_8(\xi_3,\xi_5) & = &\begin{array}[t]{l} \frac{16}{30375}(\xi_3^2+\xi_3\xi_5+\xi_5^2)(550819\xi_3^6+1652457\xi_3^5\xi_5+1389264\xi_3^4\xi_5^2+24433\xi_3^3\xi_5^3 \\ +1389264\xi_3^2\xi_5^4 +1652457\xi_3\xi_5^5+550819\xi_5^6), \end{array} \\
\psi_{10}(\xi_3,\xi_5) & = & \begin{array}[t]{l} \frac{96}{109375}(20743\xi_3^6+62229\xi_3^5\xi_5+41208\xi_3^4\xi_5^2-21299\xi_3^3\xi_5^3+41208\xi_3^2\xi_5^4+62229\xi_3\xi_5^5 \\ +20743\xi_5^6)(\xi_3^2  +\xi_3\xi_5+\xi_5^2)^2, \end{array}  \\
\psi_{12}(\xi_3,\xi_5) & = & \begin{array}[t]{l} -\frac{42062501701}{398671875}\xi_3^8\xi_5^4-\frac{4423023418}{102515625}\xi_3^9\xi_5^3-\frac{62899959716}{5980078125}\xi_3^{10}\xi_5^2-\frac{347826674932}{1993359375}\xi_3^7\xi_5^5 \\ -\frac{3081278138}{17940234375}\xi_3^{12}  -\frac{3081278138}{17940234375}\xi_5^{12}-\frac{175228928248}{854296875}\xi_3^6\xi_5^6-\frac{6162556276}{5980078125}\xi_3^{11}\xi_5 \\ -\frac{347826674932}{1993359375}\xi_3^5\xi_5^7  -\frac{42062501701}{398671875}\xi_3^4\xi_5^8  -\frac{4423023418}{102515625}\xi_3^3\xi_5^9-\frac{62899959716}{5980078125}\xi_3^2\xi_5^{10} \\ -\frac{6162556276}{5980078125}\xi_3\xi_5^{11} , \end{array}\\
\psi_{14}(\xi_3,\xi_5) & = & \begin{array}[t]{l} -\frac{4}{30903847734375}(\xi_3^2+\xi_3\xi_5+\xi_5^2)(1511960253367\xi_3^{12}+9071761520202\xi_3^{11}\xi_5 \\ +67786465629432\xi_3^{10}\xi_5^2+255774514211975\xi_3^9\xi_5^3+617323843488330\xi_3^8\xi_5^4 \\ +1034437665403692\xi_3^7\xi_5^5+1226835303782847\xi_3^6\xi_5^6+1034437665403692\xi_3^5\xi_5^7 \\ +617323843488330\xi_3^4\xi_5^8+255774514211975\xi_3^3\xi_5^9+67786465629432\xi_3^2\xi_5^{10} \\ +9071761520202\xi_3\xi_5^{11}+1511960253367\xi_5^{12}), \end{array}\\
\end{array}\right.$}
\raisebox{-.55\height}{$\left\lbrace\begin{array}{ccl}
\psi_{18}(\xi_3,\xi_5) & = & \begin{array}[t]{l} -\frac{49900582548245699977888}{128185297421220703125}\xi_3^{18} -\frac{49900582548245699977888}{128185297421220703125}\xi_5^{18}-\frac{1808994581776446325173376}{42728432473740234375}\xi_3^3\xi_5^{15} \\ -\frac{43351951625476282697248}{2848562164916015625}\xi_3^2\xi_5^{16}-\frac{49900582548245699977888}{14242810824580078125}\xi_3\xi_5^{17}  -\frac{49900582548245699977888}{14242810824580078125}\xi_3^{17}\xi_5 \\ -\frac{43351951625476282697248}{2848562164916015625}\xi_3^{16}\xi_5^2-\frac{1808994581776446325173376}{42728432473740234375}\xi_3^{15}\xi_5^3  -\frac{1224969840491929611874048}{14242810824580078125}\xi_3^{14}\xi_5^4 \\ -\frac{283014291225008940645632}{2034687260654296875}\xi_3^{13}\xi_5^5-\frac{8202907266598286263520384}{42728432473740234375}\xi_3^{12}\xi_5^6  -\frac{3383893531113795266600128}{14242810824580078125}\xi_3^{11}\xi_5^7 \\ -\frac{349873020975151098628384}{1294800984052734375}\xi_3^{10}\xi_5^8  -\frac{3288297534494448854110112}{11653208856474609375}\xi_3^9\xi_5^9-\frac{349873020975151098628384}{1294800984052734375}\xi_3^8\xi_5^{10} \\ -\frac{3383893531113795266600128}{14242810824580078125}\xi_3^7\xi_5^{11}-\frac{8202907266598286263520384}{42728432473740234375}\xi_3^6\xi_5^{12}-\frac{283014291225008940645632}{2034687260654296875}\xi_3^5\xi_5^{13} \\ -\frac{1224969840491929611874048}{14242810824580078125}\xi_3^4\xi_5^{14}. \end{array}\\
\end{array}\right.$}
\end{center}
\medskip

\noindent We can verify that 
\begin{center}
\renewcommand{\arraystretch}{1.4}
\raisebox{-.80\height}{$\left\lbrace\begin{array}{ccl} 
\psi_8 & = & -\frac{2252645}{81}\psi_2^4+\frac{473}{27}\psi_2\psi_6, \\
\psi_{10} & = & -\frac{557383}{105}\psi_2^5+\frac{111}{35}\psi_2^2\psi_6,\\
\psi_{12} & = & -\frac{43251895481}{24494400}\psi_2^6+\frac{1079173}{1360800}\psi_2^3\psi_6-\frac{1}{103680}\psi_6^2,\\
\psi_{14} & = & -\frac{573683065303}{145496736}\psi_2^7+\frac{10112840293}{4688228160}\psi_2^4\psi_6-\frac{17821}{89299584}\psi_2\psi_6^2,\\
\psi_{18} & = & -\frac{15896711538141155833}{4023348492240}\psi_2^9+\frac{391876556269181513}{64820614597200}\psi_2^6\psi_6-\frac{7868764351687}{3601145255400}\psi_2^3\psi_6^2-\frac{5}{419904}\psi_6^3.\\
\end{array}\right.$}
\end{center}
\medskip

Define 
\begin{center}
 
$\begin{array}[t]{rccc}
f :  & (\mathfrak{h}/W)^\Omega & \rightarrow & \mathfrak{h}'/W' \\
   & \begin{pmatrix}
   t_2  \\
   0 \\
   t_6 \\
   0
   \end{pmatrix} & \mapsto & \renewcommand{\arraystretch}{1.4}\begin{pmatrix}
   t_2  \\
   \frac{18610}{9}t_2^3+18000t_6 \\
   \frac{2044595}{243}t_2^4+\frac{946000}{3}t_2t_6 \\
   \frac{6247}{5}t_2^5+\frac{399600}{7}t_2^2t_6 \\
   -\frac{877545367}{5248800}t_2^6+\frac{30746815}{2268}t_2^3t_6-3125t_6^2 \\
   -\frac{4251411945217}{12658216032}t_2^7+\frac{42570028475}{1775844}t_2^4t_6-\frac{556906250}{8613}t_2t_6^2 \\
   -\frac{134750219913739937987}{150013422353520}t_2^9-\frac{277874830706221330}{4910652621}t_2^6t_6-\frac{488086012279345000}{666878751}t_2^3t_6^2-\frac{625000000}{9}t_6^3
   \end{pmatrix}
\end{array}$
\end{center}
The morphism $f$ is injective and if $(\psi_2,\psi_6,\psi_8,\psi_{10},\psi_{12},\psi_{14},\psi_{18}) \in \pi'(H_{\alpha_1} \bigcap H_{\alpha_2} \bigcap H_{\alpha_3}\bigcap H_{\alpha_5} \bigcap H_{\alpha_7})$, then by setting
\begin{center}
 \renewcommand{\arraystretch}{1.2}
\raisebox{-.55\height}{$\left\lbrace\begin{array}{l}
t_2=\psi_2, \\
t_6=\frac{1}{18000}\psi_6-\frac{1861}{16200}\psi_2^3,
\end{array}\right.$}
\end{center}
it follows that $f(t_2,0,t_6,0)=(\psi_2,\psi_6,\psi_8,\psi_{10},\psi_{12},\psi_{14},\psi_{18})$. Hence
\begin{center}
 \begin{tikzpicture}[scale=1,  transform shape]
\node (1) at ( 0,0)  {$f :  (\mathfrak{h}/W)^\Omega$};
\node (2) at ( 6,0)  {$\pi'(H_{\alpha_1} \bigcap H_{\alpha_2} \bigcap H_{\alpha_3}\bigcap H_{\alpha_5} \bigcap H_{\alpha_7})  \ \mathlarger{\mathlarger{\subset}} \  \mathfrak{h}'/W'.$};
\node (3) at (1.8,0.2)  {$\cong$};
\draw  [decoration={markings,mark=at position 1 with
    {\arrow[scale=1.2,>=stealth]{>}}},postaction={decorate}] (1)  --  (2);
\end{tikzpicture}
\end{center}

Let $\phi'$ be a sub-root system in Table~\ref{realizationsD4andS3}, and set $h' \in \bigcap_{\alpha \in \phi'}H_\alpha$. With the preceding expressions, we compute $(t_2,0,t_6,0)=f^{-1}(\pi'(h'))$, and verify the following correspondence : 

\begin{center}
\renewcommand{\arraystretch}{1.5}
\begin{tabular}[h]{lcl}
$\phi'$ of type $A_2+A_1+A_1+A_1$ &  \begin{tikzpicture}[scale=1,  transform shape]
\node (1) at ( 0,0)  {};
\node (2) at ( 2,0)  {};
\draw  [decoration={markings,mark=at position 1 with
    {\arrow[scale=1.2,>=stealth]{>}}},postaction={decorate}] (1)  --  (2);
\draw  [decoration={markings,mark=at position 1 with
    {\arrow[scale=1.2,>=stealth]{>}}},postaction={decorate}] (2)  --  (1);
\end{tikzpicture} & $(t_2,0,t_6,0)$ generic, \\
$\phi'$ of type $A_3+A_2+A_1$ &  \begin{tikzpicture}[scale=1,  transform shape]
\node (1) at ( 0,0)  {};
\node (2) at ( 2,0)  {};
\draw  [decoration={markings,mark=at position 1 with
    {\arrow[scale=1.2,>=stealth]{>}}},postaction={decorate}] (1)  --  (2);
\draw  [decoration={markings,mark=at position 1 with
    {\arrow[scale=1.2,>=stealth]{>}}},postaction={decorate}] (2)  --  (1);
\end{tikzpicture} & $(t_2,0,t_6,0) \in \{t_6=-\frac{t_2^3}{108} \neq 0\}$, \\
$\phi'$ of type $D_5+A_1$ &  \begin{tikzpicture}[scale=1,  transform shape]
\node (1) at ( 0,0)  {};
\node (2) at ( 2,0)  {};
\draw  [decoration={markings,mark=at position 1 with
    {\arrow[scale=1.2,>=stealth]{>}}},postaction={decorate}] (1)  --  (2);
\draw  [decoration={markings,mark=at position 1 with
    {\arrow[scale=1.2,>=stealth]{>}}},postaction={decorate}] (2)  --  (1);
\end{tikzpicture} & $(t_2,0,t_6,0)  \in \{t_6=\frac{t_2^3}{108} \neq 0\}$, \\
$\phi'$ of type $E_6$ &  \begin{tikzpicture}[scale=1,  transform shape]
\node (1) at ( 0,0)  {};
\node (2) at ( 2,0)  {};
\draw  [decoration={markings,mark=at position 1 with
    {\arrow[scale=1.2,>=stealth]{>}}},postaction={decorate}] (1)  --  (2);
\draw  [decoration={markings,mark=at position 1 with
    {\arrow[scale=1.2,>=stealth]{>}}},postaction={decorate}] (2)  --  (1);
\end{tikzpicture} & $(t_2,0,t_6,0)=(0,0,0,0)$. \\

\end{tabular}
\end{center}

With Table~\ref{TypesofsingularD4andS3}, we see that the singular configuration in the fiber of $\overline{\alpha^\Omega}$ above $f^{-1}(\pi'(h'))$ is of the same type as $\phi'$. Hence for $D_4-G_2-E_7$, the map $f$ realizes a bijection between the singular configurations of the fibers of $\overline{\alpha^\Omega}$ and the sub-root systems of $E_7$ containing $\Theta=\{\alpha_1,\alpha_2,\alpha_3,\alpha_5,\alpha_7\}$.
\medskip

\subsection{Case $E_6-F_4-E_7$}

Set $\Gamma=\mathcal{T}$, $\Gamma'=\mathcal{O}$ and $\Omega=\mathbb{Z}/2\mathbb{Z}=<\sigma>$. According to \cite{Cara17}, a semiuniversal deformation $\alpha^\Omega$ of a simple singularity of type $F_4$ is given by
\begin{center}
\begin{tikzpicture}[scale=1]
\node (1) at ( 0,0) {$X_{\Gamma, \Omega}=\{((x,y,z),(t_2,0,t_6,t_8,0,t_{12}) \in \cc^3 \times  \mathfrak{h}/W \ | \ -\frac{1}{4}x^4+y^3+z^2-\frac{t_2}{4}x^2y+\frac{1}{48}(t_6-\frac{t_2^3}{8})x^2$};
\node (2) at ( 0,-0.5) {$+\frac{1}{48}(-t_8+\frac{t_6t_2}{4}-\frac{t_2^4}{192})y+\frac{1}{576}(t_{12}-\frac{t_8t_2^2}{8}-\frac{t_6^2}{8}+\frac{t_6t_2^3}{96})=0\}$};
\node (3) at ( 0,-2.5) {$(\mathfrak{h}/W)^\Omega=\{ (t_2,0,t_6,t_8,0,t_{12}) \in  \mathfrak{h}/W\}$};
\node (4) at ( 0.5, -1.5)  {$\alpha^\Omega$};

\draw  [decoration={markings,mark=at position 1 with
    {\arrow[scale=1.2,>=stealth]{>}}},postaction={decorate}] (2)  --  (3);
\end{tikzpicture}
\end{center}
Furthermore, the action of $\Omega$ on a fiber of $\alpha^\Omega$ is $\sigma.(x,y,z)=(-x,y,-z)$ and the quotient $\overline{\alpha^\Omega}$ is
\begin{center}
\begin{tikzpicture}[scale=1]
\node (1) at ( 0,0) {$X_{\Gamma, \Omega}/\Omega=\{((X,Y,Z),(t_2,0,t_6,t_8,0,t_{12})) \in \cc^3 \times  \mathfrak{h}/W \ | \ -\frac{1}{4}X^3+XY^3+Z^2-\frac{t_2}{4}X^2Y+\frac{1}{48}(t_6-\frac{t_2^3}{8})X^2$};
\node (2) at ( 0,-0.5) {$+\frac{1}{48}(-t_8+\frac{t_6t_2}{4}-\frac{t_2^4}{192})XY+\frac{1}{576}(t_{12}-\frac{t_8t_2^2}{8}-\frac{t_6^2}{8}+\frac{t_6t_2^3}{96})X=0\}$};
\node (3) at ( 0,-2.5) {$(\mathfrak{h}/W)^\Omega=\{ (t_2,0,t_6,t_8,0,t_{12}) \in  \mathfrak{h}/W\}$};
\node (4) at ( 0.5, -1.5)  {$\overline{\alpha^\Omega}$};

\draw  [decoration={markings,mark=at position 1 with
    {\arrow[scale=1.2,>=stealth]{>}}},postaction={decorate}] (2)  --  (3);
\end{tikzpicture}
\end{center}
Before quotient, the discriminant of $\alpha^\Omega$ is $\mathcal{H}_1 \bigcup \mathcal{H}_2$ with 
\begin{itemize}
\item $\mathcal{H}_1=\{t_2^{12}-144t_2^9t_6+576t_2^8t_8+5184t_2^6t_6^2-13824t_2^5t_6t_8-138240t_2^4t_8^2-69120t_2^3t_6^3-331776t_{12}t_2^3t_6+829440t_2^2t_6^2t_8+3981312t_{12}t_2^2t_8-5308416t_2t_6t_8^2-248832t_6^4+3981312t_{12}t_6^2+7077888t_8^3-15925248t_{12}^2=0\}$, \\
\item $\mathcal{H}_2=\{t_2^{12}+144t_2^9t_6+576t_2^8t_8+5184t_2^6t_6^2+13824t_2^5t_6t_8-138240t_2^4t_8^2+69120t_2^3t_6^3-331776t_{12}t_2^3t_6+829440t_2^2t_6^2t_8-3981312t_{12}t_2^2t_8+5308416t_2t_6t_8^2-248832t_6^4-3981312t_{12}t_6^2+7077888t_8^3-15925248t_{12}^2=0 \ \mathrm{and} \  t_2 \neq 0\}$.
\end{itemize}
After quotient, it is known from Theorem~\ref{ThmSingular} that all of $(\mathfrak{h}/W)^\Omega$ composes the discriminant. By studying explicitly the singularities in the fibers $(\alpha^\Omega)^{-1}(t_2,0,t_6,t_8,0,t_{12})$ and $(\overline{\alpha^\Omega})^{-1}(t_2,0,t_6,t_8,0,t_{12})$, we obtain the following types : 
\medskip

\begin{center}
\begin{adjustwidth}{-1cm}{-1cm}
\renewcommand{\arraystretch}{1.8}
\begin{longtable}{|c|c|c|}
\hline
$(t_2,0,t_6,t_8,0,t_{12})$  & Configuration of $(\alpha^\Omega)^{-1}(t_2,0,t_6,t_8,0,t_{12})$ & $(\overline{\alpha^\Omega})^{-1}(t_2,0,t_6,t_8,0,t_{12})$ \\
\hhline{|=|=|=|}
Generic point & $p_1+p_2+p_3$ & $A_1+A_1+A_1$ \\
\hline
$\mathcal{H}_1$ & $A_1+p$ & $A_3+A_1$ \\
\hline
$\mathcal{H}_1 \bigcap \{t_8=\frac{1}{96}\frac{t_2^6+96t_6^2}{t_2^2}\}$ & $A_3+p$ & $D_4+A_1$ \\
\hline
$\mathcal{H}_1 \bigcap \{t_8=\frac{1}{96}\frac{t_2^6+96t_6^2}{t_2^2}\} \bigcap \{t_2^3+8t_6=0\}$ & $A_5+p$ & $D_5+A_1$ \\
\hline
$\mathcal{H}_1 \bigcap \{t_8=\frac{1}{96}\frac{t_2^6+96t_6^2}{t_2^2}\} \bigcap \{t_2^3-8t_6=0\}$ & $D_4$ & $D_6$ \\
\hline
$\mathcal{H}_1 \bigcap \{t_8=-\frac{1}{192}t_2^4+\frac{1}{4}t_2t_6\}$ & $A_2$ & $A_5$ \\
\hline
\multirow{-2}{*}{$\mathcal{H}_2 \backslash \mathcal{H}_1$} & \begin{tikzpicture}[scale=0.8]
\node (1) at ( 0,0)  {$A_1$};
\node (2) at ( 1, 0) {$+$};
\node (3) at ( 2,0) {$A_1$};
\node (4) at ( 2.5, 0) {$+$};
\node (5) at ( 3, 0) {$p_1$};
\node (6) at ( 3.5, 0) {$+$};
\node (7) at ( 4, 0) {$p_2$};
\node (8) at ( 4.5, 0) {$+$};
\node (9) at ( 5, 0) {$p_3$};
\node (10) at ( 1.8, - 0.25) {};
\node (11) at (  0.2, - 0.25) {};
\node (12) at ( 1, -1.1) {$\Omega$};
\draw[<->] (11.west) to[out= - 70, in= - 110] (10.east);
\end{tikzpicture} & \multirow{-2}{*}{$A_1+A_1+A_1+A_1$} \\
\hline
\multirow{-2}{*}{$\mathcal{H}_2 \bigcap \{t_8 = -\frac{1}{192}t_2^4-\frac{1}{4}t_2t_6\}$} & \begin{tikzpicture}[scale=0.8]
\node (1) at ( 0,0)  {$A_2$};
\node (2) at ( 1, 0) {$+$};
\node (3) at ( 2,0) {$A_2$};
\node (4) at ( 2.5, 0) {$+$};
\node (5) at ( 3, 0) {$p_1$};
\node (6) at ( 3.5, 0) {$+$};
\node (7) at ( 4, 0) {$p_2$};
\node (8) at ( 4.5, 0) {$+$};
\node (9) at ( 5, 0) {$p_3$};
\node (10) at ( 1.8, - 0.25) {};
\node (11) at (  0.2, - 0.25) {};
\node (12) at ( 1, -1.1) {$\Omega$};
\draw[<->] (11.west) to[out= - 70, in= - 110] (10.east);
\end{tikzpicture} & \multirow{-2}{*}{$A_2+A_1+A_1+A_1$} \\
\hline
\multirow{-2}{*}{$\mathcal{H}_1 \bigcap\mathcal{H}_2$} & \begin{tikzpicture}[scale=0.8]
\node (1) at ( 0,0)  {$A_1$};
\node (2) at ( 1, 0) {$+$};
\node (3) at ( 2,0) {$A_1$};
\node (4) at ( 2.5, 0) {$+$};
\node (5) at ( 3, 0) {$A_1$};
\node (6) at ( 3.5, 0) {$+$};
\node (7) at ( 4, 0) {$p$};
\node (8) at ( 1.8, - 0.25) {};
\node (9) at (  0.2, - 0.25) {};
\node (10) at ( 1, -1.1) {$\Omega$};
\draw[<->] (9.west) to[out= - 70, in= - 110] (8.east);
\end{tikzpicture} & \multirow{-2}{*}{$A_3+A_1+A_1$} \\
\hline
\multirow{-2}{*}{$\mathcal{H}_1 \bigcap\mathcal{H}_2 \bigcap \{t_8 = -\frac{1}{192}t_2^4-\frac{1}{4}t_2t_6\}$} & \begin{tikzpicture}[scale=0.8]
\node (1) at ( 0,0)  {$A_2$};
\node (2) at ( 1, 0) {$+$};
\node (3) at ( 2,0) {$A_2$};
\node (4) at ( 2.5, 0) {$+$};
\node (5) at ( 3, 0) {$A_1$};
\node (6) at ( 3.5, 0) {$+$};
\node (7) at ( 4, 0) {$p$};
\node (8) at ( 1.8, - 0.25) {};
\node (9) at (  0.2, - 0.25) {};
\node (10) at ( 1, -1.1) {$\Omega$};
\draw[<->] (9.west) to[out= - 70, in= - 110] (8.east);
\end{tikzpicture}  & \multirow{-2}{*}{$A_3+A_2+A_1$} \\
\hline
\multirow{-2}{*}{$\mathcal{H}_1 \bigcap\mathcal{H}_2 \bigcap \{t_8 = -\frac{1}{192}t_2^4+\frac{1}{4}t_2t_6\}$} & \begin{tikzpicture}[scale=0.8]
\node (1) at ( -2,0)  {$A_2$};
\node (2) at ( -1, 0) {$+$};
\node (3) at ( 0,0)  {$A_1$};
\node (4) at ( 1, 0) {$+$};
\node (5) at ( 2,0) {$A_1$};
\node (6) at ( 1.8, - 0.25) {};
\node (7) at (  0.2, - 0.25) {};
\node (8) at ( 1, -1.1) {$\Omega$};
\draw[<->] (7.west) to[out= - 70, in= - 110] (6.east);
\end{tikzpicture}  & \multirow{-2}{*}{$A_5+A_1$} \\
\hline
$t_2=t_6=t_8=t_{12}=0$ & $E_6$ & $E_7$ \\
\hline
\end{longtable}
{\addtocounter{table}{-1}}
  \end{adjustwidth}
 \captionof{table}{}
  \label{TypesofsingularE6}
\end{center}
\medskip

\noindent with $p_i$ ($i=1,2,3, \emptyset$) denoting a smooth point, and if the action of $\Omega$ is not mentioned, the point (smooth or singular) is fixed. \\
\newline

Let $\Pi'=\{\alpha_1,\alpha_2,\alpha_3,\alpha_4,\alpha_5,\alpha_6,\alpha_7\}$ be a set of simple roots of the root system of type $E_7$. All the sub-root systems $\phi'$ of $E_7$ containing $\Theta=\{\alpha_2,\alpha_5,\alpha_7\}$, as well as their realizations are given in the following table (except in the first case $A_1+A_1+A_1$ and the last one $E_7$, we will not repeat that $\phi'$ contains $\Theta$, for the sake of readability) : 
\medskip

{\footnotesize
\begin{center}
\renewcommand{\arraystretch}{1.5}
\begin{longtable}{|c|>{\centering\arraybackslash}p{12cm}|}

\hline
Type of $\phi'$ & Realizations of $\phi'$ \\
\hhline{|=|=|}
$A_1+A_1+A_1$ &  $<\alpha_2,\alpha_5,\alpha_7>$ \\
\hline
 \multirow{2.5}{*}{$A_3+A_1$}  & $<\alpha_4>$, $<\alpha_3+\alpha_4>$, $<\alpha_1+\alpha_3+\alpha_4>$, $<\alpha_1+\alpha_2+2\alpha_3+3\alpha_4+2\alpha_5+2\alpha_6+\alpha_7>$, $<\alpha_4+\alpha_5+\alpha_6>$, $<\alpha_3+\alpha_4+\alpha_5+\alpha_6>$, $<\alpha_1+\alpha_3+\alpha_4+\alpha_5+\alpha_6>$, $<\alpha_1+\alpha_2+2\alpha_3+3\alpha_4+2\alpha_5+\alpha_6>$, $<\alpha_6>$, $<\alpha_2+\alpha_3+2\alpha_4+\alpha_5+\alpha_6>$, $<\alpha_1+\alpha_2+\alpha_3+2\alpha_4+\alpha_5+\alpha_6>$, $<\alpha_1+\alpha_2+2\alpha_3+2\alpha_4+\alpha_5+\alpha_6>$ \\
\hline
   \multirow{5.5}{*}{$D_4+A_1$} & $<\alpha_1 , \alpha_3+\alpha_4>$,  $< \alpha_1 , \alpha_3+\alpha_4+\alpha_5+\alpha_6 >$,  $< \alpha_1 ,\alpha_2+ \alpha_3+2\alpha_4+\alpha_5+\alpha_6  >$,   $<\alpha_3 ,\alpha_4>$,   $<\alpha_3 , \alpha_4+\alpha_5+\alpha_6>$,  $< \alpha_3 ,\alpha_1+\alpha_2+ \alpha_3+2\alpha_4+\alpha_5+\alpha_6  >$,  $< \alpha_1+\alpha_3, \alpha_4>$,  $< \alpha_1+\alpha_3,\alpha_4+\alpha_5+\alpha_6 >$,  $< \alpha_1+\alpha_3,\alpha_2+ \alpha_3+2\alpha_4+\alpha_5+\alpha_6 >$,  $<\alpha_1+\alpha_2+ \alpha_3+2\alpha_4+\alpha_5 , \alpha_3 + \alpha_4+\alpha_5+\alpha_6>$,  $<\alpha_1+\alpha_2+ \alpha_3+2\alpha_4+\alpha_5 ,\alpha_6 >$,  $< \alpha_1+\alpha_2+ 2\alpha_3+2\alpha_4+\alpha_5, \alpha_4+\alpha_5+\alpha_6 >$,  $< \alpha_1+\alpha_2+ 2\alpha_3+2\alpha_4+\alpha_5,\alpha_6 >$,  $<\alpha_2+ \alpha_3+2\alpha_4+2\alpha_5+2\alpha_6+\alpha_7 , \alpha_1+\alpha_3+\alpha_4 >$,  $< \alpha_1+\alpha_2+ \alpha_3+2\alpha_4+2\alpha_5+2\alpha_6+\alpha_7 , \alpha_3+\alpha_4 >$,  $<  \alpha_1+\alpha_2+ 2\alpha_3+2\alpha_4+2\alpha_5+2\alpha_6+\alpha_7 ,\alpha_4 >$,  $< \alpha_2+ \alpha_3+2\alpha_4+\alpha_5,\alpha_1+ \alpha_3+\alpha_4+\alpha_5 +\alpha_6 >$,  $<\alpha_2+ \alpha_3+2\alpha_4+\alpha_5 ,\alpha_6 >$ \\
\hline
   \multirow{4.5}{*}{$D_5+A_1$} & $<\alpha_1, \alpha_3+\alpha_4,  \alpha_2+\alpha_3+ 2\alpha_4+2\alpha_5+2\alpha_6+\alpha_7>$, $<\alpha_1, \alpha_3+\alpha_4+ \alpha_5+\alpha_6, \alpha_2+\alpha_3+ 2\alpha_4+\alpha_5>$, $<\alpha_1, \alpha_2+\alpha_3+ 2\alpha_4+\alpha_5+\alpha_6,\alpha_3>$,
  $<\alpha_3, \alpha_4,\alpha_1>$, $<\alpha_3,\alpha_4, \alpha_1+ \alpha_2+\alpha_3+ 2\alpha_4+2\alpha_5+2\alpha_6+\alpha_7>$, $<\alpha_3,\alpha_4+\alpha_5+\alpha_6,\alpha_1>$, $<\alpha_3,\alpha_4+\alpha_5+\alpha_6, \alpha_1+\alpha_2+\alpha_3+ 2\alpha_4+\alpha_5>$, 
  $<\alpha_1+\alpha_3,\alpha_4,\alpha_2+\alpha_3+ 2\alpha_4+2\alpha_5+2\alpha_6+\alpha_7>$, $<\alpha_1+\alpha_3,\alpha_4+\alpha_5+\alpha_6,\alpha_2+\alpha_3+ 2\alpha_4+\alpha_5>$, 
  $< \alpha_1+\alpha_2+ \alpha_3+2\alpha_4+\alpha_5,\alpha_6,\alpha_3>$,
  $< \alpha_2+ \alpha_3+2\alpha_4+\alpha_5,\alpha_6,\alpha_1>$, $< \alpha_2+ \alpha_3+2\alpha_4+\alpha_5,\alpha_6,\alpha_1+\alpha_3>$\\
\hline
  \multirow{4}{*}{$D_6$} & $<\alpha_4,\alpha_6,\alpha_3>$, $<\alpha_4,\alpha_6,\alpha_1+\alpha_3>$, $<\alpha_4,\alpha_6,\alpha_1+\alpha_2+ 2\alpha_3+2\alpha_4+\alpha_5 >$,
  $<\alpha_3+\alpha_4,\alpha_6,\alpha_1>$, $<\alpha_3+\alpha_4,\alpha_6,\alpha_1+\alpha_2+ \alpha_3+2\alpha_4+\alpha_5 >$,
   $<\alpha_1+\alpha_3+\alpha_4,\alpha_6,\alpha_2+ \alpha_3+2\alpha_4+\alpha_5 >$,
   $<\alpha_4,\alpha_3+\alpha_4+\alpha_5+\alpha_6,\alpha_1+\alpha_3>$, $<\alpha_4,\alpha_3+\alpha_4+\alpha_5+\alpha_6,\alpha_1>$, $<\alpha_4,\alpha_1+\alpha_3+\alpha_4+\alpha_5+\alpha_6,\alpha_3>$,
  $<\alpha_3+\alpha_4,\alpha_4+\alpha_5+\alpha_6,\alpha_1>$, $<\alpha_3+\alpha_4,\alpha_4+\alpha_5+\alpha_6,\alpha_1+\alpha_3>$,
   $<\alpha_1+\alpha_3+\alpha_4,\alpha_4+\alpha_5+\alpha_6,\alpha_3>$ \\
\hline
 \multirow{6}{*}{$A_5$} & $<\alpha_4,\alpha_6>$, $<\alpha_4,\alpha_1+\alpha_2+2\alpha_3+2\alpha_4+\alpha_5+\alpha_6>$,
   $<\alpha_3+\alpha_4,\alpha_6>$, $<\alpha_3+\alpha_4,\alpha_1+\alpha_2+\alpha_3+2\alpha_4+\alpha_5+\alpha_6>$, 
   $<\alpha_1+\alpha_3+\alpha_4,\alpha_6>$, $<\alpha_1+\alpha_3+\alpha_4,\alpha_2+\alpha_3+2\alpha_4+\alpha_5+\alpha_6>$,  
   $<\alpha_4+\alpha_5+\alpha_6,\alpha_1+\alpha_2+2\alpha_3+2\alpha_4+\alpha_5+\alpha_6>$,
   $<\alpha_3+\alpha_4+\alpha_5+\alpha_6,\alpha_1+\alpha_2+\alpha_3+2\alpha_4+\alpha_5+\alpha_6>$, 
   $<\alpha_1+\alpha_3+\alpha_4+\alpha_5+\alpha_6,\alpha_2+\alpha_3+2\alpha_4+\alpha_5+\alpha_6>$, 
   $<\alpha_1+\alpha_2+2\alpha_3+3\alpha_4+2\alpha_5+\alpha_6,\alpha_6>$, 
   $<\alpha_4,\alpha_3+\alpha_4+\alpha_5+\alpha_6>$, $<\alpha_4,\alpha_1+\alpha_3+\alpha_4+\alpha_5+\alpha_6>$,
   $<\alpha_3+\alpha_4,\alpha_4+\alpha_5+\alpha_6>$, $<\alpha_3+\alpha_4,\alpha_1+\alpha_3+\alpha_4+\alpha_5+\alpha_6>$, 
   $<\alpha_1+\alpha_3+\alpha_4,\alpha_4+\alpha_5+\alpha_6>$, $<\alpha_1+\alpha_3+\alpha_4,\alpha_3+\alpha_4+\alpha_5+\alpha_6>$  \\
\hline
 \multirow{4}{*}{$A_1+A_1+A_1+A_1$} & $<\alpha_1>$, 
   $<\alpha_3>$, 
   $<\alpha_1+\alpha_3>$,
   $<\alpha_2+\alpha_3+2\alpha_4+\alpha_5>$, 
   $<\alpha_1+\alpha_2+\alpha_3+2\alpha_4+\alpha_5>$, 
   $<\alpha_1+\alpha_2+2\alpha_3+2\alpha_4+\alpha_5>$, 
   $<\alpha_2+\alpha_3+2\alpha_4+2\alpha_5+2\alpha_6+\alpha_7>$, 
   $<\alpha_1+\alpha_2+\alpha_3+2\alpha_4+2\alpha_5+2\alpha_6+\alpha_7>$, 
   $<\alpha_1+\alpha_2+2\alpha_3+2\alpha_4+2\alpha_5+2\alpha_6+\alpha_7>$, 
   $<\alpha_1+2\alpha_2+2\alpha_3+4\alpha_4+3\alpha_5+2\alpha_6+\alpha_7>$, 
   $<\alpha_1+2\alpha_2+3\alpha_3+4\alpha_4+3\alpha_5+2\alpha_6+\alpha_7>$, 
   $<2\alpha_1+2\alpha_2+3\alpha_3+4\alpha_4+3\alpha_5+2\alpha_6+\alpha_7>$   \\
\hline
 \multirow{8}{*}{$A_2+A_1+A_1+A_1$} &  $<\alpha_1,\alpha_3>$, $<\alpha_1,\alpha_2+\alpha_3+2\alpha_4+\alpha_5>$, $<\alpha_1,\alpha_2+\alpha_3+2\alpha_4+2\alpha_5+2\alpha_6+\alpha_7>$, $<\alpha_1,\alpha_1+2\alpha_2+3\alpha_3+4\alpha_4+3\alpha_5+2\alpha_6+\alpha_7>$,
   $<\alpha_3,\alpha_1+\alpha_2+\alpha_3+2\alpha_4+\alpha_5>$, $<\alpha_3,\alpha_1+\alpha_2+\alpha_3+2\alpha_4+2\alpha_5+2\alpha_6+\alpha_7>$, $<\alpha_3,\alpha_1+2\alpha_2+2\alpha_3+4\alpha_4+3\alpha_5+2\alpha_6+\alpha_7>$, 
   $<\alpha_1+\alpha_3,\alpha_2+\alpha_3+2\alpha_4+\alpha_5>$, $<\alpha_1+\alpha_3,\alpha_2+\alpha_3+2\alpha_4+2\alpha_5+2\alpha_6+\alpha_7>$, $<\alpha_1+\alpha_3,\alpha_1+2\alpha_2+2\alpha_3+4\alpha_4+3\alpha_5+2\alpha_6+\alpha_7>$, 
   $<\alpha_2+\alpha_3+2\alpha_4+\alpha_5,\alpha_1+\alpha_2+\alpha_3+2\alpha_4+2\alpha_5+2\alpha_6+\alpha_7>$,  $<\alpha_2+\alpha_3+2\alpha_4+\alpha_5,\alpha_1+\alpha_2+2\alpha_3+2\alpha_4+2\alpha_5+2\alpha_6+\alpha_7>$,
   $<\alpha_1+\alpha_2+\alpha_3+2\alpha_4+\alpha_5,\alpha_2+\alpha_3+2\alpha_4+2\alpha_5+2\alpha_6+\alpha_7>$,  $<\alpha_1+\alpha_2+\alpha_3+2\alpha_4+\alpha_5,\alpha_1+\alpha_2+2\alpha_3+2\alpha_4+2\alpha_5+2\alpha_6+\alpha_7>$,
   $<\alpha_1+\alpha_2+2\alpha_3+2\alpha_4+\alpha_5,\alpha_2+\alpha_3+2\alpha_4+2\alpha_5+2\alpha_6+\alpha_7>$,  $<\alpha_1+\alpha_2+2\alpha_3+2\alpha_4+\alpha_5,\alpha_1+\alpha_2+\alpha_3+2\alpha_4+2\alpha_5+2\alpha_6+\alpha_7>$ \\
\hline
 \multirow{35}{*}{$A_3+A_1+A_1$} &  $<\alpha_4,\alpha_1>$, $<\alpha_4,\alpha_1+\alpha_2+2\alpha_3+2\alpha_4+\alpha_5>$, $<\alpha_4,\alpha_2+\alpha_3+2\alpha_4+2\alpha_5+2\alpha_6+\alpha_7>$, $<\alpha_4,\alpha_1+\alpha_2+\alpha_3+2\alpha_4+2\alpha_5+2\alpha_6+\alpha_7>$, $<\alpha_4,\alpha_1+2\alpha_2+3\alpha_3+4\alpha_4+3\alpha_5+2\alpha_6+\alpha_7>$, $<\alpha_4,2\alpha_1+2\alpha_2+3\alpha_3+4\alpha_4+3\alpha_5+2\alpha_6+\alpha_7>$, 
  $<\alpha_3+\alpha_4,\alpha_1+\alpha_3>$, $<\alpha_3+\alpha_4,\alpha_1+\alpha_2+\alpha_3+2\alpha_4+\alpha_5>$, $<\alpha_3+\alpha_4,\alpha_2+\alpha_3+2\alpha_4+2\alpha_5+2\alpha_6+\alpha_7>$, $<\alpha_3+\alpha_4,\alpha_1+\alpha_2+2\alpha_3+2\alpha_4+2\alpha_5+2\alpha_6+\alpha_7>$, $<\alpha_3+\alpha_4,\alpha_1+2\alpha_2+2\alpha_3+4\alpha_4+3\alpha_5+2\alpha_6+\alpha_7>$, $<\alpha_3+\alpha_4,2\alpha_1+2\alpha_2+3\alpha_3+4\alpha_4+3\alpha_5+2\alpha_6+\alpha_7>$,  
  $<\alpha_1+\alpha_3+\alpha_4,\alpha_3>$, $<\alpha_1+\alpha_3+\alpha_4,\alpha_2+\alpha_3+2\alpha_4+\alpha_5>$, $<\alpha_1+\alpha_3+\alpha_4,\alpha_1+\alpha_2+\alpha_3+2\alpha_4+2\alpha_5+2\alpha_6+\alpha_7>$, $<\alpha_1+\alpha_3+\alpha_4,\alpha_1+\alpha_2+2\alpha_3+2\alpha_4+2\alpha_5+2\alpha_6+\alpha_7>$, $<\alpha_1+\alpha_3+\alpha_4,\alpha_1+2\alpha_2+2\alpha_3+4\alpha_4+3\alpha_5+2\alpha_6+\alpha_7>$, $<\alpha_1+\alpha_3+\alpha_4,\alpha_1+2\alpha_2+3\alpha_3+4\alpha_4+3\alpha_5+2\alpha_6+\alpha_7>$,
  $<\alpha_1+\alpha_2+2\alpha_3+3\alpha_4+2\alpha_5+2\alpha_6+\alpha_7,\alpha_1>$,  $<\alpha_1+\alpha_2+2\alpha_3+3\alpha_4+2\alpha_5+2\alpha_6+\alpha_7,\alpha_3>$,  $<\alpha_1+\alpha_2+2\alpha_3+3\alpha_4+2\alpha_5+2\alpha_6+\alpha_7,\alpha_1+\alpha_3>$,  $<\alpha_1+\alpha_2+2\alpha_3+3\alpha_4+2\alpha_5+2\alpha_6+\alpha_7,\alpha_2+\alpha_3+2\alpha_4+\alpha_5>$,  $<\alpha_1+\alpha_2+2\alpha_3+3\alpha_4+2\alpha_5+2\alpha_6+\alpha_7,\alpha_1+\alpha_2+\alpha_3+2\alpha_4+\alpha_5>$,  $<\alpha_1+\alpha_2+2\alpha_3+3\alpha_4+2\alpha_5+2\alpha_6+\alpha_7,\alpha_1+\alpha_2+2\alpha_3+2\alpha_4+\alpha_5>$,
  \newline
  \newline
   $<\alpha_4+\alpha_5+\alpha_6,\alpha_1>$, $<\alpha_4+\alpha_5+\alpha_6,\alpha_2+\alpha_3+2\alpha_4+\alpha_5>$, $<\alpha_4+\alpha_5+\alpha_6,\alpha_1+\alpha_2+\alpha_3+2\alpha_4+\alpha_5>$, $<\alpha_4+\alpha_5+\alpha_6,\alpha_1+\alpha_2+2\alpha_3+2\alpha_4+2\alpha_5+2\alpha_6+\alpha_7>$, $<\alpha_4+\alpha_5+\alpha_6,\alpha_1+2\alpha_2+3\alpha_3+4\alpha_4+3\alpha_5+2\alpha_6+\alpha_7>$, $<\alpha_4+\alpha_5+\alpha_6,2\alpha_1+2\alpha_2+3\alpha_3+4\alpha_4+3\alpha_5+2\alpha_6+\alpha_7>$, 
  $<\alpha_3+\alpha_4+\alpha_5+\alpha_6,\alpha_1+\alpha_3>$, $<\alpha_3+\alpha_4+\alpha_5+\alpha_6,\alpha_2+\alpha_3+2\alpha_4+\alpha_5>$, $<\alpha_3+\alpha_4+\alpha_5+\alpha_6,\alpha_1+\alpha_2+2\alpha_3+2\alpha_4+\alpha_5>$, $<\alpha_3+\alpha_4+\alpha_5+\alpha_6,\alpha_1+\alpha_2+\alpha_3+2\alpha_4+2\alpha_5+2\alpha_6+\alpha_7>$, $<\alpha_3+\alpha_4+\alpha_5+\alpha_6,\alpha_1+2\alpha_2+2\alpha_3+4\alpha_4+3\alpha_5+2\alpha_6+\alpha_7>$, $<\alpha_3+\alpha_4+\alpha_5+\alpha_6,2\alpha_1+2\alpha_2+3\alpha_3+4\alpha_4+3\alpha_5+2\alpha_6+\alpha_7>$, 
  $<\alpha_1+\alpha_3+\alpha_4+\alpha_5+\alpha_6,\alpha_3>$, $<\alpha_1+\alpha_3+\alpha_4+\alpha_5+\alpha_6,\alpha_1+\alpha_2+\alpha_3+2\alpha_4+\alpha_5>$, $<\alpha_1+\alpha_3+\alpha_4+\alpha_5+\alpha_6,\alpha_1+\alpha_2+2\alpha_3+2\alpha_4+\alpha_5>$, $<\alpha_1+\alpha_3+\alpha_4+\alpha_5+\alpha_6,\alpha_2+\alpha_3+2\alpha_4+2\alpha_5+2\alpha_6+\alpha_7>$, 
  $<\alpha_1+\alpha_3+\alpha_4+\alpha_5+\alpha_6,\alpha_1+2\alpha_2+2\alpha_3+4\alpha_4+3\alpha_5+2\alpha_6+\alpha_7>$, $<\alpha_1+\alpha_3+\alpha_4+\alpha_5+\alpha_6, \alpha_1+ 2\alpha_2+ 3\alpha_3+ 4\alpha_4+ 3\alpha_5+2\alpha_6+\alpha_7>$,  
  $<\alpha_1+\alpha_2+2\alpha_3+3\alpha_4+2\alpha_5+\alpha_6,\alpha_1>$, $<\alpha_1+\alpha_2+2\alpha_3+3\alpha_4+2\alpha_5+\alpha_6,\alpha_3>$, $<\alpha_1+\alpha_2+2\alpha_3+3\alpha_4+2\alpha_5+\alpha_6,\alpha_1+\alpha_3>$, $<\alpha_1+\alpha_2+2\alpha_3+3\alpha_4+2\alpha_5+\alpha_6,\alpha_2+\alpha_3+2\alpha_4+2\alpha_5+2\alpha_6+\alpha_7>$, $<\alpha_1+\alpha_2+2\alpha_3+3\alpha_4+2\alpha_5+\alpha_6,\alpha_1+\alpha_2+\alpha_3+2\alpha_4+2\alpha_5+2\alpha_6+\alpha_7>$, $<\alpha_1+\alpha_2+2\alpha_3+3\alpha_4+2\alpha_5+\alpha_6,\alpha_1+\alpha_2+2\alpha_3+2\alpha_4+2\alpha_5+2\alpha_6+\alpha_7>$,
  \newline
  \newline
  $<\alpha_6,\alpha_1>$, $<\alpha_6,\alpha_3>$, $<\alpha_6,\alpha_1+\alpha_3>$, $<\alpha_6,\alpha_1+2\alpha_2+2\alpha_3+4\alpha_4+3\alpha_5+2\alpha_6+\alpha_7>$, $<\alpha_6,\alpha_1+2\alpha_2+3\alpha_3+4\alpha_4+3\alpha_5+2\alpha_6+\alpha_7>$, $<\alpha_6,2\alpha_1+2\alpha_2+3\alpha_3+4\alpha_4+3\alpha_5+2\alpha_6+\alpha_7>$, 
  $<\alpha_2+\alpha_3+2\alpha_4+\alpha_5+\alpha_6,\alpha_3>$, $<\alpha_2+\alpha_3+2\alpha_4+\alpha_5+\alpha_6,\alpha_1+\alpha_2+\alpha_3+2\alpha_4+\alpha_5>$, $<\alpha_2+\alpha_3+2\alpha_4+\alpha_5+\alpha_6,\alpha_1+\alpha_2+2\alpha_3+2\alpha_4+\alpha_5>$, $<\alpha_2+\alpha_3+2\alpha_4+\alpha_5+\alpha_6,\alpha_1+\alpha_2+\alpha_3+2\alpha_4+2\alpha_5+2\alpha_6+\alpha_7>$, $<\alpha_2+\alpha_3+2\alpha_4+\alpha_5+\alpha_6,\alpha_1+\alpha_2+2\alpha_3+2\alpha_4+2\alpha_5+2\alpha_6+\alpha_7>$, $<\alpha_2+\alpha_3+2\alpha_4+\alpha_5+\alpha_6,2\alpha_1+2\alpha_2+3\alpha_3+4\alpha_4+3\alpha_5+2\alpha_6+\alpha_7>$, 
  $<\alpha_1+\alpha_2+\alpha_3+2\alpha_4+\alpha_5+\alpha_6,\alpha_1+\alpha_3>$,  $<\alpha_1+\alpha_2+\alpha_3+2\alpha_4+\alpha_5+\alpha_6,\alpha_2+\alpha_3+2\alpha_4+\alpha_5>$,  $<\alpha_1+\alpha_2+\alpha_3+2\alpha_4+\alpha_5+\alpha_6,\alpha_1+\alpha_2+2\alpha_3+2\alpha_4+\alpha_5>$, $<\alpha_1+\alpha_2+\alpha_3+2\alpha_4+\alpha_5+\alpha_6,\alpha_2+\alpha_3+2\alpha_4+2\alpha_5+2\alpha_6+\alpha_7>$, 
 $<\alpha_1+\alpha_2+\alpha_3+2\alpha_4+\alpha_5+\alpha_6,\alpha_1+\alpha_2+2\alpha_3+2\alpha_4+2\alpha_5+2\alpha_6+\alpha_7>$,  $<\alpha_1+\alpha_2+\alpha_3+2\alpha_4+\alpha_5+\alpha_6,\alpha_1+2\alpha_2+3\alpha_3+4\alpha_4+3\alpha_5+2\alpha_6+\alpha_7>$,    
   $<\alpha_1+\alpha_2+2\alpha_3+2\alpha_4+\alpha_5+\alpha_6,\alpha_1>$, $<\alpha_1+\alpha_2+2\alpha_3+2\alpha_4+\alpha_5+\alpha_6,\alpha_2+\alpha_3+2\alpha_4+\alpha_5>$, $<\alpha_1+\alpha_2+2\alpha_3+2\alpha_4+\alpha_5+\alpha_6,\alpha_1+\alpha_2+\alpha_3+2\alpha_4+\alpha_5>$, $<\alpha_1+\alpha_2+2\alpha_3+2\alpha_4+\alpha_5+\alpha_6,\alpha_2+\alpha_3+2\alpha_4+2\alpha_5+2\alpha_6+\alpha_7>$, $<\alpha_1+\alpha_2+2\alpha_3+2\alpha_4+\alpha_5+\alpha_6,\alpha_1+\alpha_2+\alpha_3+2\alpha_4+2\alpha_5+2\alpha_6+\alpha_7>$, $<\alpha_1+\alpha_2+2\alpha_3+2\alpha_4+\alpha_5+\alpha_6,\alpha_1+2\alpha_2+2\alpha_3+4\alpha_4+3\alpha_5+2\alpha_6+\alpha_7>$  \\
\hline
\multirow{30}{*}{$A_3+A_2+A_1$} &  $<\alpha_4,\alpha_1 ,\alpha_2+\alpha_3+2\alpha_4+2\alpha_5+2\alpha_6+\alpha_7 >$,  $<\alpha_4,\alpha_1 ,\alpha_1+2\alpha_2+3\alpha_3+4\alpha_4+3\alpha_5+2\alpha_6+\alpha_7 >$,  $<\alpha_4, \alpha_1+\alpha_2+2\alpha_3+2\alpha_4+\alpha_5, \alpha_2+\alpha_3+2\alpha_4+2\alpha_5+2\alpha_6+\alpha_7>$,  $<\alpha_4,\alpha_1+\alpha_2+2\alpha_3+2\alpha_4+\alpha_5 , \alpha_1+\alpha_2+\alpha_3+2\alpha_4+2\alpha_5+2\alpha_6+\alpha_7>$, 

  $<\alpha_3+\alpha_4, \alpha_1+\alpha_3 ,\alpha_2+\alpha_3+2\alpha_4+2\alpha_5+2\alpha_6+\alpha_7>$, $<\alpha_3+\alpha_4, \alpha_1+\alpha_3,\alpha_1+2\alpha_2+2\alpha_3+4\alpha_4+3\alpha_5+2\alpha_6+\alpha_7 >$, $<\alpha_3+\alpha_4,\alpha_1+\alpha_2+\alpha_3+2\alpha_4+\alpha_5 ,\alpha_2+\alpha_3+2\alpha_4+2\alpha_5+2\alpha_6+\alpha_7 >$, $<\alpha_3+\alpha_4, \alpha_1+\alpha_2+\alpha_3+2\alpha_4+\alpha_5, \alpha_1+\alpha_2+2\alpha_3+2\alpha_4+2\alpha_5+2\alpha_6+\alpha_7>$, 
  
 $<\alpha_1+\alpha_3+\alpha_4,\alpha_3,  \alpha_1+\alpha_2+\alpha_3+2\alpha_4+2\alpha_5+2\alpha_6+\alpha_7 >$, $<\alpha_1+\alpha_3+\alpha_4, \alpha_3, \alpha_1+2\alpha_2+2\alpha_3+4\alpha_4+3\alpha_5+2\alpha_6+\alpha_7 >$, $<\alpha_1+\alpha_3+\alpha_4,\alpha_2+ \alpha_3+2\alpha_4+\alpha_5,  \alpha_1+\alpha_2+\alpha_3+2\alpha_4+2\alpha_5+2\alpha_6+\alpha_7  >$, $<\alpha_1+\alpha_3+\alpha_4, \alpha_2+\alpha_3+2\alpha_4+\alpha_5 , \alpha_1+\alpha_2+2\alpha_3+2\alpha_4+2\alpha_5+2\alpha_6+\alpha_7  >$,  
 
  $<\alpha_1+\alpha_2+2\alpha_3+3\alpha_4+2\alpha_5+2\alpha_6+\alpha_7, \alpha_1, \alpha_3 >$,  $<\alpha_1+\alpha_2+2\alpha_3+3\alpha_4+2\alpha_5+2\alpha_6+\alpha_7, \alpha_1,\alpha_2+\alpha_3+2\alpha_4+\alpha_5 >$,  $<\alpha_1+\alpha_2+2\alpha_3+3\alpha_4+2\alpha_5+2\alpha_6+\alpha_7, \alpha_3, \alpha_1+\alpha_2+\alpha_3+2\alpha_4+\alpha_5>$,  $<\alpha_1+\alpha_2+2\alpha_3+3\alpha_4+2\alpha_5+2\alpha_6+\alpha_7, \alpha_1+\alpha_3, \alpha_2+\alpha_3+2\alpha_4+\alpha_5>$,  
  
  $<\alpha_4+\alpha_5+\alpha_6, \alpha_1, \alpha_2+ \alpha_3+2\alpha_4+\alpha_5>$, $<\alpha_4+\alpha_5+\alpha_6, \alpha_1, \alpha_1+2\alpha_2+3\alpha_3+4\alpha_4+3\alpha_5+2\alpha_6+\alpha_7>$, $<\alpha_4+\alpha_5+\alpha_6, \alpha_2+ \alpha_3+2\alpha_4+\alpha_5, \alpha_1+\alpha_2+2\alpha_3+2\alpha_4+2\alpha_5+2\alpha_6+\alpha_7>$, $<\alpha_4+\alpha_5+\alpha_6, \alpha_1+\alpha_2+ \alpha_3+2\alpha_4+\alpha_5, \alpha_1+\alpha_2+2\alpha_3+2\alpha_4+2\alpha_5+2\alpha_6+\alpha_7 >$, 
  
  $<\alpha_3+\alpha_4+\alpha_5+\alpha_6, \alpha_1+\alpha_3,  \alpha_2+ \alpha_3+2\alpha_4+\alpha_5>$, $<\alpha_3+\alpha_4+\alpha_5+\alpha_6,  \alpha_1+\alpha_3,  \alpha_1+2\alpha_2+2\alpha_3+4\alpha_4+3\alpha_5+2\alpha_6+\alpha_7>$, $<\alpha_3+\alpha_4+\alpha_5+\alpha_6,  \alpha_2+ \alpha_3+2\alpha_4+\alpha_5,  \alpha_1+\alpha_2+\alpha_3+2\alpha_4+2\alpha_5+2\alpha_6+\alpha_7>$, $<\alpha_3+\alpha_4+\alpha_5+\alpha_6, \alpha_1+ \alpha_2+ 2\alpha_3+2\alpha_4+\alpha_5,  \alpha_1+\alpha_2+\alpha_3+2\alpha_4+2\alpha_5+2\alpha_6+\alpha_7>$, 
  
  $<\alpha_1+\alpha_3+\alpha_4+\alpha_5+\alpha_6, \alpha_3,  \alpha_1+ \alpha_2+ \alpha_3+2\alpha_4+\alpha_5>$, $<\alpha_1+\alpha_3+\alpha_4+\alpha_5+\alpha_6, \alpha_3, \alpha_1+2\alpha_2+2\alpha_3+4\alpha_4+3\alpha_5+2\alpha_6+\alpha_7>$, $<\alpha_1+\alpha_3+\alpha_4+\alpha_5+\alpha_6,  \alpha_1+ \alpha_2+ \alpha_3+2\alpha_4+\alpha_5, \alpha_2+\alpha_3+2\alpha_4+2\alpha_5+2\alpha_6+\alpha_7>$, $<\alpha_1+\alpha_3+\alpha_4+\alpha_5+\alpha_6,  \alpha_1+ \alpha_2+ 2\alpha_3+2\alpha_4+\alpha_5, \alpha_2+\alpha_3+2\alpha_4+2\alpha_5+2\alpha_6+\alpha_7>$,
     
  $<\alpha_1+\alpha_2+2\alpha_3+3\alpha_4+2\alpha_5+\alpha_6, \alpha_1, \alpha_3>$, $<\alpha_1+\alpha_2+2\alpha_3+3\alpha_4+2\alpha_5+\alpha_6, \alpha_1,  \alpha_2+\alpha_3+2\alpha_4+2\alpha_5+2\alpha_6+\alpha_7>$, $<\alpha_1+\alpha_2+2\alpha_3+3\alpha_4+2\alpha_5+\alpha_6, \alpha_3, \alpha_1+ \alpha_2+\alpha_3+2\alpha_4+2\alpha_5+2\alpha_6+\alpha_7>$, $<\alpha_1+\alpha_2+2\alpha_3+3\alpha_4+2\alpha_5+\alpha_6, \alpha_1+\alpha_3,  \alpha_2+\alpha_3+2\alpha_4+2\alpha_5+2\alpha_6+\alpha_7>$,  
  
  $<\alpha_6, \alpha_1, \alpha_3>$, $<\alpha_6, \alpha_1, \alpha_1+2\alpha_2+3\alpha_3+4\alpha_4+3\alpha_5+2\alpha_6+\alpha_7>$, $<\alpha_6, \alpha_3, \alpha_1+2\alpha_2+2\alpha_3+4\alpha_4+3\alpha_5+2\alpha_6+\alpha_7>$, $<\alpha_6, \alpha_1+\alpha_3, \alpha_1+2\alpha_2+2\alpha_3+4\alpha_4+3\alpha_5+2\alpha_6+\alpha_7>$,   
  
  $<\alpha_2+\alpha_3+2\alpha_4+\alpha_5+\alpha_6, \alpha_3, \alpha_1+\alpha_2+\alpha_3+2\alpha_4+\alpha_5>$, $<\alpha_2+\alpha_3+2\alpha_4+\alpha_5+\alpha_6, \alpha_3,  \alpha_1+\alpha_2+\alpha_3+2\alpha_4+2\alpha_5+2\alpha_6+\alpha_7>$, $<\alpha_2+\alpha_3+2\alpha_4+\alpha_5+\alpha_6,  \alpha_1+\alpha_2+\alpha_3+2\alpha_4+\alpha_5, \alpha_1+\alpha_2+2\alpha_3+2\alpha_4+2\alpha_5+2\alpha_6+\alpha_7>$, $<\alpha_2+\alpha_3+2\alpha_4+\alpha_5+\alpha_6, \alpha_1+\alpha_2+2\alpha_3+2\alpha_4+\alpha_5, \alpha_1+\alpha_2+\alpha_3+2\alpha_4+2\alpha_5+2\alpha_6+\alpha_7>$, 
  
  $<\alpha_1+\alpha_2+\alpha_3+2\alpha_4+\alpha_5+\alpha_6, \alpha_1+\alpha_3, \alpha_2+\alpha_3+2\alpha_4+\alpha_5>$,  $<\alpha_1+\alpha_2+\alpha_3+2\alpha_4+\alpha_5+\alpha_6, \alpha_1+\alpha_3, \alpha_2+\alpha_3+2\alpha_4+2\alpha_5+2\alpha_6+\alpha_7>$,  $<\alpha_1+\alpha_2+\alpha_3+2\alpha_4+\alpha_5+\alpha_6, \alpha_2+\alpha_3+2\alpha_4+\alpha_5, \alpha_1+\alpha_2+2\alpha_3+2\alpha_4+2\alpha_5+2\alpha_6+\alpha_7>$,  $<\alpha_1+\alpha_2+\alpha_3+2\alpha_4+\alpha_5+\alpha_6, \alpha_1+\alpha_2+2\alpha_3+2\alpha_4+\alpha_5, \alpha_2+\alpha_3+2\alpha_4+2\alpha_5+2\alpha_6+\alpha_7>$, 
  
   $<\alpha_1+\alpha_2+2\alpha_3+2\alpha_4+\alpha_5+\alpha_6, \alpha_1, \alpha_2+\alpha_3+2\alpha_4+\alpha_5>$, $<\alpha_1+\alpha_2+2\alpha_3+2\alpha_4+\alpha_5+\alpha_6, \alpha_1, \alpha_2+\alpha_3+2\alpha_4+2\alpha_5+2\alpha_6+\alpha_7>$, $<\alpha_1+\alpha_2+2\alpha_3+2\alpha_4+\alpha_5+\alpha_6, \alpha_2+\alpha_3+2\alpha_4+\alpha_5, \alpha_1+\alpha_2+\alpha_3+2\alpha_4+2\alpha_5+2\alpha_6+\alpha_7>$, $<\alpha_1+\alpha_2+2\alpha_3+2\alpha_4+\alpha_5+\alpha_6, \alpha_1+\alpha_2+\alpha_3+2\alpha_4+\alpha_5, \alpha_2+\alpha_3+2\alpha_4+2\alpha_5+2\alpha_6+\alpha_7>$  \\
\hline
\multirow{30}{*}{$A_5+A_1$}  &  $<\alpha_4, \alpha_6, \alpha_1>$, $<\alpha_4, \alpha_6, \alpha_1+2\alpha_2+3\alpha_3+4\alpha_4+3\alpha_5+2\alpha_6+\alpha_7>$, $<\alpha_4, \alpha_6, 2\alpha_1+2\alpha_2+3\alpha_3+4\alpha_4+3\alpha_5+2\alpha_6+\alpha_7>$, $<\alpha_4, \alpha_1+\alpha_2+2\alpha_3+2\alpha_4+\alpha_5+\alpha_6, \alpha_1>$, $<\alpha_4, \alpha_1+\alpha_2+2\alpha_3+2\alpha_4+\alpha_5+\alpha_6, \alpha_2+\alpha_3+2\alpha_4+2\alpha_5+2\alpha_6+\alpha_7>$, $<\alpha_4, \alpha_1+\alpha_2+2\alpha_3+2\alpha_4+\alpha_5+\alpha_6, \alpha_1+ \alpha_2+\alpha_3+2\alpha_4+2\alpha_5+2\alpha_6+\alpha_7>$, 
   $<\alpha_3+\alpha_4, \alpha_6, \alpha_1+\alpha_3>$, $<\alpha_3+\alpha_4, \alpha_6,  \alpha_1+ 2\alpha_2+2\alpha_3+4\alpha_4+3\alpha_5+2\alpha_6+\alpha_7>$, $<\alpha_3+\alpha_4, \alpha_6,  2\alpha_1+ 2\alpha_2+3\alpha_3+4\alpha_4+3\alpha_5+2\alpha_6+\alpha_7>$, $<\alpha_3+\alpha_4, \alpha_1+\alpha_2+\alpha_3+2\alpha_4+\alpha_5+\alpha_6, \alpha_1+\alpha_3>$, $<\alpha_3+\alpha_4,  \alpha_1+\alpha_2+\alpha_3+2\alpha_4+\alpha_5+\alpha_6,  \alpha_2+\alpha_3+2\alpha_4+2\alpha_5+2\alpha_6+\alpha_7>$, $<\alpha_3+\alpha_4,  \alpha_1+\alpha_2+\alpha_3+2\alpha_4+\alpha_5+\alpha_6,  \alpha_1+ \alpha_2+2\alpha_3+2\alpha_4+2\alpha_5+2\alpha_6+\alpha_7>$, 
   
   $<\alpha_1+\alpha_3+\alpha_4, \alpha_6, \alpha_3>$, $<\alpha_1+\alpha_3+\alpha_4, \alpha_6,  \alpha_1+2 \alpha_2+2\alpha_3+4\alpha_4+3\alpha_5+2\alpha_6+\alpha_7>$, $<\alpha_1+\alpha_3+\alpha_4, \alpha_6,  \alpha_1+ 2\alpha_2+3\alpha_3+4\alpha_4+3\alpha_5+2\alpha_6+\alpha_7 >$, $<\alpha_1+\alpha_3+\alpha_4, \alpha_2+\alpha_3+2\alpha_4+\alpha_5+\alpha_6, \alpha_3>$, $<\alpha_1+\alpha_3+\alpha_4,  \alpha_2+\alpha_3+2\alpha_4+\alpha_5+\alpha_6,  \alpha_1+ \alpha_2+\alpha_3+2\alpha_4+2\alpha_5+2\alpha_6+\alpha_7>$, $<\alpha_1+\alpha_3+\alpha_4,  \alpha_2+\alpha_3+2\alpha_4+\alpha_5+\alpha_6,  \alpha_1+ \alpha_2+2\alpha_3+2\alpha_4+2\alpha_5+2\alpha_6+\alpha_7>$,  
   
   $<\alpha_4+\alpha_5+\alpha_6,  \alpha_1+ \alpha_2+2\alpha_3+2\alpha_4+\alpha_5+\alpha_6, \alpha_1>$, $<\alpha_4+\alpha_5+\alpha_6,  \alpha_1+ \alpha_2+2\alpha_3+2\alpha_4+\alpha_5+\alpha_6,  \alpha_2+\alpha_3+2\alpha_4+\alpha_5>$, $<\alpha_4+\alpha_5+\alpha_6,  \alpha_1+ \alpha_2+2\alpha_3+2\alpha_4+\alpha_5+\alpha_6,  \alpha_1+ \alpha_2+\alpha_3+2\alpha_4+\alpha_5>$, 
   
   $<\alpha_3+\alpha_4+\alpha_5+\alpha_6,  \alpha_1+ \alpha_2+\alpha_3+2\alpha_4+\alpha_5+\alpha_6, \alpha_1+\alpha_3>$, $<\alpha_3+\alpha_4+\alpha_5+\alpha_6,  \alpha_1+ \alpha_2+\alpha_3+2\alpha_4+\alpha_5+\alpha_6,  \alpha_2+\alpha_3+2\alpha_4+\alpha_5>$, $<\alpha_3+\alpha_4+\alpha_5+\alpha_6,  \alpha_1+ \alpha_2+\alpha_3+2\alpha_4+\alpha_5+\alpha_6,  \alpha_1+ \alpha_2+2\alpha_3+2\alpha_4+\alpha_5>$, 
   
   $<\alpha_1+\alpha_3+\alpha_4+\alpha_5+\alpha_6,  \alpha_2+\alpha_3+2\alpha_4+\alpha_5+\alpha_6, \alpha_3>$, $<\alpha_1+\alpha_3+\alpha_4+\alpha_5+\alpha_6,  \alpha_2+\alpha_3+2\alpha_4+\alpha_5+\alpha_6,   \alpha_1+\alpha_2+\alpha_3+2\alpha_4+\alpha_5>$, $<\alpha_1+\alpha_3+\alpha_4+\alpha_5+\alpha_6,  \alpha_2+\alpha_3+2\alpha_4+\alpha_5+\alpha_6,  \alpha_1+ \alpha_2+2\alpha_3+2\alpha_4+\alpha_5>$, 
   
   $<\alpha_1+\alpha_2+2\alpha_3+3\alpha_4+2\alpha_5+\alpha_6, \alpha_6, \alpha_1>$, $<\alpha_1+\alpha_2+2\alpha_3+3\alpha_4+2\alpha_5+\alpha_6, \alpha_6, \alpha_3>$, $<\alpha_1+\alpha_2+2\alpha_3+3\alpha_4+2\alpha_5+\alpha_6, \alpha_6, \alpha_1+\alpha_3>$, 

   $<\alpha_4, \alpha_3+\alpha_4+\alpha_5+\alpha_6,  \alpha_1+ \alpha_2+2\alpha_3+2\alpha_4+\alpha_5>$, $<\alpha_4, \alpha_3+\alpha_4+\alpha_5+\alpha_6,  \alpha_1+ \alpha_2+\alpha_3+2\alpha_4+2\alpha_5+2\alpha_6+\alpha_7>$, $<\alpha_4, \alpha_3+\alpha_4+\alpha_5+\alpha_6,  2\alpha_1+ 2\alpha_2+3\alpha_3+4\alpha_4+3\alpha_5+2\alpha_6+\alpha_7>$, $<\alpha_4, \alpha_1+\alpha_3+\alpha_4+\alpha_5+\alpha_6,  \alpha_1+ \alpha_2+2\alpha_3+2\alpha_4+\alpha_5>$, $<\alpha_4, \alpha_1+\alpha_3+\alpha_4+\alpha_5+\alpha_6, \alpha_2+\alpha_3+2\alpha_4+2\alpha_5+2\alpha_6+\alpha_7>$, $<\alpha_4, \alpha_1+\alpha_3+\alpha_4+\alpha_5+\alpha_6,  \alpha_1+ 2\alpha_2+3\alpha_3+4\alpha_4+3\alpha_5+2\alpha_6+\alpha_7>$, 
   
   $<\alpha_3+\alpha_4, \alpha_4+\alpha_5+\alpha_6,  \alpha_1+ \alpha_2+\alpha_3+2\alpha_4+\alpha_5>$, $<\alpha_3+\alpha_4, \alpha_4+\alpha_5+\alpha_6,  \alpha_1+ \alpha_2+2\alpha_3+2\alpha_4+2\alpha_5+2\alpha_6+\alpha_7>$, $<\alpha_3+\alpha_4, \alpha_4+\alpha_5+\alpha_6,  2\alpha_1+ 2\alpha_2+3\alpha_3+4\alpha_4+3\alpha_5+2\alpha_6+\alpha_7>$, $<\alpha_3+\alpha_4, \alpha_1+\alpha_3+\alpha_4+\alpha_5+\alpha_6,  \alpha_1+ \alpha_2+\alpha_3+2\alpha_4+\alpha_5>$, $<\alpha_3+\alpha_4,  \alpha_1+\alpha_3+\alpha_4+\alpha_5+\alpha_6,  \alpha_2+\alpha_3+2\alpha_4+2\alpha_5+2\alpha_6+\alpha_7>$, $<\alpha_3+\alpha_4,  \alpha_1+\alpha_3+\alpha_4+\alpha_5+\alpha_6,  \alpha_1+ 2\alpha_2+2\alpha_3+4\alpha_4+3\alpha_5+2\alpha_6+\alpha_7>$, 
   
   $<\alpha_1+\alpha_3+\alpha_4, \alpha_4+\alpha_5+\alpha_6,  \alpha_2+\alpha_3+2\alpha_4+\alpha_5>$, $<\alpha_1+\alpha_3+\alpha_4, \alpha_4+\alpha_5+\alpha_6,  \alpha_1+ \alpha_2+2\alpha_3+2\alpha_4+2\alpha_5+2\alpha_6+\alpha_7>$, $<\alpha_1+\alpha_3+\alpha_4, \alpha_4+\alpha_5+\alpha_6,  \alpha_1+ 2\alpha_2+3\alpha_3+4\alpha_4+3\alpha_5+2\alpha_6+\alpha_7>$, $<\alpha_1+\alpha_3+\alpha_4,  \alpha_3+\alpha_4+\alpha_5+\alpha_6,  \alpha_2+\alpha_3+2\alpha_4+\alpha_5>$, $<\alpha_1+\alpha_3+\alpha_4,  \alpha_3+\alpha_4+\alpha_5+\alpha_6,  \alpha_1+ \alpha_2+\alpha_3+2\alpha_4+2\alpha_5+2\alpha_6+\alpha_7>$, $<\alpha_1+\alpha_3+\alpha_4,  \alpha_3+\alpha_4+\alpha_5+\alpha_6,  \alpha_1+ 2\alpha_2+2\alpha_3+4\alpha_4+3\alpha_5+2\alpha_6+\alpha_7>$  \\
\hline
$E_7$ & $<\alpha_1,\alpha_2,\alpha_3,\alpha_4,\alpha_5,\alpha_6,\alpha_7>$   \\
\hline
\end{longtable}
{\addtocounter{table}{-1}}
\captionof{table}{}
  \label{realizationsE6}
\end{center}
}
\medskip

The description of $E_7$ will be the same as the one in Section~\ref{D_4-G_2-E_7}. Therefore $h' \in H_{\alpha_2} \bigcap H_{\alpha_5}  \bigcap H_{\alpha_7}$ if and only if $\xi_2=-\xi_1$, $\xi_4=\xi_3$ and $\xi_6=\xi_5$. \\
\newline

Using \cite{Abria09} and \cite{Mehta88}, one can compute the flat coordinates of $E_7$ restricted to the subspace \\
$ H_{\alpha_2} \bigcap H_{\alpha_5} \bigcap H_{\alpha_7}   \ \mathlarger{\mathlarger{\subset}} \  \mathfrak{h}'$, i.e. as functions of $\xi_1,\xi_3,\xi_5$ and $\xi_7$. We will not give these expressions because of their extensive size. For example, $\psi_{18}$ contains 220 terms. \\
\medskip

It can be verified that 
\begin{center}
\renewcommand{\arraystretch}{1.4}
\raisebox{-.80\height}{$\left\lbrace\begin{array}[h]{ccl} 
\psi_{10} & = & \frac{82928}{45}\psi_2^5-\frac{4}{3}\psi_2^2\psi_6+ \frac{9}{35}\psi_2\psi_8,\\
\psi_{14} & = & \begin{array}[t]{l} -\frac{12190772504}{2219805}\psi_2^7+\frac{7281979}{1522152}\psi_2^4\psi_6-\frac{471547}{789264}\psi_2^3\psi_8+\frac{3779}{6765120}\psi_2\psi_6^2 \\ +\frac{73490}{8613}\psi_{12}\psi_2-\frac{1}{25920}\psi_6\psi_8, \end{array}\\
\psi_{18} & = & \begin{array}[t]{l} -\frac{1271044247268145576}{94705443405}\psi_2^9+\frac{256749355304}{25982289}\psi_2^6\psi_6-\frac{12475637391961}{9093801150}\psi_2^5\psi_8 \\ +\frac{6360724111}{3117874680}\psi_2^3\psi_6^2+\frac{470160383920}{31756131}\psi_{12}\psi_2^3-\frac{5935967}{41810580}\psi_2^2\psi_6\psi_8 \\ -\frac{101699}{30970800}\psi_2\psi_8^2-\frac{1}{52488}\psi_6^3-\frac{20}{27}\psi_{12}\psi_6. \end{array}\\
\end{array}\right.$}
\end{center}
\medskip

Define 
 \begin{adjustwidth}{-1cm}{1.5cm} 
\begin{center}
$\begin{array}[t]{rccc}
f :  & (\mathfrak{h}/W)^\Omega & \rightarrow & \mathfrak{h}'/W' \\
   & \begin{pmatrix}
   t_2  \\
   0 \\
   t_6 \\
   t_8 \\
   0 \\
   t_{12}
   \end{pmatrix} & \mapsto & \resizebox{\linewidth}{!}{$ \renewcommand{\arraystretch}{1.5}\begin{pmatrix}
   t_2 \\
   \frac{16735}{9}t_2^3-3000t_6 \\
   \frac{4884005}{972}t_2^4-\frac{360500}{9}t_2t_6+50000t_8 \\
   \frac{13113}{20}t_2^5-6300t_2^2t_6+\frac{90000}{7}t_2t_8 \\
    -\frac{1533855367}{5248800}t_2^6-\frac{5865805}{7776}t_2^3t_6+\frac{2927375}{504}t_2^2t_8-\frac{34375}{72}t_6^2-3125t_{12} \\
   -\frac{4794135161101}{9205975296}t_2^7+\frac{2226935425}{6088608}t_2^4t_6+\frac{4765011875}{295974}t_2^3t_8-\frac{103796875}{28188}t_2t_6^2-\frac{229656250}{8613}t_2t_{12}+\frac{156250}{27}t_6t_8  \\
 \begin{smallmatrix}
       \frac{53845033157553005}{8418261636}t_2^6t_6-\frac{2973773239515500}{545628069}t_2^5t_8-\frac{2028480753289375}{155893734}t_2^3t_6^2-\frac{3997692756437500}{95268393}t_2^3t_{12} -\frac{635618750000}{77427}t_2t_8^2 \\ -\frac{62500000}{9}t_{12}t_6+\frac{32999899562500}{696843}t_2^2t_6t_8-\frac{74647399081995101197}{218201341605120}t_2^9-\frac{132812500}{243}t_6^3
        \end{smallmatrix}   \\
   \end{pmatrix}$}
\end{array}$
\end{center}
 \end{adjustwidth}
\bigskip
 
\noindent The morphism $f$ is injective and if $(\psi_2,\psi_6,\psi_8,\psi_{10},\psi_{12},\psi_{14},\psi_{18}) \in \pi'(H_{\alpha_2} \bigcap H_{\alpha_5} \bigcap H_{\alpha_7})$, then by setting 
\begin{center}
\renewcommand{\arraystretch}{1.5}
\raisebox{-.55\height}{$\left\lbrace\begin{array}{l}
t_2=\psi_2, \\
t_6= \frac{3347}{5400}\psi_2^3-\frac{1}{3000}\psi_6, \\
t_8= \frac{1}{50000}\psi_8+\frac{1283191}{3240000}\psi_2^4-\frac{721}{2700000}\psi_2\psi_6, \\
t_{12}=\frac{18995770783}{43740000000}\psi_2^6-\frac{342859}{972000000}\psi_2^3\psi_6+\frac{23419}{630000000}\psi_2^2\psi_8-\frac{11}{648000000}\psi_6^2-\frac{1}{3125}\psi_{12},
\end{array}\right.$}
\end{center}
it follows that $f(t_2,0,t_6,t_8,0,t_{12})=(\psi_2,\psi_6,\psi_8,\psi_{10},\psi_{12},\psi_{14},\psi_{18})$. Hence
\begin{center}
 \begin{tikzpicture}[scale=1,  transform shape]
\node (1) at ( 0,0)  {$f :  (\mathfrak{h}/W)^\Omega$};
\node (2) at ( 6,0)  {$\pi'(H_{\alpha_2} \bigcap H_{\alpha_5} \bigcap H_{\alpha_7})  \ \mathlarger{\mathlarger{\subset}} \  \mathfrak{h}'/W'.$};
\node (3) at (2.3,0.2)  {$\cong$};
\draw  [decoration={markings,mark=at position 1 with
    {\arrow[scale=1.2,>=stealth]{>}}},postaction={decorate}] (1)  --  (2);
\end{tikzpicture}
\end{center}

Let $\phi'$ be a sub-root system in Table~\ref{realizationsE6}, and set $h' \in \bigcap_{\alpha \in \phi'}H_\alpha$. With the preceding expressions, we compute $(t_2,0,t_6,t_8,0,t_{12})=f^{-1}(\pi'(h'))$, and verify the following correspondence : 
\begin{center}
\renewcommand{\arraystretch}{1.5}
\begin{longtable}[h]{lcl}
$\phi'$ of type $A_1+A_1+A_1$ &  \begin{tikzpicture}[scale=1,  transform shape]
\node (1) at ( 0,0)  {};
\node (2) at ( 2,0)  {};
\draw  [decoration={markings,mark=at position 1 with
    {\arrow[scale=1.2,>=stealth]{>}}},postaction={decorate}] (1)  --  (2);
\draw  [decoration={markings,mark=at position 1 with
    {\arrow[scale=1.2,>=stealth]{>}}},postaction={decorate}] (2)  --  (1);
\end{tikzpicture} & $(t_2,t_6,t_8,t_{12})$ generic, \\
$\phi'$ of type $A_3+A_1$ &  \begin{tikzpicture}[scale=1,  transform shape]
\node (1) at ( 0,0)  {};
\node (2) at ( 2,0)  {};
\draw  [decoration={markings,mark=at position 1 with
    {\arrow[scale=1.2,>=stealth]{>}}},postaction={decorate}] (1)  --  (2);
\draw  [decoration={markings,mark=at position 1 with
    {\arrow[scale=1.2,>=stealth]{>}}},postaction={decorate}] (2)  --  (1);
\end{tikzpicture} & $(t_2,t_6,t_8,t_{12}) \in \mathcal{H}_1$, \\
$\phi'$ of type $D_4+A_1$ &  \begin{tikzpicture}[scale=1,  transform shape]
\node (1) at ( 0,0)  {};
\node (2) at ( 2,0)  {};
\draw  [decoration={markings,mark=at position 1 with
    {\arrow[scale=1.2,>=stealth]{>}}},postaction={decorate}] (1)  --  (2);
\draw  [decoration={markings,mark=at position 1 with
    {\arrow[scale=1.2,>=stealth]{>}}},postaction={decorate}] (2)  --  (1);
\end{tikzpicture} & $(t_2,t_6,t_8,t_{12}) \in \mathcal{H}_1 \bigcap \{t_8=\frac{1}{96}\frac{t_2^6+96t_6^2}{t_2^2}\}$, \\
$\phi'$ of type $D_5+A_1$ &  \begin{tikzpicture}[scale=1,  transform shape]
\node (1) at ( 0,0)  {};
\node (2) at ( 2,0)  {};
\draw  [decoration={markings,mark=at position 1 with
    {\arrow[scale=1.2,>=stealth]{>}}},postaction={decorate}] (1)  --  (2);
\draw  [decoration={markings,mark=at position 1 with
    {\arrow[scale=1.2,>=stealth]{>}}},postaction={decorate}] (2)  --  (1);
\end{tikzpicture} & $(t_2,t_6,t_8,t_{12}) \in \mathcal{H}_1 \bigcap \{t_8=\frac{1}{96}\frac{t_2^6+96t_6^2}{t_2^2}\} \bigcap \{t_2^3+8t_6=0\}$, \\
$\phi'$ of type $D_6$ &  \begin{tikzpicture}[scale=1,  transform shape]
\node (1) at ( 0,0)  {};
\node (2) at ( 2,0)  {};
\draw  [decoration={markings,mark=at position 1 with
    {\arrow[scale=1.2,>=stealth]{>}}},postaction={decorate}] (1)  --  (2);
\draw  [decoration={markings,mark=at position 1 with
    {\arrow[scale=1.2,>=stealth]{>}}},postaction={decorate}] (2)  --  (1);
\end{tikzpicture} & $(t_2,t_6,t_8,t_{12}) \in \mathcal{H}_1 \bigcap \{t_8=\frac{1}{96}\frac{t_2^6+96t_6^2}{t_2^2}\} \bigcap \{t_2^3-8t_6=0\}$, \\
$\phi'$ of type $A_5$ &  \begin{tikzpicture}[scale=1,  transform shape]
\node (1) at ( 0,0)  {};
\node (2) at ( 2,0)  {};
\draw  [decoration={markings,mark=at position 1 with
    {\arrow[scale=1.2,>=stealth]{>}}},postaction={decorate}] (1)  --  (2);
\draw  [decoration={markings,mark=at position 1 with
    {\arrow[scale=1.2,>=stealth]{>}}},postaction={decorate}] (2)  --  (1);
\end{tikzpicture} & $(t_2,t_6,t_8,t_{12}) \in \mathcal{H}_1 \bigcap \{t_8=-\frac{1}{192}t_2^4+\frac{1}{4}t_2t_6\}$, \\
$\phi'$ of type $A_1+A_1+A_1+A_1$ &  \begin{tikzpicture}[scale=1,  transform shape]
\node (1) at ( 0,0)  {};
\node (2) at ( 2,0)  {};
\draw  [decoration={markings,mark=at position 1 with
    {\arrow[scale=1.2,>=stealth]{>}}},postaction={decorate}] (1)  --  (2);
\draw  [decoration={markings,mark=at position 1 with
    {\arrow[scale=1.2,>=stealth]{>}}},postaction={decorate}] (2)  --  (1);
\end{tikzpicture} & $(t_2,t_6,t_8,t_{12}) \in \mathcal{H}_2 \backslash \mathcal{H}_1$, \\
$\phi'$ of type $A_2+A_1+A_1+A_1$ &  \begin{tikzpicture}[scale=1,  transform shape]
\node (1) at ( 0,0)  {};
\node (2) at ( 2,0)  {};
\draw  [decoration={markings,mark=at position 1 with
    {\arrow[scale=1.2,>=stealth]{>}}},postaction={decorate}] (1)  --  (2);
\draw  [decoration={markings,mark=at position 1 with
    {\arrow[scale=1.2,>=stealth]{>}}},postaction={decorate}] (2)  --  (1);
\end{tikzpicture} & $(t_2,t_6,t_8,t_{12}) \in \mathcal{H}_2 \bigcap \{t_8 = -\frac{1}{192}t_2^4-\frac{1}{4}t_2t_6\}$, \\
$\phi'$ of type $A_3+A_1+A_1$ &  \begin{tikzpicture}[scale=1,  transform shape]
\node (1) at ( 0,0)  {};
\node (2) at ( 2,0)  {};
\draw  [decoration={markings,mark=at position 1 with
    {\arrow[scale=1.2,>=stealth]{>}}},postaction={decorate}] (1)  --  (2);
\draw  [decoration={markings,mark=at position 1 with
    {\arrow[scale=1.2,>=stealth]{>}}},postaction={decorate}] (2)  --  (1);
\end{tikzpicture} & $(t_2,t_6,t_8,t_{12}) \in \mathcal{H}_1 \bigcap\mathcal{H}_2$, \\
$\phi'$ of type $A_3+A_2+A_1$ &  \begin{tikzpicture}[scale=1,  transform shape]
\node (1) at ( 0,0)  {};
\node (2) at ( 2,0)  {};
\draw  [decoration={markings,mark=at position 1 with
    {\arrow[scale=1.2,>=stealth]{>}}},postaction={decorate}] (1)  --  (2);
\draw  [decoration={markings,mark=at position 1 with
    {\arrow[scale=1.2,>=stealth]{>}}},postaction={decorate}] (2)  --  (1);
\end{tikzpicture} & $(t_2,t_6,t_8,t_{12}) \in \mathcal{H}_1 \bigcap\mathcal{H}_2 \bigcap \{t_8 = -\frac{1}{192}t_2^4-\frac{1}{4}t_2t_6\}$, \\
$\phi'$ of type $A_5+A_1$ &  \begin{tikzpicture}[scale=1,  transform shape]
\node (1) at ( 0,0)  {};
\node (2) at ( 2,0)  {};
\draw  [decoration={markings,mark=at position 1 with
    {\arrow[scale=1.2,>=stealth]{>}}},postaction={decorate}] (1)  --  (2);
\draw  [decoration={markings,mark=at position 1 with
    {\arrow[scale=1.2,>=stealth]{>}}},postaction={decorate}] (2)  --  (1);
\end{tikzpicture} & $(t_2,t_6,t_8,t_{12}) \in \mathcal{H}_1 \bigcap\mathcal{H}_2 \bigcap \{t_8 = -\frac{1}{192}t_2^4+\frac{1}{4}t_2t_6\}$, \\
$\phi'$ of type $E_7$ &  \begin{tikzpicture}[scale=1,  transform shape]
\node (1) at ( 0,0)  {};
\node (2) at ( 2,0)  {};
\draw  [decoration={markings,mark=at position 1 with
    {\arrow[scale=1.2,>=stealth]{>}}},postaction={decorate}] (1)  --  (2);
\draw  [decoration={markings,mark=at position 1 with
    {\arrow[scale=1.2,>=stealth]{>}}},postaction={decorate}] (2)  --  (1);
\end{tikzpicture} & $(t_2,t_6,t_8,t_{12})=(0,0,0,0)$. \\

\end{longtable}
\end{center}
With Table~\ref{TypesofsingularE6}, we see that the singular configuration in the fiber of $\overline{\alpha^\Omega}$ above $f^{-1}(\pi'(h'))$ is of the same type as $\phi'$. Hence for $E_6-F_4-E_7$, the map $f$ realizes a bijection between the singular configurations of the fibers of $\overline{\alpha^\Omega}$ and the sub-root systems of $E_7$ containing $\Theta=\{\alpha_2,\alpha_5,\alpha_7\}$.
\medskip

\section*{Conclusion}
\addcontentsline{toc}{section}{Conclusion}
Conjecture~\ref{Conjecture} and the subsequent proof (Theorem~\ref{Maintheorem}) for small rank cases give a direct link between the singular configurations of the fibers of the deformation $\overline{\alpha^\Omega}$ of a simple singularity of type $\Delta(\Gamma')$ and sub-root systems of a root system of type $\Delta(\Gamma')$. As every fiber of $\overline{\alpha^\Omega}$ is singular (Theorem~\ref{ThmSingular}) and as the Dynkin diagram associated to the singular configuration of such a fiber has to contain the subset $\Theta$ of simple roots of a root system of type $\Delta(\Gamma')$, this subset $\Theta$ characterizes how far the morphism $\overline{\alpha^\Omega}$ is from being isomorphic to the semiuniversal deformation $\chi'_{| S_{e'}}$ of the simple singularity of type $\Delta(\Gamma')$. \\
The full proof of Conjecture~\ref{Conjecture} requires another approach. Indeed the proof presented in Section~\ref{Proof} is based on direct computations of the singular configurations of the fibers of $\alpha^\Omega$ and of $\overline{\alpha^\Omega}$. But these computations themselves stem from the equation of the deformation $\alpha^\Omega$, which is not known when $\alpha^\Omega$ is the semiuniversal deformation of a simple singularity of type $C_{r}$, $r \geq 6$ (i.e. $\Gamma=\mathcal{D}_{r-1}$ and $\Gamma'=\mathcal{D}_{2(r-1)}$). This is due to the fact that in order to obtain $\alpha^\Omega$, we need to determine the independent cycles of $M(\Gamma)$ (cf. Section~\ref{SlodowyCassens}), and as the dimension vector of $M(\Gamma)$ is $(1,1,2,....,2,1,1)$, the $2$-dimensional vertices introduce a complexity in the relations between the cycles, complicating the computation of $\alpha^\Omega$ when the rank $r$ increases (cf. \cite{Cara17} Section 4.3.7 for details). The author is therefore working on a more general approach which would provide a proof of Conjecture~\ref{Conjecture} in all generality.

\section*{Acknowledgements}
\addcontentsline{toc}{section}{Acknowledgements}
The author would like to thank Michaël Bulois for his helpful comments and remarks that helped improve the form of the present article, as well as Kenji Iohara for the constancy of his support and guidance.

\phantomsection
\addcontentsline{toc}{section}{References}
\bibliographystyle{acm}
\bibliography{bibliography}

\end{document}